\documentclass[10pt]{amsart}
\usepackage{amsmath}
\usepackage{mathrsfs}
\usepackage{dsfont}
\usepackage{stmaryrd}
\usepackage{euscript}
\usepackage{expdlist}
\usepackage{enumerate}

\input xy
\xyoption{all}
\usepackage[OT2,T1]{fontenc}

\theoremstyle{plain}
\newtheorem{thm}{Theorem}[section]

\newtheorem*{thma}{Theorem A}

\newtheorem*{thm*}{Theorem}

\newtheorem{lm}[thm]{Lemma}
\newtheorem{cor}[thm]{Corollary}
\newtheorem*{corb}{Corollary B}
\newtheorem{prop}[thm]{Proposition}
\newtheorem*{conj*}{Conjecture}

\theoremstyle{remark}

\theoremstyle{definition}
\newtheorem*{defn*}{Definition}
\newtheorem{Remark}[thm]{Remark}
\newtheorem{I_Remark*}{Remark}
\newtheorem{defn}[thm]{Definition}

\newcommand{\nc}{\newcommand}

\newcommand{\beq}{\begin{equation}}
\newcommand{\eeq}{\end{equation}}
\newcommand{\bpmx}{\begin{pmatrix}}
\newcommand{\epmx}{\end{pmatrix}}
\newcommand{\bbmx}{\begin{bmatrix}}
\newcommand{\ebmx}{\end{bmatrix}}
\newcommand{\wh}{\widehat}
\newcommand{\wtd}{\widetilde}

\newcommand{\beqcd}[1]{\begin{equation*}\label{#1}\tag{#1}}
\newcommand{\eeqcd}{\end{equation*}}

\numberwithin{equation}{section}
\newenvironment{mylist}{
  \begin{enumerate}{}{%
      \setlength{\itemsep}{5pt} \setlength{\parsep}{0in}
      \setlength{\parskip}{0in} \setlength{\topsep}{0in}
      \setlength{\partopsep}{0in}
      \setlength{\leftmargin}{0.17in}}}{\end{enumerate}}

\def\parref#1{\ref{#1}}
\def\thmref#1{Theorem~\parref{#1}}

\def\propref#1{Proposition~\parref{#1}}
\def\corref#1{Corollary~\parref{#1}}     \def\remref#1{Remark~\parref{#1}}
\def\secref#1{\S\parref{#1}}

\def\lmref#1{Lemma~\parref{#1}}
\def\subsecref#1{\S\parref{#1}}

\def\makeop#1{\expandafter\def\csname#1\endcsname
  {\mathop{\rm #1}\nolimits}\ignorespaces}
\makeop{Hom}   \makeop{End}   \makeop{Aut}   %\makeop{Isom}
\makeop{Pic} \makeop{Gal}       \makeop{Div} \makeop{Lie}
\makeop{PGL}   \makeop{Corr} \makeop{PSL} \makeop{sgn} \makeop{Spf}
 \makeop{Tr} \makeop{Nr} \makeop{Fr} \makeop{disc}
\makeop{Proj} \makeop{supp} \makeop{ker}   \makeop{Im} \makeop{dom}
\makeop{coker} \makeop{Stab} \makeop{SO} \makeop{SL} \makeop{SL}
\makeop{Cl}    \makeop{cond} \makeop{Br} \makeop{inv} \makeop{rank}
\makeop{id}    \makeop{Fil} \makeop{Frac}  \makeop{GL} \makeop{SU}
\makeop{Trd}   \makeop{Sp} \makeop{Tr}    \makeop{Trd} \makeop{Res}
\makeop{ind} \makeop{depth} \makeop{Tr} \makeop{st} \makeop{Ad}
\makeop{Int} \makeop{tr}    \makeop{Sym} \makeop{can} \makeop{SO}
\makeop{torsion} \makeop{GSp} \makeop{Tor}\makeop{Ker} \makeop{rec}
\makeop{Ind} \makeop{Coker}
 \makeop{vol} \makeop{Ext} \makeop{gr} \makeop{ad}
 \makeop{Gr}\makeop{corank} \makeop{Ann}
\makeop{Hol} %Holomorphic
\makeop{Fitt} \makeop{Mp} \makeop{CAP}

\DeclareMathOperator{\GO}{GO}
\DeclareMathOperator{\GSO}{GSO}

\DeclareMathAlphabet{\mathpzc}{OT1}{pzc}{m}{it}
\DeclareSymbolFont{cyrletters}{OT2}{wncyr}{m}{n}
\DeclareMathSymbol{\SHA}{\mathalpha}{cyrletters}{"58}

\def\makebb#1{\expandafter\def
  \csname bb#1\endcsname{{\mathbb{#1}}}\ignorespaces}
\def\makebf#1{\expandafter\def\csname bf#1\endcsname{{\bf
      #1}}\ignorespaces}
\def\makegr#1{\expandafter\def
  \csname gr#1\endcsname{{\mathfrak{#1}}}\ignorespaces}
\def\makescr#1{\expandafter\def
  \csname scr#1\endcsname{{\EuScript{#1}}}\ignorespaces}
\def\makecal#1{\expandafter\def\csname cal#1\endcsname{{\mathcal
      #1}}\ignorespaces}
\def\doLetters#1{#1A #1B #1C #1D #1E #1F #1G #1H #1I #1J #1K #1L #1M
                 #1N #1O #1P #1Q #1R #1S #1T #1U #1V #1W #1X #1Y #1Z}
\def\doletters#1{#1a #1b #1c #1d #1e #1f #1g #1h #1i #1j #1k #1l #1m
                 #1n #1o #1p #1q #1r #1s #1t #1u #1v #1w #1x #1y #1z}
\doLetters\makebb   \doLetters\makecal  \doLetters\makebf
\doLetters\makescr
\doletters\makebf   \doLetters\makegr   \doletters\makegr

\def\semidirect{\rtimes}

\normalsize

\makeop{Ram} \makeop{Rep} \makeop{mass}

\makeop{Bl}
\def\abs#1{\left|#1\right|}

\def\Qp{\Q_p}

\def\Zp{\Z_p}

\def\rmN{{\mathrm N}}

\def\cA{{\mathcal A}}  %automorphic forms
\def\cB{\mathcal B}

\def\cG{{\mathcal G}}
\def\cL{{\mathcal L}}

\def\cI{\mathcal I}
\def\cJ{\mathcal J}
\def\cM{\mathcal M}
\def\cR{{\mathcal R}}
\def\cO{\mathcal O}
\def\cS{{\mathcal S}}
\def\cf{{\mathcal f}}
\def\cW{{\mathcal W}}

\def\cZ{\mathcal Z}

\def\cV{{\mathcal V}}
\def\cP{{\mathcal P}}

\def\cJ{\mathcal J}

\def\cU{\mathcal U}

\def\bfc{\mathbf c}

\def\bff{\mathbf f}

\def\bfi{\mathbf i}

\def\sF{\mathscr F}

\def\sV{\mathscr V}

\newcommand{\Z}{\mathbf Z}
\newcommand{\Q}{\mathbf Q}
\newcommand{\R}{\mathbf R}
\newcommand{\C}{\mathbf C}
\newcommand{\A}{\mathbf A}    % for adele

\def\bbH{\mathbb H}

\def\frakp{{\mathfrak p}}

\def\frakH{{\mathfrak H}}

\def\frakA{\mathfrak A}

\def\frakS{\mathfrak S}

\def\frakN{\mathfrak N}

\def\bfone{{\mathbf 1}}

\def\etale{{\'{e}tale }}
\def\padic{\text{$p$-adic }}

\def\BS{Bruhat-Schwartz }

\newcommand{\<}{\langle}   %\< is not defined yet.
\renewcommand{\>}{\rangle} %\> is already defined.

\def\ot{\otimes}

\def\hookto{\hookrightarrow}

\def\ol{\overline}  \nc{\opp}{\mathrm{opp}} \nc{\ul}{\underline}

\newcommand{\pair}[2]{\< #1, #2\>}
\newcommand{\bbpair}[2]{\<\!\<#1,#2\>\!\>}
\newcommand{\altpair}[2]{(#1,#2)}
\newcommand{\rpair}[2]{(#1,#2)}

\newcommand{\pairing}{\pair{\,}{\,}}
\def\bbpairing{\bbpair{\,}{\,}}

\def\rpairing{\altpair{\,}{\,}}

\def\XYmatrix{\xymatrix@M=8pt} % make \xymatrix not too cluttered
\def\ncmd{\newcommand}
\ncmd{\xysubset}[1][r]{\ar@<-2.5pt>@{^(-}[#1]\ar@<2.5pt>@{_(-}[#1]}
\ncmd{\XYmatrixc}[1]{\vcenter{\XYmatrix{#1}}}
\ncmd{\xyto}[1][r]{\ar@{->}[#1]}
\ncmd{\xyinj}[1][r]{\ar@{^(->}[#1]}
\ncmd{\xysurj}[1][r]{\ar@{->>}[#1]}
\ncmd{\xyline}[1][r]{\ar@{-}[#1]}
\ncmd{\xydotsto}[1][r]{\ar@{.>}[#1]}
\ncmd{\xydots}[1][r]{\ar@{.}[#1]}
\ncmd{\xyleadsto}[1][r]{\ar@{~>}[#1]}
\ncmd{\xyeq}[1][r]{\ar@{=}[#1]} \ncmd{\xyequal}[1][r]{\ar@{=}[#1]}
\ncmd{\xyequals}[1][r]{\ar@{=}[#1]}
\ncmd{\xymapsto}[1][r]{l\ar@{|->}[#1]}\ncmd{\xyimplies}[1][r]{\ar@{=>}[#1]}
\ncmd{\xyiso}{\ar[r]_-{\sim}}
\def\injxy{\ar@{^(->}}

\newcommand{\pMX}[4]{\begin{pmatrix}
{#1}& {#2}\\
{#3}&{#4}\end{pmatrix} }

 \newcommand{\pDII}[2]{\begin{pmatrix}{#1}&0
 \\0&{#2}\end{pmatrix}}

\newcommand{\seesaw}[4]{{#1}\ar@{-}[rd]\ar@{-}[d]&{#2}\ar@{-}[d]\\
{#3}\ar@{-}[ru]&{#4}}

\def\cf{\mbox{{\it cf.} }}

\def\uf{\varpi} %uniformizer
\def\Abs{{|\!\cdot\!|}} %adelic absolute value
\def\ndivides{\nmid}

\def\x{{\times}}
\def\xx{\,\times\,}

\def\onehalf{{\frac{1}{2}}}
\def\e{\varepsilon} % episilon factor
\def\al{\alpha}

\def\om{\omega}

\def\iso{\simeq}
\def\con{\equiv}
\def\bksl{\backslash}
\newcommand\stt[1]{\left\{#1\right\}}
\def\ep{\epsilon}

\def\lam{\lambda}
\def\sg{\sigma}

\def\disjoint{\sqcup}
\def\bigot{\bigotimes}

\def\divides{\mid}

\renewcommand\pmod[1]{\,(\mbox{mod }{#1})}

\renewcommand\Re{\text{Re}\,}
\usepackage{amsmath,amsthm,graphicx,amscd,multirow}
\usepackage{amssymb} % \mathfrak
\usepackage{mathrsfs} % \mathsfs
\usepackage{hyperref}

\DeclareMathAlphabet{\mathpzc}{OT1}{pzc}{m}{it} %\mathpzc
\allowdisplaybreaks[1]

\theoremstyle{definition}
\numberwithin{equation}{section}
\renewcommand{\bar}{\overline}

\def\rmd{{\rm d}}
\def\ulk{{\ul{k}}}
\def\test{\varphi^\star}
\def\q{p}
\def\Re{{\mathrm Re}\,}
\title[Inner product formula for Yoshida lifts]{Inner product formula for Yoshida lifts}
\author[M.L. Hsieh]{Ming-Lun Hsieh}
\address{Institute of Mathematics, Academia Sinica~\\ Taipei 10617, Taiwan\and National Center for Theoretic Sciences}
\email{mlhsieh@math.sinica.edu.tw}

\author[K. Namikawa]{Kenichi Namikawa}
\address{ Department of Mathematics, School of Engineering, Tokyo Denki University~\\
5, Asahicho, Senju, Adachi city, Tokyo, 120-8551, Japan~
}
\email{namikawa@mail.dendai.ac.jp}
\date{\today}
\begin{document}
\begin{abstract}
We prove an explicit inner product formula for vector-valued Yoshida lifts by an explicit calculation of local zeta integrals in the Rallis inner product formula for ${\rm O}(4)$ and $\Sp(4)$. As a consequence, we obtain the non-vanishing of Yoshida lifts. 
\end{abstract}
\maketitle
\tableofcontents
\def\Asai{{\rm Asai}}

\section{Introduction}\label{intro}
Explicit formulas for Petersson norms of modular forms play an important role in the study of the connection between congruences among modular forms and special values of $L$-functions. The aim of this paper is to give an explicit formula for the Petersson norm of Yoshida lifts. Let $F$ be either $\Q\oplus\Q$ or a real quadratic field of $\Q$ and let $\frakN$ be an ideal of the ring of integers $\cO_F$ of $F$. Let $\ulk=(k_1,k_2)$ be a pair of non-negative integers with $k_1\geq k_2$. Yoshida lifts are explicit vector-valued Siegel modular forms of genus two and weight $\Sym^{2k_2}(\C^2)\ot\det^{k_1-k_2+2}$ associated with a holomorphic newform $f$ on $\PGL_2(\A_F)$ of conductor $\frakN$ and weight $(2k_1+2,2k_2+2)$. Note that $f$ is given by a pair $(f_1,f_2)$ of elliptic newforms if $F=\Q\oplus \Q$, and $f$ is a Hilbert modular newform over a real quadratic field if $F$ is a real quadratic field. The scalar-valued Yoshida lifts ($k_2=0$) were constructed by H. Yoshida in \cite{yo80} via theta lifting from ${\rm SO}(4)$ to $\Sp(4)$, and his construction was extended to the vector-valued Yoshida lifts ($k_2>0$) by  B\"{o}cherer and Schulze-Pillot (\cf \cite[\S 1]{bsp97} and \cite[\S 3]{hn15}). In the sequel, Yoshida lifts are said to be of type (I) if $F=\Q\oplus \Q$ and of type (II) if $F$ is a real quadratic field. The non-vanishing of Yoshida lifts was also conjectured by Yoshida himself, which was later proved in \cite{bsp97} for Yoshida lifts of type (I). Then our main result is an explicit Rallis inner product formula for Yoshida lifts, which relates the Petersson norm of the Yoshida lift to special values of the Asai $L$-function $L({\rm As}^+(f),s)$ attached to $f$. As a consequence of our formula, we prove the non-vanishing of Yoshida lifts of type (I) and (II). 
 
To state our main result precisely, we introduce some notation. Let $c$ be the non-trivial automorphism of $F$ and let $\Delta_F$ be the discriminant of $F$. Denote by $f^c$ the Galois conjugation of $f$. 
We assume that $f$ is not Galois self-dual, namely
\[f\not =f^c,\]
and that the conductor $\frakN$ of $f$ is a square-free product of prime ideals of $\cO_F$ with $(\frakN,\Delta_F)=1$ and is divisible by $N^-$ a square-free product of an odd number of rational primes split in $F$. Let $D_0$ be the definite quaternion algebra over $\Q$ of absolute discriminant $N^-$ and let $D=D_0\ot_\Q F$. By our assumption on $N^-$, $D$ is the totally definite quaternion algebra over $F$ ramified at $N^-\cO_F$. Let $R$ denote the Eichler order in $D$ of level $\frakN^+$. For $i=1,2$, let $\cW_{k_i}(\C):=\Sym^{2k_i}(\C^2)\ot\det^{-k_i}$ be the algebraic representation of $\GL_2(\C)$ of the highest weight $(k_i,-k_i)$. By the Jacquet-Langlands-Shimizu correspondence, there exist a vector-valued newform $\bff:D^\x\bksl D_\A^\x/\wh R^\x\to \cW_{k_1}(\C)\boxtimes\cW_{k_2}(\C)$ unique up to scalar such that $\bff$ shares with same Hecke eigenvalues with $f$ at $p\nmid N^-$ (\cf\cite[\S 3.3]{hn15}).  Let $*$ be the main involution of $D$ and let $V=\stt{x\in D\mid x^*=x^c}$ be the four dimensional $\Q$-vector space with the positive definite quadratic form ${\rm n}(x)=xx^*$. Following \cite[p.196]{yo80}, the group $G':=\stt{x\in D^\x\mid {\rm n}(x)=1}$ acts on $V$ via the action $\varrho(a)x=ax(a^c)^*$ and the image $\varrho(G')\subset \Aut(V)$ is the special orthogonal group $\SO(V)$. We can thus view $\bff$ as an automorphic form on ${\rm SO}(V)(\A)$ and consider its theta lifts to $\Sp(4)$. Let $N=\frakN\cap\Z$ and let $N_F=N\Delta_F$. Let $\frakH_2$ be the Siegel upper half plane of degree two and $\Gamma_0^{(2)}(N_F)\subset \Sp_4(\Z)$ be the Siegel parabolic subgroup of level $N_F$. Let $\cL(\C):=\Sym^{2k_2}(\C^2)\det^{k_1-k_2+2}$ be the representation of $\GL_2(\C)$ of the highest weight $(k_1+k_2+2,k_1-k_2+2)$. In \cite[\S 3.7]{hn15}, we apply the theta lifting from ${\rm SO}(V)$ to $\Sp(4)$ to obtain the vector-valued Yoshida lift $\theta_\bff^*:\frakH_2\to \cL(\C)$ attached to $\bff$ and a distinguished \BS funciton $\test$ on $V_\A^{\oplus2}$ with value in $\cW_{k_1}(\C)\ot\cW_{k_2}(\C)\ot\cL(\C)$ (see \secref{secYos} for more details). In particular, $\theta^*_\bff$ is exactly the scalar-valued Siegel modular form constructed in \cite{yo80} when $k_2=0$. The Yoshida lift $\theta^*_\bff$ is a vector-valued Siegel modular form of level $\Gamma^{(2)}_0(N_F)$, which is an eigenfunction of Hecke operators at $p\ndivides N_F$, and the associated spin $L$-function $L(\theta_\bff^*,s)$ is given by $L(f_1,s-k_2)L(f_2,s-k_1)$ if $F=\Q\oplus\Q$ and $f=(f_1,f_2)$ and by $L(f,s-k_2)$ if $F$ is a real quadratic field. Let $\cB_\cL:\cL(\C)\ot\cL(\C)\to\C$ be the positive definite Hermitian pairing defined in \eqref{E:pairBL} and define the Petersson norm of $\theta^*_\bff$ by \[
     \pair{\theta^*_{\bff}}{\theta^*_{\bff}}_{\frakH_2} 
  =  \int_{\Gamma_0^{(2)}(N_F) \backslash \frakH_2} \cB_\cL(\theta^*_{\bff}(Z),\theta^*_{\bff}(Z)) (\det Y)^{k_1+2} \frac{\rmd X\rmd Y}{(\det Y)^{3} }.
\]
Let $\pair{\bff}{\bff}_R$ be the Peterson norm of $\bff$ defined in \eqref{E:fnorm} of \secref{S:PetYL}.
Now we state our main result in the simple case $\frakN=N\cO_F$. For $p\divides N$, denote by $\e_p\in\stt{\pm 1}$ the Atkin-Lehner eigenvalues of $\bff$ at $p$ (see \eqref{E:Atkineigenvalue} for the definition). 
  
\begin{thma}[\thmref{T:ClaInnPrd}]\label{main} Suppose further that $\frakN=N\cO_F$. Then we have 
\begin{align*}
    \frac{\pair{\theta^*_{\bff}}{\theta^*_{\bff}}_{\frakH_2}}{\pair{\bff}{\bff}_R}  
=&L({\rm As}^+(f),k_1+k_2+2)\cdot (4\pi)^{-(2k_1+3)}\Gamma(k_1+k_2+2)\Gamma(k_1-k_2+1)\\
&\times \frac{N\cdot 2^{4r_{F,2}-r_F-2}}{(2k_1+1)(2k_2+1)} \cdot \prod_{p\divides N}(1+\e_p)\cdot \prod_{  p\mid \Delta_F  } (1+p^{-1}),\end{align*}
where $r_F$ is the number of prime factors of $\Delta_F$, and $r_{F,2}=1$ if $2\divides \Delta_F$ and $0$ otherwise.
\end{thma}If $\frakN\not =N\cO_F$, then we reply $\bff$ with the stablilized newform $\bff^\dagger$ defined in \eqref{E:levelraisingf}, and \thmref{T:ClaInnPrd} provides the formula for the Petersson norm of $\theta^*_{\bff^\dagger}$. Note that the left-hand side of the formula is independent of the choice of the newform $\bff$ since a newform is unique up to a scalar. In fact, $\bff$ can be normalized so that $\bff|_{\wh D^\x}$ takes values in the Hecke field $\Q(f)$ of $f$, namely the field generated by the Fourier coefficients of $f$ over $\Q$, and hence $\theta_\bff^*$ is defined over $\Q(f)$ in view of the formula for Fourier coefficients of $\theta_\bff^*$ in the proof of \cite[Proposition 5.1]{hn15}. This allows us to the study algebraicity of the special value $L({\rm As}^+(f),k_1+k_2+2)$ by the method in \cite{saha15}.

\begin{Remark} We give some additional comments on the case where $F=\Q\oplus\Q$ and $f=(f_1,f_2)$ is given by a pair of elliptic newforms. \begin{enumerate}
\item[(i)] The Asai $L$-function $L({\rm As}^+(f),s)=L(f_1\ot f_2,s)$ is the Rankin-Selberg convolution of $f_1$ and $f_2$. In this case, an inner product formula for Yoshida lifts of type (I) was also derived in \cite[Corollary 8.8]{bdsp12} by the Rankin-Selberg method. In \cite[Conjecture 5.19]{ak13}, Agarwal and Klosin formulated a conjecture on an explicit inner product formula of scalar-valued Yoshida lifts of type (I). Our Theorem A confirms their conjecture and further generalizes to the vector-valued Yoshida lifts.
\item[(ii)]  Given a prime $p>k_1$, under some mild assumptions on the residual \padic Galois representations attached to $f_1$ and $f_2$, it is known that $\bff$ can be normalized such that $\theta_\bff^*$ has Fourier coefficients in the ring of integers of the Hecke field of $f_1$ and $f_2$ localized at $p$ and is non-vanishing modulo $p$ (See \cite[\S 5]{hn15}) and that the Petersson norm $\pair{\bff}{\bff}$ is given by a product of the congruence numbers of $f_1$ and $f_2$ up to a \padic unit (See \cite{pw11} and \cite{ch16}). In particular, the period ratio $\Omega_{1,2}$ in \cite[Remark 6.4]{ak13} is a $p$-unit in many situations.
\end{enumerate}
\end{Remark}
Let $\pi_f$ be the unitary cuspidal automorphic representation of $\GL_2(\A_F)$ associated with $f$. Then in terms of automorphic $L$-functions, we have \[L(s,{\rm As}^+(\pi_f))=\Gamma_\C(s+k_1+k_2+1)\Gamma_\C(s+k_1-k_2)L({\rm As}^+(f), s+k_1+k_2+1).\]
By the non-vanishing of $L(1,{\rm As}^+(\pi))$ (\cf\cite[Theorem 4.3]{shahidi15}), we obtain the following consequence on the non-vanishing of Yoshida lifts, generalizing the main result in \cite{bsp97} to Yoshida lifts of type (II).
\begin{corb}If $\e_p=1$ for every $p\divides N$, then the vector-valued Yoshida lift $\theta^*_\bff$ is non-zero.
\end{corb}

\begin{Remark} \begin{enumerate}\item[(i)] The necessary condition on the Atkin-Lehner eigenvalues for the non-vanishing of
Yoshida lifts already appeared in \cite{yo80}.
\item[(ii)] Our results are different from those obtained by the representation theoretic method in \cite[Theorem 8.3]{ro01}, \cite[Theorem 1.1]{takeda09} and \cite[Proposition 3.1]{saha13} for these authors prove the non-vanishing of
the space generated by some Yoshida lift, while we prove the non-vanishing of a particular Yoshida
lift with integral Fourier coefficients.
\item[(iii)] Note that a paramodular Yoshida lift of type (II) was constructed in \cite{lbrooks12}, but our Yoshida
lift is Siegel parahoric. The local component of the automorphic representation
generated by $\theta_\bff^*$ at a prime $p\divides \Delta_F$ provides an example of generic non-endoscopic supercuspidal
representations of $\GSp4(\Qp)$ possessed of a Siegel parahoic fixed vector.
\end{enumerate}
\end{Remark}
In addition to the application to the non-vanishing of Yoshida lifts, our main motivation for the explicit Petersson norm formula for Yoshida lifts of type (I) originates from the study on the congruences between Hecke eigen-systems of Yoshida lifts and stable forms on ${\rm GSp}(4)$, the so-called \emph{Yoshida congruence} as well as its application to the Bloch-Kato conjecture for special values of Asai $L$-functions. The Yoshida congruence was first investigated by the independent works \cite{bdsp12} and \cite{ak13}, where the Petersson norm formula was used to relate the congruence primes of Yoshida lifts of type (I) to special values of the Rankin-Selberg $L$-functions. More precisely, in \cite[Corollary 9.2]{bdsp12} and \cite[Theorem 6.6]{ak13}, the authors proved that if a prime $p$ divides the algebraic part of the $L$-values $L(f_1\ot f_2,k_1+k_2+2)$, then $p$ is a congruence prime for Yoshida lifts attached to a pair of elliptic newforms $(f_1,f_2)$ of weight $(2k_1+2,2k_2+2)$ under some restricted hypotheses. It is our hope that this Petersson norm formula together with our previous result on the non-vanishing of the Yoshida lift $\theta_\bff^*$ modulo a prime in \cite{hn15} serve the first step towards the understanding of Yoshida congruence in a more general setting.

This paper is organized as follows.
In \secref{secNot}, we fix the notation and definitions, and in \secref{secAsL}, we introduce the Asai $L$-functions. In \secref{secYos}, we recall the construction of Yoshida lifts in \cite{hn15}, which depends on the choice of the particular test function $\test=\ot_p\test_p$ given in \subsecref{SS:testfunctions}. In \secref{secInnprd}, we realize Yoshida lifts as theta lifts from ${\rm GO}(4)$ to $\Sp(4)$ and then apply the Rallis inner product formula in \cite{gi11} and \cite{gqt14} to reduce the Petersson norm formula to the explicit computation of certain local zeta integrals $\cI(\test_\infty)$ at the archimedean place and $\cZ(\test_p,f_p^\dagger)$ at non-archimedean places (see \propref{RallisToInt}). In \secref{LocComp}, we carry out the bulk of this paper, the explicit calculation of these local zeta integrals at all places.

\section{Notation and definitions}\label{secNot}
\subsection{Basic notation}

For a number field $F$, we denote by $\mathcal O_F$ (resp. $\Delta_F, {\mathfrak D}_{F/\Q}$ ) the ring of integers of $F$ (resp.  the discriminant of $F/\Q$, the different of $F/\Q$).
Let $\A_F$ be the ring of adeles of $F$. For an element $a$ of $\A_F$ and a place $v$ of $F$, we denote by $a_v$ the $v$-component of $a$. 

Let $\Sigma_\Q$ be the set of places of the rational number field $\Q$. We write $\A$ for $\A_\Q$. Let $\psi=\prod_{v\in\Sigma_\Q}\psi_v:\A/\Q\to \C^\times$ be the additive character 
with $\psi(x_\infty)=\exp(2\pi\sqrt{-1}x_\infty)$ for $x_\infty\in \R=\Q_\infty$.

Let $\wh\Z$ be the profinite completion of $\Z$. If $M$ is an abelian group, let $\wh M=M\ot_\Z\wh \Z$. For a place $v\in\Sigma_\Q$, $M_v=M\ot_\Z\Z_v$.

For an algebraic group $G$ over $\Q$, let $Z_G$ be the center of $G$ and 
%Denote by $G_\Q=G(\Q)$ (resp. $G_v=G(\Q_v)$) the group of $\Q$-rational points (resp. $\Q_v$-points) of $G$, and 
%denote by $G_{\mathbf A}=G(\A)$ and $G_{{\mathbf A}_f}=G(\A_f)$ the adelization of $G$ and its finite part.
let $[G]$ be the quotient space $G(\Q)\backslash G(\A)$.

For a set $S$, denote by ${\mathbb I}_S$ the characteristic function of $S$ and by $\# S$ the cardinality of $S$.

\subsection{Algebraic representations of $\GL(2)$}\label{algrep}
Let $A$ be an $\Z$-algebra.
Let $A[X,Y]_n$ denote the space if two variable homogeneous polynomial of degree $n$ over $A$.
Suppose $n!$ is invertible in $A$.
We define the perfect pairing $\langle\cdot ,\cdot \rangle_n:A[X,Y]_n\times A[X,Y]_n\to A$ by
\begin{align*}
  \langle X^iY^{n-i}, X^jY^{n-j} \rangle_n
 =\begin{cases}  (-1)^i\binom{n}{i}^{-1},  & {\rm if }\ i+j=n, \\
                       0, & {\rm if }\ i+j\neq n, \end{cases} 
\end{align*}
where $\binom{a}{b}$ is the binomial coefficient defined by
\begin{align*}
   \binom{a}{b} = \frac{\Gamma(a+1)}{\Gamma(a-b+1)\Gamma(b+1)} \quad(a,b\in\Z).
\end{align*}
For each polynomial $P\in A[X,Y]_n$ and each $g\in {\rm GL}_2(A)$, define the polynomial $g\cdot P$ to be
\begin{align*}
 (g\cdot P)(X,Y) = P((X,Y)g).   
\end{align*}
Then, it is well-known that the pairing $\langle\cdot,\cdot\rangle_n$ on $A[X,Y]_n$ satisfies
\begin{align*}
     \langle g\cdot P, g\cdot Q \rangle_n
  = (\det g)^{n} \cdot \langle P,Q \rangle_n \quad (P,Q\in A[X,Y]_n, g\in {\rm GL}_2(A)).
\end{align*}

For $\kappa = (n+b,b)\in \Z^2$ with $n\in \Z_{\geq 0}$,
let ${\mathcal L}_\kappa(A):=A[X,Y]_n$ 
and let $\rho_\kappa:{\rm GL}_2(A) \to {\rm Aut}_A{\mathcal L}_\kappa(A)$ be the representation given by
\begin{align*}
  \rho_\kappa(g) P(X,Y) = P((X,Y)g) \cdot (\det g)^b.  
\end{align*}
Then $(\rho_\kappa,{\mathcal L}_\kappa(A))$ is the algebraic representation  of ${\rm GL}_2(A)$
with the highest weight $\kappa$.
For each non-negative integer $k$, we put
\begin{align*}
  (\tau_k,{\mathcal W}_k(A)) := ( \rho_{(k,-k)},A[X,Y]_{2k}). 
\end{align*}
Then $({\mathcal W}_k(A), \tau_k)$ is the algebraic representation of ${\rm PGL}_2(A) = {\rm GL}_2(A)/A^\times$.
%and the pairing $\langle\cdot ,\cdot \rangle_{2k}$ is ${\rm GL}_2(A)$-equivariant.

\subsection{Representations of $\GL(2)$ over local fields}
If $F$ is a local field, let $\Abs$ be the standard absolute value of $F$. Denote by $\pi(\mu,\nu)$ the principal series representation of $\GL_2(F)$ with characters $\mu,\nu:F^\times \to \C^\times$ such that $\mu\nu^{-1}\not =\Abs^\pm$ and by ${\rm St}\ot(\chi\circ\det)$ the special representation of ${\rm GL}_2(F)$ attached to a character $\chi:F^\x\to\C^\x$. 

The $L$-functions in this paper are always referred to the \emph{complete} $L$-function. In particular, the Riemann zeta function $\zeta(s)$ is given by \[\zeta(s)=\prod_v\zeta_v(s),\]
where $\zeta_\infty(s)=\pi^{-s/2}\Gamma(s/2)$ and $\zeta_p(s)=(1-p^{-s})^{-1}$. 
For later use, we recall the $\Gamma$-factors $\Gamma_\R(s)$ and $\Gamma_\C(s)$
which are defined as follows:
\begin{align*}
\Gamma_\R(s)= \pi^{-s/2}\Gamma(s/2), \quad  \Gamma_\C(s)= 2(2\pi)^{-s}\Gamma(s).
\end{align*}

\subsection{Siegel modular forms of genus two}\label{secSCF}
Let ${\rm GSp}_{4}$ be the algebraic group defined by
\begin{align*}
    {\rm GSp}_4 
 = \left\{  g\in {\rm GL}_4 :  
              g\begin{pmatrix} 0 & \bfone_2 \\ -\bfone_2 & 0 \end{pmatrix} {}^{\rm t} g 
              = \nu(g) \begin{pmatrix} 0 & \bfone_2 \\ -\bfone_2 & 0 \end{pmatrix}  
                 \right\}  
\end{align*}
with the similitude character $\nu: {\rm GSp}_4\to {\mathbb G}_m$.
For a positive integer $N$, define 
\begin{align*}
 \Gamma^{(2)}_0(N) = \left\{  \begin{pmatrix} A & B \\ C & D \end{pmatrix} \in {\rm GSp}_{4}(\widehat{\Z}) : 
                                  C \equiv 0 \text{ mod } N   \right\}.  
\end{align*}
%If $n=1$, we write $\Gamma^{(n)}_0(N)$ for $\Gamma_0(N)$. 
Define the automorphy factor $J:{\rm GSp}_4(\R)^+\times{\mathfrak H}_2 \to {\rm GL}_2(\C)$ by
\begin{align*}
  J(g,Z) = CZ+D \quad (g\in \begin{pmatrix} A & B \\ C & D \end{pmatrix}). 
\end{align*}
Let ${\mathbf i}:=\sqrt{-1}\cdot I_2$. 
Let $\eta$ be a quadratic Hecke character of $\A^\x$.  A holomorphic Siegel cusp form $F: {\rm GSp}_4({\mathbf A}) \to {\mathcal L}_\kappa(\C)$ 
is said to be of weight $\kappa$, level $\Gamma^{(2)}_0(N)$ and type $\eta$ with the trivial central character if for every $g\in\GSp_4(\A)$, we have
          \begin{align*}
           F(z\gamma g u_\infty u_f) &= \rho_\kappa(J(u_\infty, {\mathbf i})^{-1})  \eta({\rm det}(D)) F(g),\\
           &(\gamma\in\GSp_4(\Q),u_\infty\in {\rm U}_2(\R),\,u_f=\pMX{A}{B}{C}{D}\in \Gamma^{(2)}_0(N)).
          \end{align*}    

\section{Asai $L$-functions}\label{secAsL}
\subsection{Local Asai transfer}If $k$ is a local field, denote by $W_k'$ the Weil-Deligne group of $k$ (\cf\cite[(4.1.1)]{tate77Corvallis}) and by $\cA(\GL_2(k))$ the set of isomorphism classes of admissible irreducible representations of $\GL_2(k)$. If $\pi\in\cA(\GL_2(k))$, denote by $\varphi_{\pi}:W_k'\to\GL_2(\C)$ a Langlands parameter of $\pi$ under the local Langlands correspondence. Consider the semi-direct product $\cG:=(\GL_2(\C)\x\GL_2(\C))\semidirect \Z/2\Z$, where $\Z/2\Z$ acts by permuting the two factors of $\GL_2(\C)\x\GL_2(\C)$.
Let $F$ be an quadratic \etale extension of $k$. Let $\pi$ be an irreducible representation of $\GL_2(F)$. Recall that a Langlands parameter $\wtd\varphi_\pi:W_k'\to\cG$ attached to the automorphidc induction of $\pi$ is defined as follows:
If $F=k\oplus k$, then $\pi=\pi_1\oplus \pi_2$ with $\pi_i\in\cA(\GL_2(k))$ and define 
$\wtd\varphi_\pi(\sg)=(\varphi_{\pi_1}(\sg),\varphi_{\pi_2}(\sg),0)$. If $F$ is a field, then $\pi\in\cA(\GL_2(F))$, and fixing a decomposition $W_k'=W_F'\disjoint W_F'c$, define 
$\wtd\varphi_\pi(\sg)=(\varphi_\pi(\sg),\varphi_\pi(c^{-1}\sg c),0)$ if $\sg\in W_F'$ and $\wtd\varphi_\pi(\sg )= (\varphi_\pi(\sg c),\varphi_\pi(c^{-1}\sg),1)$ for $\sg\in W_F' c$. Let $r^\pm:\cG\to \GL(\C^2\ot \C^2)$be the four dimensional representations given by 
\[r^\pm(g_1,g_2,0) (v\ot w)=g_1v\ot g_2w;\quad r^\pm(\bfone_2,\bfone_2,1)(v\ot w)= \pm w\ot v, \]
for each $v, w \in \C^2$. 
Then the local Asai transfer ${\rm As}^\pm(\pi)$ is defined to be the irreducible representation of $\GL_4(k)$ corresponding to the Weil-Deligne representation $r^\pm\circ\wtd\varphi_\pi$ under the local Langlands correspondence (\cite[\S 2]{kris12}). 
\subsection{Asai $L$-functions} Let $F/\Q$ be an \etale quadratic extension. Let $\pi$ be a unitary cuspidal automorphic representation of $\GL_2(\A_F)$ and factorize it into the restricted tensor product $\pi=\ot_v\pi_v$, where $v$ runs over all places of $\Q$ and $\pi_v$ is an irreducible admissible representation of $\GL_2(F_v)$. Define ${\rm As}^\pm(\pi):=\ot_v{\rm As}^\pm(\pi_v)$, which is known to be an isobaric automorphic representation of $\GL_4(\A)$ (\cite[Theorem 6.7]{kris03}). Note that by definition, we have ${\rm As}^-(\pi)={\rm As}^+(\pi)\ot\tau_{F/\Q}$, where $\tau_{F/\Q}$ is the quadratic character corresponding to $F/\Q$. Let $L(s,{\rm As}^\pm(\pi_v))$ be the local $L$-function attached the Weil-Deligne representation $r^\pm\circ\wtd\varphi_{\pi_v}$ (\cite[(4.1.6)]{tate77Corvallis}) and let $L(s,{\rm As}^\pm(\pi))=\prod_v L(s,{\rm As}^\pm(\pi_v))$ be the automorphic $L$-function of ${\rm As}^\pm(\pi)$. For the convenience in the later application, we give the complete list of local $L$-functions $L(s,{\rm As}^+(\pi_v))$ if $v=\infty$ and if $\pi_v$ is either a unramified principal series or a special representation.

\begin{defn}\label{D:AsaiL}
\begin{mylist}
\item If $v=w\bar w$ is split in $F$, and $\pi_v=\pi_w\ot \pi_{\bar w}$, then \[L(s,{\rm As}^+(\pi_v))=L(s,\pi_w\ot\pi_{\bar w})\] is the local tensor product $L$-function for $\pi_w\ot\pi_{\bar w}$ defined in \cite{gj78}.
\item If $v$ is non-split in $F$ and $\pi_v=\pi(\mu,\nu)$ is a unramified principal series with two characters $\mu,\nu:F_v^\x\to\C^\x$, then \[L(s,{\rm As}^+(\pi_v))=L(s,\mu|_{\Q_v^\x})L(s,\mu\nu)L(s,\nu|_{\Q_v^\x}).\]
\item If $v$ is non-split and $\pi_v={\rm St}\ot(\chi\circ\det)$ is the special representation twisted by a character $\chi:F_v^\x\to\C^\x$, then 
\[L(s,{\rm As}^+(\pi_v))=L(s+1,\chi|_{\Q_v^\x})L(s,\tau_{F_v/\Q_v}\chi|_{\Q_v^\x}).\]
\item For the archimedean place $v=\infty$ of $\Q$, we define  
 \[  L(s,{\rm As}^+(\pi_\infty)) = \Gamma_\C(s+k_1+k_2+1) \Gamma_\C(s+k_1-k_2). \]
\end{mylist}
\end{defn}
 In particular, if $F=\Q\oplus\Q$, then $\pi=\pi_1\oplus\pi_2$ is a direct sum of two automorphic representations of $\GL_2(\A)$ and $L(s,{\rm As}^+(\pi))=L(s,\pi_1\ot\pi_2)$. 
 \begin{defn}\label{D:GaloisSelfDual}Let $c$ be the non-trival automorphism of $F/\Q$. We say $\pi$ is \emph{Galois self-dual} if the contragradient representation $\pi^\vee$ is isomorphic to the Galois conjugate $\pi^c$. \end{defn} 
 The following theorem gives the complete description of the analytic properties of Asai $L$-functions when $\pi$ is not Galois self-dual.
\begin{thm}\label{T:AsNonvan}
The Asai $L$-function $L(s,{\rm As}^+(\pi))$ is meromorphically continued to the whole $\C$-plane with possible pole at $s=0$ or $1$. 
Furthermore, if $\pi$ is not Galois self-dual, then $L(s,{\rm As}^+(\pi))$ is entire and $L(s,{\rm As}^+(\pi))$ is non-zero for $\Re s\geq 1$ or $\Re s\leq 0$. 
\end{thm}
\begin{proof}This is a special case of \cite[Theorem 4.3]{shahidi15}.
\end{proof}

% !TEX root = YPNorm.tex

\def\newform{\bff^\circ}
\section{Yoshida lifts}\label{secYos}
In this section, we recall the construction of vector-valued Yoshida lifts in \cite[\S 3]{hn15}.
\subsection{Groups}\label{secYos1}
\def\ep{\varepsilon}
Let $D_0$ be a definite quaternion algebra over $\Q$ of discriminant $N^-$ 
and let $F$ be a quadratic \etale algebra over $\Q$. 
Let $D = D_0\otimes_\Q F$. 
We assume that 
\beqcd{split}\text{every place dividing $\infty N^-$ is split in $F$.}\eeqcd 
It follows that $F$ is either $\Q\oplus \Q$ or a real quadratic field over $\Q$, and $D$ is precisely ramified at $\infty N^-$. 
Denote by $x\mapsto x^\ast$ the main involution of $D_0$ 
    and by $x\mapsto x^c$ the non-trivial automorphism of $F/\Q$, 
    which are extended to automorphisms of $D$ naturally. 
We define the four dimensional quadratic space $(V,{\rm n})$ over $\Q$ by \[
  V = \left\{ x\in D : x^* = x^c   \right\};\quad {\rm n}(x)=xx^*.\]
  Let $H^0$ be the algebraic group over $\Q$ given by 
\[  H^0 (\Q)=(D^\x \times\Q^\x)/F^\x,
\]
where $F^\x$ sits inside $D^\x\times\Q^\x$ as $(z,\rmN_{F/\Q}(z))$. Then $H^0$ acts on $V$ via $\varrho: H^0 \to \Aut V$ given by
\[   \varrho(a,\alpha)(x) = \alpha^{-1} a x (a^c)^*\quad (x\in V, (a,\alpha)\in H^0).  \label{varrho}   
\]
This induces an identification $\varrho: H^0\iso \GSO(V)$ with the similitude map given by 
\[ \nu(\varrho(a,\alpha))=\al^{-2}{\rm N}_{F/\Q}(aa^*). 
\]
For $a\in D_\A^\x$, we write $\varrho(a)=\varrho(a,1)$. 
Put
\[ H_1^0=\stt{h\in H^0\mid \nu(\varrho(h))=1}\iso \SO(V).
\]

\begin{Remark}\label{R:split}
If $v=w_1w_2$ is a place split in $F$, then $F\ot_\Q\Q_v=\Q_v e_{w_1}\oplus \Q_v e_{w_2}$, where $e_{w_1}$ and $e_{w_2}$ are idempotents corresponding to $w_1$ and $w_2$ respectively. Let $D_{0,v}=D\ot_\Q \Q_v$. For a place $w_1$ lying above $v$, in the sequel we make the identifications \beq\label{E:splittrivialization}\begin{aligned}
(D_{0,v}^\x\times D_{0,v}^\x)/\Q_v^\x\iso& H^0(\Q_v),\quad (a,d)\mapsto (ae_{w_1}+de_{w_2},{\rm n}(d));\\ 
D_{0,v}\iso &V\ot_\Q\Q_v,\quad  x\mapsto xe_{w_1}+x^*e_{w_2},\end{aligned}\eeq
where $\Q_v^\x$ sits inside $(D^\x_{0,v}\times D^\x_{0,v})$ as $(z,z)$.
For $(a,d)\in D_{0,v}^\x\xx D_{0,v}^\x$, we have $\varrho(a,d)x=axd^{-1}$ for $x\in D_{0,v}$.
\end{Remark}
\def\bfx{x}

\subsection{Notation for quaternion algebras}\label{SS:DATA}
We fix an isomorphism $\Phi_p:{\rm M}_2(\Qp)\to D_0\otimes\Qp$ for each $p\nmid N^-\infty$ once and for all 
and we put $\Phi=\prod_{p\ndivides N^-\infty}\Phi_p$. 
Let $\cO_{D_0}$ be the maximal order of $D_0$ such that $\cO_{D_0}\ot\Zp=\Phi_p({\rm M}_2(\Zp))$ 
for all $p\ndivides N^-$ and let $R^0:=\cO_{D_0}\ot_\Z\cO_F$ be a maximal order of $D$.
If $\frakA$ is an ideal of $\cO_F$ with $(\frakA,N^-)=1$, denote by $R_{\frakA}$ the standard Eichler order of $D$ of level $\frakA$ contained in $R^0$. 

For any ring $A$, the main involution $*$ on ${\rm M_2}(A)$ is given by 
\[ \pMX{a}{b}{c}{d}^*=\pMX{d}{-b}{-c}{a}.
\]

Let $\bbH$ be the Hamilton's quaternion algebra given by
\[  \bbH = \stt{ \pMX{z}{w}{-\ol{w}}{\ol{z}} \in {\rm M}_2(\C) }. 
  \] 
The main involution $\ast:{\mathbb H} \to {\mathbb H}$ is given by 
  $x\mapsto {}^{\rm t}\bar{x}$. 
Fix  an identification $\Phi_\infty: D_{0,\infty} \cong \bbH$
such that  $\Phi_\infty(x^\ast) = \Phi_\infty(x)^\ast$, 
which induces an embedding $\Phi_\infty: D^\times_{0,\infty} \cong \bbH^\times \hookto {\rm GL}_2(\C)$.

\subsection{Weil representation on ${\rm O}(V)\times\Sp(4)$} \label{secWeil}
Let $(\cdot, \cdot): V\times V \to \Q$ be the bilinear form defined by
$(x,y) = {\rm n}(x+y) - {\rm n}(x) - {\rm n}(y)$. 
Denote by ${\rm GO}(V)$ the orthogonal similitude group with the similitude morphism 
$\nu:{\rm GO}(V) \to \bbG_m$.
Let $\bfX = V\oplus V$. 
For a place $v$ of $\Q$, let $V_v=V\otimes_\Q\Q_v$ and $\bfX_v=\bfX\otimes_\Q\Q_v$. 
%The quadratic character attached to $V_v$ is given by the quadratic character $\chi_{F_v/\Q_v}$ attached to $F_v/\Q_v$. 
Denote by $\cS(\bfX_v)$ the space of $\C$-valued Bruhat-Schwartz functions on $\bfX_v$. Let $\cB_{\omega_v}:\cS(\bfX_v)\otimes \cS(\bfX_v)\to\C$ be the Hermitian pairing given by 
\begin{align*}
       {\mathcal B}_{\omega_v}(\varphi_{1,v}, \varphi_{2,v})
  = & \int_{\bfX_v}   \varphi_{1,v} (x_v) \bar{ \varphi_{2,v} (x_v)}\rmd x_v
\end{align*}
for $\varphi_{1,v}, \varphi_{2,v}\in {\mathcal S}({\mathbf X}_v)$. Here $\rmd x_v$ is the self-dual measure on $\bfX_v$ with respect to the Fourier transform determined by $\psi_v$.
%For each $\bfx=(x_1,x_2)\in \bfX_v=V_v\oplus V_v$, we put
%\[  S_\bfx= \pMX{ {\rm n}(x_1) }{ \frac{1}{2}(x_1,x_2)  }{\frac{1}{2}(x_1,x_2)}{{\rm n}(x_2)}. \]
Throughout, we consider the standard Schr\"{o}dinger realization of the Weil representation 
$\omega_{V_v}: {\rm Sp}_4(\bfQ_v) \to \Aut_\C(\bfX_v)$, 
which is explicitly given in \cite[Section 3.4]{hn15}.  
Let $\cR(\GO(V)\x \GSp_4)$ be the $R$-group
\[\cR(\GO(V)\x \GSp_4)=\stt{(h,g)\in \GO(V)\x\GSp_4\mid \nu(h)=\nu(g)}.\]
Then the Weil representation can be extended to the $R$-group by
\begin{align*}
\omega_v:&\cR({\rm GO}(V_v)\times {\rm GSp}_4(\Q_v))\to {\rm Aut}_\C\cS(\bfX_v),\\
    \omega_v(h,g)\varphi(x)
  = &|\nu(h)|^{-2}_v(\omega_{V_v}(g_1)\varphi)(h^{-1}x)
\quad (g_1 = \pMX{\bfone_2}{0}{0}{\nu(g)^{-1}\bfone_2 }g). 
\end{align*}
Let ${\cS}({\bfX}_{\A}) = \otimes_v'{\cS}({\bfX}_v)$ and let  
$\omega=\otimes_v\omega_v: \cR({\rm GO}(V)_{\A}\times {\rm GSp}_4({\A}))
\to {\rm Aut}_\C\cS(\bfX_{\A})$.
%\subsection{Yoshida lifts}\label{subsec:yoshidalifts}
\subsection{Representations of $H^0(\A)$}\label{subsec:fromsonH}
We fix $\ul{k}=(k_1,k_2)$ a pair of non-negative integers with $k_1\geq k_2$ and let $\frakN^+$ be a square-free product of primes ideal of $\cO_F$ with \[(\frakN^+,N^-\Delta_F)=1.\]  When $F=\Q\oplus \Q$, $\frakN^+$ is given by a pair of square-free positive integers $(N^+_1,N^+_2)$. 
Let $N^+$ be the square-free integer such that $N^+\Z={\mathfrak N}^+\cap \Z$ 
(so $N^+={\rm l.c.m}(N_1^+,N_2^+)$ if $F=\Q\oplus \Q$). Let $\frakN=\frakN^+N^-$. Let $f^{\rm new}$ be a \emph{newform}  on $\PGL_2(\A_F)$ of weight $2\ul{k}+2=(2k_1+2,2k_2+2)$ and level $\frakN$. Namely, $f^{\rm new}$ is a pair of elliptic modular newforms $(f_1,f_2)$ of level $(\Gamma_0(N_1^+N^-)$, $\Gamma_0(N_2^+N^-))$ and weight $(2k_1+2, 2k_2+2)$ if $F=\Q\oplus \Q$, while $f^{\rm new}$ is a Hilbert modular newform of level $\Gamma_0(\frakN)$ and weight $2\ul{k}+2$ if $F$ is a real quadratic field. Let $\pi$ be the cuspidal automorphic representation of $\PGL_2(\A_F)$ generated by the newform $f^{\rm new}$.

% \subsubsection{Automorphic forms on $H^0(\A)$}
Let $(\tau_{\ul{k}},\cW_\ulk(\C)):=(\tau_{k_1}\otimes\tau_{k_2},{\mathcal W}_{k_1}(\C)\otimes_\C {\mathcal W}_{k_2}(\C))$ be an algebraic representation of $D^\x$ via $\Phi_\infty$ 
and let $\pairing_\cW$ be the pairing on $\cW_{\ul{k}}(\C)$ given by 
$\pair{v_1\ot v_2}{v'_1\ot v'_2}_{\cW}=\pair{v_1}{v'_1}_{2k_1}\pair{v_2}{v'_2}_{2k_2}$, 
where $\langle\cdot,\cdot\rangle_{2k_i} (i=1,2)$ is the pairing introduced in Section \ref{algrep}. 
%For any $F'$-algebra $L$, we let $\bfM_{\ul{k}}(\wh R^\x,L)$ denote the space of modular forms of weight $\ul{k}$ defined over $L$ which consists of functions $f:\wh D^\x\to \cW_{\ul{k}}(L)$ such that  \begin{align*}
% &f(\gamma h uz) = \tau_{\ul{k}}(\gamma) f(h),\\
 %(h=&\wh D^\times, \, (\gamma, u,z)\in D^\times \times \wh R^\x\times F^\times_{\mathbf A}).    
%\end{align*}
Let $D_\A=D\ot_\Q\A$. For an ideal $\frakA$ of $\cO_F$ with $(\frakA,N^-)=1$, denote by ${\mathcal M}_{\ul{k}}(D^\times_{\mathbf A},\frakA)$ the space of $\cW_\ulk(\C)$-valued modular forms on $D_\A^\x$, consisting of functions
$\bff:D^\times_{\mathbf A}\to {\mathcal W}_{\ul{k}}(\C)$
such that
\begin{align*}
 &\bff(z\gamma h u) = \tau_{\ul{k}}(h^{-1}_\infty) \bff(h_f),\\
 (h=&(h_\infty,h_f)\in D^\times_{\mathbf A},\, (z,\gamma, u)\in F^\times_{\mathbf A}\times D^\times \times \wh R_{\frakA}^\x).    
\end{align*}
Hereafter, we shall view $\bff$ as a $\cW_\ulk(\C)$-valued function on $Z_{H^0}(\A)\bksl H^0(\A)$ by the rule $\bff(a,\alpha):=\bff(a)$. For $u\in\cW_\ulk(\C)$, let $\bff_u(h):=\pair{\bff(h)}{u}_\cW$. Then $\bff_u$ is an automorphic form on $H^0(\A)$.

Let $\pi^D$ be the irreducible automorphic representation of $D_\A^\x$ obtained via the Jacquet-Langlands transfer of $\pi$ and let $\cA_{\pi^D}$ be the corresponding space of $\pi^D$.  Then we have an identification $i:\cM_\ulk(D_\A^\x,\frakA)\iso \bigoplus_\pi\Hom_{D_\infty^\x}(\cW_\ulk(\C),\cA_{\pi^D}^{\wh R_\frakA^\x})$ given by $i(\bff)(u)=\bff_u$. In addition, to the newform $f^{\rm new}$, we can associate a $\cW_\ulk(\C)$-valued modular form on $\newform\in \cM_\ulk(D^\x_\A,\frakN^+)$, which is characterized by the property that $\newform$ shares the same Hecke eigenvalues of $f^{\rm new}$ at primes not dividing $\frakN$. Moreover, $\newform$ is unique up to a scalar by strong multiplicity one for $\GL(2)$ and local theory of newforms. 

We recall the local Atkin-Lehner involutions. If $p\ndivides N^-$, then $H^0(\Q_\q)=(\GL_2(F_\q)\times \Q_\q^\x)/F_\q^\x$, and for each prime $\frakp$ of $\cO_F$ lying above $p$, let $\uf_\frakp$ be a uniformizer of $\frakp$ and put
\beq\label{E:Atkin1}\eta_\frakp:=(\pMX{0}{1}{-\uf_\frakp}{0},1)\in H^0(\Q_\q).\eeq
If $p\divides N^-$, then $D_{0,\q}$ is the division algebra over $\Q_\q$ and $p=\frakp\frakp^c$ is split in $F$. %Identifying $H^0(\Q_\q)$ with $(D_{0,\q}^\x\times D_{0,\q}^\x)/\Delta(\Q_\q^\x)$ with respect to $\frakp$ as in \remref{R:split}, 
Let $\uf^D_\frakp\in D_\q^\x$ such that ${\rm n}(\uf^D_\frakp)\in F_\frakp^\x$ is a uniformizer of $\frakp$. Put
\beq\label{E:Atkin2}\eta_\frakp:=(\uf^D_\frakp,1),\quad \eta_{\frakp^c}=(\uf^D_{\frakp^c},1)\in H^0(\Q_\q).\eeq
It is well known that if $\frakp\divides\frakN^+$, then $\newform$ is an eigenfunction of the right translation of $\eta_\frakp$ (Atkin-Lehner involution at $\frakp$). In other words, we have
\beq\label{E:Atkineigenvalue}\newform(h\eta_\frakp)=\e_\frakp\cdot\newform(h)\text{ for }\frakp\divides \frakN^+.\eeq
We call $\e_\frakp\in\stt{\pm 1}$ the Atkin-Lehner eigenvalue of $\newform$ at $\frakp$. Moreover, if we denote by $\ep(\pi_\frakp)$ the local root number of the local component $\pi_\frakp$ of $\pi$ at $\frakp$, then $\e_\frakp=\ep(\pi_\frakp)$ for $\frakp \ndivides N^+$ and $\e_\frakp=-\ep(\pi_\frakp)$ for $\frakp\divides N^-$.

We next introduce the modular form $\bff^\dag$ obtained by applying certain level-raising operators to the newform $\newform$. Let $\cP$ be a finite subset of finite places of $\Q$ given by 
\beq\label{E:badprime} \begin{aligned} \cP =&\stt{\text{rational primes  }p\mid p=\frakp\frakp^c\text{ is split in }F\text{ with } \frakp\ndivides \frakN^+\text{ and }\frakp^c\divides\frakN^+}\\
=&\stt{\text{prime factors $p$ of }N^+\mid p\ndivides \frakN^+}.\end{aligned}
\eeq
 %For each $\q\in\cP$, we identify $H^0(\Q_\q)$ with $(\GL_2(\Q_\q)\x\GL_2(\Q_\q))/\Delta(\Q_\q^\x)$ with respect to $\frakq_1$ as in \remref{R:split}, and hence $\pi^D_\q\iso\pi_{\frakq_1}\boxtimes\pi_{\frakq_2}$ with irreducible admissible representations $\pi_{\frakq_i}$ of $\PGL_2(\Q_\q)$ such that $\pi_{\frakq_1}$ is spherical and $\pi_{\frakq_2}$ is unramified special. Put
%\[\eta_\q:=(\pMX{0}{1}{q}{0},\bfone_2)\in H^0(\Q_\q)\]
%Let $\varepsilon_{\frakq_2}\in\stt{\pm 1}$ be the Atkin-Lehner eigenvalue of $\pi_{\frakq_2}$, so we have
%\[\bff^\circ(h\eta_\q^c)=\varepsilon_{\frakq_2}\cdot \bff^\circ(h),\quad \eta_\q^c=(\bfone_2,\pMX{0}{1}{q}{0}).\]
Define the level raising operator $\sV_\q:\cM_\ulk(D_\A^\x,\frakN^+)\to\cM_\ulk(D_\A^\x,\frakN^+\frakp)$ for each $p\in \cP$ by 
\begin{align*}
   \sV_\q(\bff)(h) = \bff (h) +\ep_{\frakp^c}\cdot \bff (h\eta_\frakp).
\end{align*}
Let $\bff^\dag\in\cM_\ulk(D^\x_\A,N^+\cO_F)$ be the \emph{$\cP$-stabilized newform} defined by 
\beq\label{E:levelraisingf}\bff^\dag=\sV_\cP(\newform),\quad\sV_\cP:=\prod_{\q\in\cP}\sV_\q.\eeq
By definition, $\bff^\dag=\newform$ is the newform if $\frakN^+=N^+\cO_F(\iff \cP=\emptyset)$.

\subsection{The test function $\test$}\label{SS:testfunctions}
 We recall a distinguished Bruhat-Schwartz function 
$\test 
    \in\cS(\bfX_\A)\ot \cW_{\underline{k}}(\C)\otimes \cL_\kappa(\C)$ introduced in \cite[Section 3.6]{hn15}.    
Let \[R:=R_{N^+\cO_F}\] be the standard Eichler order of level $N^+\cO_F$ and let
\[L:=R\cap V\]be the lattice of $V$ determined by $R$. At each finite place $p$, $L_p=L\ot_\Z\Zp$ and define $\test_p\in \cS(\bfX_p)$ by 
\beq\label{finitetest}
  \test_p =   {\mathbb I}_{L_p\oplus L_p} \text{ the characteristic function of $L_p\oplus L_p$. }  
\eeq
Note that $\test_p$ is invariant by $R_p^\x\x \Zp^\x$ under the Weil representation $\om_p$. At the archimedean place $\infty$, 
we have identified $H^0(\R)$ with $(\bbH^\x\times\bbH^\x)/\R^\x$   
and $V_\infty$ with $\bbH$ in (\ref{E:splittrivialization}) with respect the inclusion $F\hookto\R$. For $0\leq\al\leq 2k_2$, define the function $P_\ulk^\al: \bfX_\infty=\bbH^{\oplus 2} 
  \to \cW_\ulk(\C)=\C[X_1,Y_1]_{2k_1} \ot_\C \C[X_2,Y_2]_{2k_2}$ by
\beq\label{E:vpalpha.def}\begin{aligned}
      &P_\ulk^\al( \begin{pmatrix} z_1 & w_1 \\ -\bar{w}_1 & \bar{z}_1 \end{pmatrix} , \begin{pmatrix} z_2 & w_2 \\ -\bar{w}_2 & \bar{z}_2 \end{pmatrix}  )\\
  = 
 &( ( z_1\bar{z}_2+w_1\bar{w}_2-\bar{w}_1w_2-\bar{z}_1z_2 ) X_1Y_1 + (z_1w_2-w_1z_2 ) X^2_1 + (\bar{z}_1\bar{w}_2-\bar{z}_2\bar{w}_1 )Y^2_1 )^{k_1-k_2}   \\
     &\times  (\bar{z}_1Y_1\otimes X_2 + w_1 X_1\otimes X_2 -  z_1 X_1\otimes Y_2 + \bar{w}_1 Y_1\otimes Y_2)^\alpha  \\
     &\times (\bar{z}_2Y_1\otimes X_2 + w_2 X_1\otimes X_2 -  z_2 X_1\otimes Y_2 + \bar{w}_2 Y_1\otimes Y_2)^{2k_2-\alpha}.  
\end{aligned}\eeq
Then the archimedean Bruhat-Schwartz function 
$\test_\infty : \bfX_\infty = \bbH^{\oplus 2} 
    \to \C[X_1, Y_1]_{2k_1} \ot_\C \C[X_2,Y_2]_{2k_2} \ot_\C \C[X,Y]_{2k_2}$ is defined by  
\beq\label{infinitetest}
  \test_\infty(x) =e^{-2\pi({\rm n}(x_1)+{\rm n}(x_2))}\sum_{\al=0}^{2k_2}P^\al(x_1,x_2)\cdot \binom{2k_2}{\al}X^\al Y^{2k_2-\al}
  \quad (x=(x_1,x_2)\in\bfX_\infty).
\eeq
We note that the following identity holds:
\begin{align*}
   (\omega_\infty(h,u)\test_\infty)(x) = \tau_{\ul{k}}(h^{-1})\ot\rho_\kappa({}^{\rm t}u)(\test_\infty(x)).  
\end{align*}
for each $(h, u) \in H^0_1({\R})\times {\rm U}_2(\R)$
by \cite[Lemma 3.5]{hn15}.

%where the polynomial $P_{\ul{k}} : \bbH^{\oplus 2} 
  %\to \C[X_1,Y_1]_{2k_1} \ot_\C \C[X_2,Y_2]_{2k_2} \ot_\C \C[X,Y]_{2k_2}$ is defined as follows: \begin{align*}
     %&  P_{\underline{k}} ( \begin{pmatrix} z_1 & w_1 \\ -\bar{w}_1 & \bar{z}_1 \end{pmatrix} , \begin{pmatrix} z_2 & w_2 \\ -\bar{w}_2 & \bar{z}_2 \end{pmatrix}  )  \\
  %= & ( ( z_1\bar{z}_2+w_1\bar{w}_2-\bar{w}_1w_2-\bar{z}_1z_2 ) X_1Y_1 + (z_1w_2-w_1z_2 ) X^2_1 + (\bar{z}_1\bar{w}_2-\bar{z}_2\bar{w}_1 )Y^2_1 )^{k_1-k_2}    \\
     %& \times \sum^{2k_2}_{\alpha=0} (\bar{z}_1Y_1\otimes X_2 + w_1 X_1\otimes X_2 -  z_1 X_1\otimes Y_2 + \bar{w}_1 Y_1\otimes Y_2)^\alpha  \\
     %& \quad \quad \times (\bar{z}_2Y_1\otimes X_2 + w_2 X_1\otimes X_2 -  z_2 X_1\otimes Y_2 + \bar{w}_2 Y_1\otimes Y_2)^{2k_2-\alpha}\binom{2k_2}{\alpha}X^{\alpha}Y^{2k_2-\alpha}.  
%\end{align*}

\subsection{Theta lifts from $\GSO(V)$ to $\GSp_4$}
%Let $\newform\in\cM_\ulk(D^\x_\A,\wh R^\x)$ be the automorphic form  on $Z_{H^0(\A)} \bksl H^0(\A)$associated with the newform $f_\pi^{\rm new}$ on $\GL_2(F_\A)$ of weight $2\ul{k}+2=(2k_1+2,2k_2+2)$ and level $\Gamma^{(1)}_0(\frakN)$.  

Let $\kappa:=(k_1+k_2+2, k_1-k_2+2)$. 
For each vector-valued Bruhat-Schwartz function 
$\varphi\in\cS(\bfX_{\A})\otimes \cW_{\ul{k}}(\C)\otimes \cL_\kappa(\C)$, 
define the theta kernel 
$\theta(-,-;\varphi): \cR({\rm GO}(V)_{\A}\times {\rm GSp}_4({\A}) ) 
   \to \cW_{\ul{k}}(\C)\ot \cL_\kappa(\C)$ by
\[
  \theta(h,g;\varphi) = \sum_{x\in\bfX} \omega(h,g)\varphi(x).
\]
Let $\GSp^+_4$ be the group of elements $g\in\GSp_4$ with $\nu(g)\in \nu(\GO(V))$. Let $U_R=\prod_v U_{R_v}$ be the open-compact subgroup of $H^0(\A)$ given by \beq\label{E:opcpt}\begin{aligned}U_{R_v}=&H^0(\R)\text{ if }v=\infty;
   \quad U_{R_v}=(R_v^\x \times {\mathbf Z}_v^\x)/\cO_{F_v}^\x\text{ if }v < \infty,\\
\cU:=&H^0_1(\A)\cap U_R. \end{aligned}
\eeq 
For a vector-valued function $\bff:H^0(\Q)\bksl H^0(\A)\to \cW_\ulk(\C)$, define the theta lift $\theta(\bff ,\varphi): \GSp_4^+(\Q)\bksl {\rm GSp}_4^+({\A})\to \cL_\kappa(\C)$ by 
\begin{align*}
  \theta(\varphi,\bff)(g) 
  = \int_{[H_1^0]} \pair{\theta( hh^\prime, g;\varphi)}{ \bff(hh^\prime)}_{\cW}{\rm d}h
  \quad (\nu(h^\prime) = \nu(g)).
\end{align*}
Here $\rmd h:=\prod_v\rmd h_v$ is the Tamagawa measure on $H_1^0(\A)$.% with $\vol(\cU_v)=1$ for $v\ndivides N^+N^-$. 
We extend uniquely $\theta(\varphi,\bff)$ to a function on $\GSp_4(\Q)\bksl \GSp_4(\A)$ 
by defining $\theta(\varphi,\bff)(g)=0$ for $g\not\in\GSp_4(\Q)\GSp_4^+(\A)$.
%We also denote $\theta(-;\bff, \varphi)$  by $\theta(\varphi, \bff)$.
\begin{defn}The Yoshida lift is the theta lift $\theta(\test,\bff^\dag)$ attached to the \BS function $\test:=\bigot_v\test_v\in\cS(\bfX_\A)\ot \cW_{\underline{k}}(\C)\otimes \cL_\kappa(\C)$ and the $\cP$-stabilized newform $\bff^\dag$ attached to the cuspidal automorphic representation $\pi$ of $\PGL_2(\A_F)$. When $k_2=0$, $\theta(\test,\bff^\dag)$ is the scalar valued Siegel modular form constructed by Yoshida \cite{yo80}.\end{defn}

\begin{prop}
Let $N_F=N^+N^-\Delta_F$. The Yoshida lift $\theta(\test,\bff^\dag)$  is a Siegel modular form of weight $\kappa$, level $\Gamma^{(2)}_0(N_F)$ 
and of type $\eta_{F/\Q}$ with the trivial central character. Moreover, $\theta(\test,\bff^\dag)$ is a cusp form if $\pi$ is not Galois self-dual. 
\end{prop}
\begin{proof}This follows directly from \cite[\S 3.7, Lemma 3.2, 3.3 and 3.4]{hn15}.\end{proof}

%\input{Heckop}
% !TEX root = YPNorm.tex
\def\bfmu{\boldsymbol{\mu}}
\def\piD{\pi^D}
\def\piHo{\sigma}
\section{Rallis Inner product formula of Yoshida lifts }\label{secInnprd}

In this section, we realize Yoshida lifts as theta lifts from $\GO(V)$ to $\GSp_4$ and apply the Rallis inner product formula to calculate the inner product of the Yoshida lift $\theta(\test, \bff^\dag)$. 
 
 \subsection{Automorphic forms on $\GO(V)$}\label{cuspGO}
In this subsection, we will retain the notation in \subsecref{subsec:fromsonH}. Let $H={\rm GO}(V)$. Let $\bft$ be the order two element in $H(\Q)$ with the action $x\mapsto x^*,\,x\in V$. 
Let $\bfmu_2=\stt{1,\bft}$ and we regard $\bfmu_2$ as the multiplicative group scheme of order $2$ 
 defined over $\Q$. 
 For each $v\in\Sigma_\Q$, let $\bft_v$ be the image of $\bft$ in $H(\Q_v)$. If $\cR$ is a subset of $\Sigma_\Q$, denote by $\bft_\cR\in\bfmu_2(\A)$ the element such that $(\bft_\cR)_v=\bft_v$ for $v\in\cR$ and $(\bft_\cR)_v=1$ if $v\not\in\cR$. Then we have $H(\A)=H^0(\A)\bfmu_2(\A)$. For $h=\varrho(a,\alpha)\in H^0(\A)=(D_\A^\x\times\A^\x)/\A_F^\x$, 
put $h^c=\varrho(a^c,\alpha)$. One verifies easily that $\bft h \bft=h^c$. 
\subsubsection{From $\GSO(V)$ to $\GO(V)$}
Recall that $\piD$ is the Jacquet-Langlands transfer of $\pi$. We have $(\piD,\cA_{\piD})\iso\ot_v(\piD_v,\cV_v)$, where $\piD_v$ is an irreducible admissible representation of $D_v^\x$ on the space $\cV_v$. Let $\piHo=\piD\boxtimes\bfone$ be an automorphic representation of $H^0(\A)$ with the space $\cA_{\piHo}=\cA_{\piD}$. Here $\bfone$ is the trivial representation of $\A^\x$. We have $\piHo\iso\ot_v\piHo_v$, where $\piHo_v=\piD_v\boxtimes \bfone$ with the same space $\cV_v$.  Let $R=R_{\frakN^+}$ is the Eichler order of $D$ of level $\frakN^+$. %For each place $v\in\Sigma_\Q$, let $H_v=H(\Q_v)$ and $H^0_v=H^0(\Q_v)$. 
If $v$ is a finite place of $\Q$, viewing $R_v^\x=(R\ot\Z_v)^\x$ as a subgroup of $H^0(\Z_v)$, the $R^\x_v$-invariant subspace of $\cV_v$ is one-dimensional by the theory of newforms of irreducible representations of $D_v^\x$. We fix a non-zero vector $f^0_v$ in $\cV_v^{R^\x_v}$.

Let $\piHo^\sharp_v:=\Ind_{H^0(\Q_v)}^{H(\Q_v)}\piHo_v$ be the induced representation of $H(\Q_v)$. Namely, the space of $\piHo^\sharp_v$ is $\cV_v^\sharp:=\cV_v\oplus\cV_v$, on which $h\in H^0(\Q_v)$ acts by $\piHo^\sharp_v(h)(x,y)=(\piHo_v(h)x,\piHo_v(h^c)y)$ and $\piHo^\sharp_v(\bft_v)(x,y)=(y,x)$ for all $x,y\in\cV_v$. We define the sub-representation $\wtd\piHo_v\subset \piHo^\sharp_v$ of $H(\Q_v)$ with the space $\wtd\cV_v\subset \cV_v^\sharp$ as follows. Let
\[\frakS=\stt{v\in\Sigma_\Q\mid \piHo_v\iso\piHo^c_v}.\]
\begin{mylist}\item $v\not\in\frakS$: in this case, $\piHo^\sharp_v$ is irreducible, and we set $(\wtd\piHo_v,\wtd\cV_v):=(\piHo_v^\sharp,\cV_v^\sharp)$.
\item $v\in\frakS$: in this case, there exist two linear maps $\xi^\pm:\cV_v\to \cV_v$ such that $\xi^\pm\circ \piHo_v(h)=\piHo_v(h^c)\circ\xi^\pm$ for all $h\in H^0(\Q_v)$, $(\xi^\pm)^2={\rm Id}$ and $\xi^+=(-1)\cdot \xi^-$. If $v$ is finite, let $\xi^+$ be chosen so that $\xi^+(f_v^0)=f_v^0$ (this is possible as $R_v^c=R_v$ for $v\in\frakS$), and if $v$ is archimedean, then $\cV_\infty=\cW_{k_1}(\C)\ot\cW_{k_2}(\C)$ with $k_1=k_2$ since $\piHo_\infty\iso\piHo^c_\infty$, let $\xi^+$ be the map $u_1\ot u_2\mapsto u_2\ot u_1$. Let $\piHo^\pm_v$ be the sub-representation of $\piHo^\sharp_v$ with the space given by 
\[\cV_v^\pm=\stt{(x,\xi^\pm(x))\in\cV_v^\sharp\mid x\in \cV_v}.\]
Then $\piHo^\pm_v\iso\piHo_v$  with the action $\bft_v$ via $\xi^\pm$. We define \[
(\wtd\piHo_\infty,\wtd\cV_\infty)=(\piHo^\sharp_\infty,\cV_\infty^\sharp);\quad(\wtd\piHo_v,\wtd\cV_v)=(\piHo_v^+,\cV_v^+)\text{ if $v$ is finite.}\]
%\item $\piHo_v\iso \piHo_v^c$ and $v=\infty$: in this case, $\piHo_\infty=\cW_\ulk(\C)$ with $k_1=k_2$. We define ${\wtd\piHo}_\infty^\pm$ to the the representation space $\piHo_\infty=\cW_{k_1}(\C)\ot\cW_{k_2}(\C)$ with the extended $H_\infty$-action by letting $\bft_\infty(u_1\ot u_2)=\pm(u_2\ot u_1)$. We set $\wtd\piHo_\infty:=\wtd\piHo_\infty^+$.
\end{mylist}
Let $\wh\piHo$ be the automorphic representation of $H(\A)$ whose space $\cA_{\wh\piHo}$ consisting of automorphic forms $f$ on $H(\A)$ such that  $f|_{H^0(\A)}\in \cA_{\piHo}$. Suppose that $\piHo\not\iso \piHo^c$. It is well known that  
\[\wh\sg\iso \bigoplus_\delta\left(\bigot_{v\in\frakS}\piHo_v^{\delta(v)}\bigot_{v\not\in\frakS}\piHo^\sharp_v\right),\]
where $\delta$ runs over all maps from $\frakS$ to $\stt{\pm 1}$such that $\delta(v)=+1$ for all but finitely $v\in\frakS$
(\cf\cite[Proposition 5.4]{takeda09}).
In particular, there exists a unique constitute $\wtd\piHo\subset \wh\piHo$ with the space $\cA_{\wtd\piHo}\subset \cA_{\wh\piHo}$ such that $\wtd\piHo\iso \ot_v\wtd\piHo_v$.  Let $\wtd\piHo^+$ be a unique irreducible constitute  of $\wtd\piHo$ with the space $\cA_{\wtd\piHo^+}\subset\cA_{\wtd\piHo}$ given as follows: \begin{mylist}\item $\infty\not\in\frakS$: let $(\wtd\piHo^+,\cA_{\wtd\piHo^+}):=(\wtd\piHo,\cA_{\wtd\piHo})$;\item$\infty\in\frakS$: then $(\wtd\piHo,\cA_{\wtd\piHo})=(\wtd\piHo^+,\cA_{\wtd\piHo^+})\oplus(\wtd\piHo^-,\cA_{\wtd\piHo^-})$ is reducible, where
\[(\wtd\piHo^\pm,\cA_{\wtd\piHo^\pm})\iso(\piHo^\pm_\infty\bigot_{v\not =\infty}\wtd\piHo_v,\cV_\infty^\pm \bigot_{v\not=\infty}\wtd\cV_v).\]
\end{mylist}
\begin{Remark}\label{R:1.I}When $v\in\frakS$, our choices of $\piHo^+_v$ agree with those in \cite[\S 6.1]{takeda11}. By \cite[Proposition 6.5]{takeda11}, the local theta lifts $\theta(\piHo^+_v)$ to $\GSp_4(\Q_v)$ is non-zero, and if $v\in\frakS$ is split in $F$, then $\theta(\piHo^-_v)$ is zero. In particular, the global theta lift $\theta(\wtd\piHo^-)$ is zero if $\infty\in\frakS$.
\end{Remark}
\subsubsection{Automorphic forms}Let $\bff\in \cM_\ulk(D_\A^\x,\frakA)$. For $h\in H^0(\A)$ and $u\in\cV_\infty=\cW_\ulk(\C)$, put
\[\bff_u(h)=\pair{\bff(h)}{u}_\cW.\]
%Let ${\mathbf f}: H^0(\A) = {\rm GSO}(V)(\A) \to {\mathcal W}_{\underline{k}}(\C)$ be the newform as before. For every $u\in \piHo_\infty=\cW_\ulk(\C)$,  define $\bff_u:H^0(\A)\to\C$ by
%\begin{align*}
  %{\mathbf f}_u(h_0) = \langle {\mathbf f}(h_0), u  \rangle_{\mathcal W} \quad (h_0\in H^0(\A)),
%\end{align*}
For $h\in H^0(\A)$ and $\bft_\cR\in\bfmu_2(\A)$, put
% and define $\wtd\bff_u: H(\A) \to \C$ by
\begin{align}\label{tildef}
     \wtd\bff_u(h \bft_\cR)
  = \begin{cases}  \bff_u(h),  &  \infty \not\in {\mathcal R}, \\ 
                          \bff_u( h^c),  &  \infty \in {\mathcal R}.        
     \end{cases}
\end{align}
Then $\bff_u\in\cA_{\piHo}$ and $\wtd\bff_u\in \cA_{\wtd\piHo}$. For each finite place $v$ and $f_v\in\cV_v$, we put $\wtd f_v=(f_v,f_v)\in\cV_v^\sharp$ if $v\not\in\frakS$, and $\wtd f_v=f_v\in \cV_v^+$ if $v\in\frakS$. We shall fix an isomorphism $j:\ot_v\wtd\piHo_v\iso\wtd\piHo$ such that  
\beq\label{E:iso.I} j((u_1,u_2)\bigot_{v<\infty} \wtd f_v^0)=
\wtd\bff^\circ_{(u_1,u_2)}:=\wtd\bff^\circ_{u_1}  +\wtd\piHo(\bft_\infty)\wtd\bff^\circ_{u_2}\in\cA_{\wtd\piHo}.\eeq
Note that if $\infty\in\frakS$, then \[\wtd\bff^\circ_{(u,\pm \bft_\infty u)}\in \cA_{\wtd\piHo^\pm},\text{ and }j(\ot_v\cV_v^\pm)=\cA_{\wtd\piHo^\pm}.\]
For each finite place $\q$, put
\beq\label{E:dagger}f_\q^\dag:=\begin{cases}f_\q^0&\text{ if }\q\not\in\cP,\\
f_\q^0+\ep_{\frakp^c}\cdot\piHo_\q(\eta_\q)f_\q^0&\text{ if }\q\in\cP.\end{cases}
\eeq
By the definition of $\cP$-stabilized newform $\bff^\dagger$ \eqref{E:levelraisingf}, we have
\[j((u_1,u_2)\bigot_\q \wtd f_\q^\dag)=\wtd\bff^\dagger_{u_1}+\wtd\piHo(\bft_\infty)\wtd\bff_{u_2}^\dag.\]

\subsubsection{Hermitian pairings} If $v$ is a finite place, let 
$\cB_{\piHo_v}:\cV_v\ot\bar\cV_v\to\C$ be the $H^0(\Q_v)$-invariant pairing such that $\cB_{\piHo_v}(f_v^0,f_v^0)=1$. If $v=\infty$, put
\beq\label{E:w0.I}\cJ=(\pMX{0}{1}{-1}{0},\pMX{0}{1}{-1}{0})\in  \bbH^\x\times\bbH^\x/\R^\x=H^0(\R),\eeq
and let $\cB_{\piHo_\infty}:\cW_\ulk(\C)\ot\cW_\ulk(\C)\to\C$ be the pairing given by \[\cB_{\piHo_\infty}(u_1,u_2)=\pair{u_1}{\tau_\ulk(\cJ)\ol{u_2}}_\cW.\] Let $\cB_{\piHo^\sharp_v}:\cV_v^\sharp\ot\bar\cV_v^\sharp\to\C$  be the pairing given by
\[\cB_{\piHo^\sharp_v}((u_1,w_1),(u_2,w_2)):=\onehalf(\cB_{\piHo_v}(u_1,u_2)+\cB_{\piHo_v}(w_1,w_2)).\]
Let $\cB_{\piHo_v^\pm}:=\cB_{\piHo^\sharp_v}|_{\cV_v^\pm}$ if $v\in\frakS$. By definition, if $v$ is finite, we have $\cB_{\wtd\piHo_v}(\wtd f^0_v,\wtd f^0_v)=1$.

Let $\rmd \wtd h$ (resp. $\rmd h_0$) be the Tamagawa measure on $Z_{H}(\A)\backslash H(\A)$ (resp. $Z_{H^0}(\A)\backslash H^0(\A)$). Let $\rmd\epsilon_v$ be the Haar measure on $\mu_2(\Q_v)$, which satisfies ${\rm vol}(\mu_2(\Q_v),\rmd\epsilon_v)=1$ and let $\rmd\epsilon$ be the product measure $\prod_v \rmd\epsilon_v$ on $\mu_2(\A)$. Then for each $f \in L^1(Z_H(\A)H(\Q)\bksl H(\A))$, 
\begin{align*}
    \int_{ Z_H(\A)H(\Q)\bksl H(\A) } f(\wtd h)\rmd \wtd h
 = \int_{\mu_2(\Q) \backslash \mu_2(\A)} 
    \int_{ Z_{H^0}(\A)H^0(\Q)\bksl H^0(\A) } f(h_0 \epsilon)  \rmd h_0 \rmd\epsilon. 
\end{align*}
Define the Petersson pairing $\cB_{\wtd\piHo}:\cA_{\wtd\piHo}\ot\bar\cA_{\wtd\piHo}\to\C$ by 
\begin{align*}
  {\mathcal B}_{{\wtd\piHo}}(f_1, f_2)
=   \int_{Z_H(\A)H(\Q)\bksl H(\A) } 
       f_1(\wtd h) \bar{ f_2(\wtd h) }\rmd \wtd h.
\end{align*}
Let $\pairing_{H^0}$ be the Hermitian pairing on $\cM_\ulk(D^\x_\A,\frakA)$ given by 
\begin{align*}
      \pair{\bff_1}{\bff_2}_{H^0}
  =  \int_{Z_{H^0}(\A)H^0(\Q)\bksl H^0(\A)} 
     \pair{\bff_1(h_0)}{\tau_\ulk(\cJ)\bar{\bff_2(h_0)}}_\cW\rmd h_0.\end{align*}

\begin{lm}\label{L:2.I} We have
\begin{mylist}\item  $\pair{\bff^\dag}{\bff^\dag}_{H^0}=\pair{\newform}{\newform}_{H^0}\cdot\prod_{p\in\cP}\cB_{\piHo_p}(f_p^\dag,f_p^\dag)$;\item if $\piHo\not\iso\piHo^c$, then \[ \cB_{\wtd\piHo}=\frac{ \pair{\newform}{\newform}_{H^0} }{ \dim \cW_{\underline{k}}(\C)}\cdot \prod_v\cB_{\wtd\piHo_v}\]
under the isomorphism $\wtd\piHo\iso\ot_v\wtd\piHo_v$ in \eqref{E:iso.I}.
\end{mylist}
\end{lm}
\begin{proof}From the Schur orthogonality relations, we see that  for $u_1,u_2\in\cV_\infty=\cW_\ulk(\C)$,
\[\cB_{\piHo}(\bff^\circ_{u_1},\bff^\circ_{u_2})=\frac{\pair{\newform}{\newform}_{H^0}}{\dim\cW_\ulk(\C)}\cdot\pair{u_1}{\tau_\ulk(\cJ)\ol{u_2}}_\cW=\frac{\pair{\newform}{\newform}_{H^0}}{\dim\cW_\ulk(\C)}\cdot\cB_{\piHo_\infty}(u_1,u_2).\]
On the other hand, note that for $\wtd f_1,\wtd f_2\in\cA_{\wtd\piHo}$, 
\[\cB_{\piHo}(\wtd f_1|_{H^0(\A)},\wtd f_2|_{H^0(\A)})=2\cB_{\wtd\piHo}(\wtd f_1,\wtd f_2)\] 
by \cite[Lemma 2.1]{gi11} and that for $f_1,f_2\in \cA_{\piHo}$, 
\[\int_{Z_{H^0}(\A)H^0(\Q)\bksl H^0(\A)}f_1(h_0)\ol{f_2(h_0^c)}\rmd h_0=0\]
since $\piHo^\vee=\piHo\not\iso\piHo^c$. We thus find that \beq\label{E:4.I}\begin{aligned}
\cB_{\wtd\piHo}(\wtd\bff^\circ_{(u_1,w_1)},\wtd\bff^\circ_{(u_2,w_2)})=&\onehalf\int_{Z_{H^0}(\A)H^0(\Q)\bksl H^0(\A)}(\bff_{u_1}(h_0)+\bff^\circ_{w_1}(h_0^c))(\bff^\circ_{u_2}(h_0)+\bff^\circ_{w_2}(h_0^c))\rmd h_0\\
=&\onehalf(\cB_{\piHo}(\bff^\circ_{u_1},\bff^\circ_{u_2})+\cB_{\piHo}(\bff^\circ_{w_1},\bff^\circ_{w_2}))\\
=&\frac{\pair{\newform}{\newform}_{H^0}}{\dim\cW_\ulk(\C)}\cdot\cB_{\wtd\piHo_\infty}((u_1,w_1),(u_2,w_2)).
\end{aligned}\eeq
If $\infty\not\in\frakS$, then $\wtd\piHo$ is irreducible, and we can write $\cB_{\wtd\piHo}=C_0\cdot\prod_v\cB_{\wtd\piHo_v}$ for some constant $C_0$, while if $\infty\in\frakS$, then $\wtd\piHo=\wtd\piHo^+\oplus\wtd\piHo^-$ is reducible and 
\[\cB_{\wtd\piHo}=(C_+\cB_{\piHo^+_\infty}+C_-\cB_{\piHo^-_\infty})\prod_{v<\infty}\cB_{\wtd\piHo_v}\]
for some constants $C_\pm$. In either of the cases, \eqref{E:4.I} implies that $C_0=C_\pm=\frac{\pair{\newform}{\newform}_{H^0}}{\dim\cW_\ulk(\C)}$, and the lemma follows.
\end{proof}

\subsection{Rallis inner product formula}\label{secRallisInn}

From here to the end of this paper, 
we always assume that  the automorphic representation $\pi$ of ${\rm PGL}_2(\A_F)$ introduced in Section \ref{subsec:fromsonH} is not Galois self-dual. 

Let $G=\GSp_4$. For $g\in G(\A)^+$, $\varphi\in\cS(\bfX_\A)$ and $f\in \cA_{\wtd\piHo}$, choose $h\in H(\A)$ such that $\nu(g)=\nu(h)$ and put
\begin{align*}
 % \theta(\varphi^\al_I, \bff_I) (g)
 %= &\int_{[H_1^0]}  \theta(h_0h^\prime,g;\varphi^\al_I)\bff_I(h_0h^\prime)\rmd h_0 \\
  \theta(\varphi,f)(g)
 = &\int_{[H_1]}  \theta(h_1h,g;\varphi)f(h_1h)  \rmd h_1,  
\end{align*}
where $H_1={\rm O}(V)$ and $\rmd h_1=\prod_v\rmd h_{1,v}$ is the Tamagawa measure of $H_1(\A)$ such that $f\in L^1(H_1(\Q_v))$, we have
\beq\label{E:H1H0} \int_{H_1(\Q_v)}f(h_{1,v})\rmd h_{1,v}=\int_{\bfmu_2(\Q_v)}\int_{H^0_1(\Q_v)}f(h_v\epsilon_v)\rmd h_v\rmd \epsilon_v.\eeq

\begin{prop}[Rallis inner product formula]\label{P:GIRallis}
Let $\varphi_1=\ot_v\varphi_{1,v},\varphi_2=\ot\varphi_{2,v}\in\cS(\bfX_\A)=\ot_v\cS(\bfX_v)$ and $f_1=\ot_v f_{1,v},\,f_2=\ot_v f_{2,v}\in\cA_{\wtd\piHo^+}\iso\ot_v\wtd\cV_v^+$. Then \begin{align*}
  \pair{\theta(\varphi_1,f_1)}{\theta(\varphi_2,f_2)}:=&  \int_{Z_G(\A)G(\Q)\bksl G(\A)}
     \theta(\varphi_1, f_1)(g) \bar{\theta(\varphi_2,f_2)(g)}\rmd g  \\
  = &  \frac{\pair{\newform}{\newform}_{H^0}}{\dim\cW_\ulk(\C)} \cdot \frac{L(1,{\rm As}^+(\pi))}{\zeta(2)\zeta(4)}\prod_{v}\cZ_v^*(\varphi_{1,v},\varphi_{2,v},f_{1,v},f_{2,v}),
  \intertext{ where }
 \cZ_v^*(\varphi_{1,v},\varphi_{2,v},f_{1,v},f_{2,v}) =&\frac{\zeta_v(2)\zeta_v(4)}{L(1,{\rm As}^+(\pi_v))}
         \int_{H_1(\Q_v)} 
      {\mathcal B}_{\omega_v} (\omega_v(h_{1,v})\varphi_{1,v}, \varphi_{2,v})
      {\mathcal B}_{{\wtd\piHo}_v}({\wtd\piHo_v}(h_{1,v})f_{1,v}, f_{2,v}) \rmd h_{1,v}.
\end{align*}
\end{prop}
\begin{proof}This is a special case of the Rallis inner product formula proved in \cite[Proposition 11.2, Theorem 11.3]{gqt14} (\cf \cite[Lemma 7.11]{gi11}) for $H(V_r)=\Sp_4$ and $G(U_n)={\rm O}(V)$. 
Apply $n=m=4, r=2, \epsilon_0=-1$ in the notation \cite{gqt14}.
The non-vanishing of $L(1,{\rm As}^+(\pi))$ follows from Theorem \ref{T:AsNonvan}.  
The local integrals are absolutely convergent  by \cite[Lemma 7.7]{gi11}.
\end{proof}

Define $\pairing_\cL:\cL_\kappa(\C)\ot\cL_\kappa(\C)\to \C$ to be the pairing $\pairing_{2k_2}$ introduced in Section \ref{algrep}. 
For vector-valued Siegel cusp forms $F_1,F_2: G(\A)\to \cL_\kappa(\C)$, we define the Hermitian pairing
\begin{align*}
     \rpair{F_1}{F_2}_G
 =  &\int_{Z_G(\A)G(\Q)\bksl G(\A)} 
     \langle  F_1(g), \bar{ F_2(g) } \rangle_{\cL}\rmd g\\
     =&\int_{Z_G(\A)G(\Q)\bksl G(\A)} \bbpair{F_1(g)}{F_2(g)}\rmd g,
\end{align*}
where $\bbpairing: \cL_\kappa(\C)\ot\cL_\kappa(\C)\to \C$ is the $\SU_2(\R)$-invariant Hermitian pairing given by 
\beq\label{E:pairing}\bbpair{v_1}{v_2}:=\int_{\SU_2(\R)}\pair{\rho_\lam(u)v_1}{\ol{\rho_\lam(u)v_2}}_\cL\rmd^* u,\eeq
where $\rmd^* u$ is the Haar measure on $\SU_2(\R)$ with $\vol(\SU_2(\R))=1$. %We apply the above formula to compute $\rpair{\theta(\test,\bff^\dag)}{\theta(\test,\bff^\dag)}_G$.
Denote the pairing $\pairing_\cW\otimes\pairing_\cL$  on $\cW(\C)\ot\cL(\C)$ by $\pairing_{\cW\ot\cL}$.
To apply Rallis inner product formula to our case, we define local zeta integrals $\cI(\test_\infty)$ and $\cZ_\q(\test_\q,f^\dag_\q)$ by 
\begin{align}
\label{E:inftyloca}\cI(\test_\infty)=&   \int_{\bfX_\infty}  
        \langle \test_\infty(x),  \test_\infty(x) \rangle_{\cW\ot\cL}\, \rmd x,\\
  \label{E:finielocal} \cZ_\q(\test_\q,f^\dag_\q)=&\int_{H_1^0(\Q_\q)}\cB_{\om_\q}(\om_\q(h_\q)\test_\q,\test_\q)\cB_{\piHo_\q}(\piHo_\q(h_\q)f^\dag_\q,f^\dag_\q)\,\rmd h_\q.
\end{align}
\begin{prop}\label{RallisToInt}We have  \begin{align*} 
     \frac{ \rpair{\theta(\test,\bff^\dagger)}{ \theta(\test,\bff^\dagger)}_G}{ 
     \pair{\newform}{\newform}_{H^0} }   
= &   \frac{\vol(H_1^0(\R)) }{(\dim\cW_\ulk(\C))^2}\cdot\frac{L(1,{\rm As}^+(\pi))}{\zeta(2)\zeta(4)}\cdot \cI^*(\test_\infty)\cdot\prod_{\q<\infty} \cZ^*_\q(\test_\q,f^\dag_\q),
\end{align*}
where \[\cI^*(\test_\infty)=\frac{\zeta_\infty(2)\zeta_\infty(4)}{L(1,{\rm As}^+(\pi_\infty))}\cdot\cI(\test_\infty);\quad \cZ^*_\q(\test_\q,f^\dag_\q)=\frac{\zeta_p(2)\zeta_p(4)}{L(1,{\rm As}^+(\pi_p))}\cdot\cZ_\q(\test_\q,f^\dag_\q).\]
\end{prop}
\begin{proof}  We begin with some notation. Define the set 
\[\bfB:=\stt{(i,j)\mid 0\leq i\leq 2k_1,\,0\leq j\leq 2k_2}.\]
Let $\stt{\bfv_I}_{I\in\bfB}$ be the standard basis of $\cW_{\ul{k}}(\C)$ given by
\[\bfv_I:= X^i_1 Y^{2k_1-i}_1 \otimes X^j_2 Y^{2k_2-j}_2\text{ if }I=(i,j).\]
Recall the pairing $\pairing_\cW$ on $\cW_{\ul{k}}(\C)$ is introduced in Section \ref{subsec:fromsonH}. 
Then the corresponding dual basis $\stt{\bfv_I^*}_{I\in \bfB}$ 
with respect to  $\pairing_\cW$  is given by  $\bfv_I^*:=\bfv_{2\ul{k}-I}\cdot\binom{2k_1}{i}\binom{2k_2}{j}(-1)^{i+j}$. Write \begin{align*}
  \test_\infty(x) =& \sum^{2k_2}_{\al=0} \varphi_\infty^\al(x) \binom{2k_2}{\al} X^{\al}Y^{2k_2-\al}, \\
  \varphi_\infty^\al(x) =& \sum_{I\in\bfB}\varphi^\al_{I,\infty}(x) \bfv_I.
\end{align*}
Then $\varphi_{I,\infty}^\al(x)=\pair{\varphi_\infty^\al(x)}{\bfv_I^*}_\cW=(-1)^\al\cdot \pair{\test_\infty(x)}{\bfv_I^*\ot X^{2k_2-\al}Y^\al}_{\cW\ot\cL}$. Put  \begin{align*}\varphi_I^\al=&\varphi^\al_{I,\infty}\bigot\limits_{v<\infty}\test_v\in\cS(\bfX_\A).
%\bff_I=&\bff_{\bfv_I}\in\piHo\text{ and }\wtd\bff_I:=\wtd\bff_{\bfv_I}\in\wtd\piHo.
\end{align*}
For each $I\in\bfB$, put
\[\sF_{\bfv_I}=\prod_{\q\in\cP}  (1+\ep_{\frakp^c}\cdot\piHo_\q(\eta_\q)) \wtd \bff^\circ_{\bfv_I}\in\cA_{\wtd\piHo},\]
where $\wtd\bff^\circ_{\bfv_I}\in\cA_{\wtd\piHo}$ is defined as in \eqref{tildef}. Then one checks that
(i) $\wtd\sg(\bft_v) \sF_{\bfv_I}=\sF_{\bfv_I}$ for any finite place $v$ by \eqref{tildef} and (ii) $\sF_{\bfv_I}|_{H^0(\A)}=\bff^\dagger_{\bfv_I}(=\pair{\bff^\dagger}{\bfv_I}_\cL)$. From \eqref{E:H1H0} and the fact that $\om_v(\bft_v)\test_v=\test_v$ for every finite place $v$, we can deduce that 
\beq\label{E:1.I}\theta(\varphi^\al_I,\sF_{\bfv_I})=2^{-1}\theta(\varphi^\al_I,\sF_{\bfv_I}|_{H^0(\A)})=2^{-1}\theta(\varphi^\al_I,\bff^\dagger_{\bfv_I}).\eeq
It follows that \[\theta(\test,\bff^\dagger)=2 \cdot\sum_{\al=0}^{2k_2}\theta(\varphi_I^\al,\sF_{\bfv_I})X^\al Y^{2k_2-\al}\binom{2k_2}{\al},\]
and  hence, for the pairing $\langle\ , \ \rangle$ in \propref{P:GIRallis}, we have  
\beq\label{E:5.I}
\rpair{\theta(\test,\bff^\dagger)}{\theta(\test,\bff^\dagger)}_G=4\cdot \sum_{\al=0}^{2k_2}\sum_{I,J\in\bfB}\pair{\theta(\varphi^\al_I,\sF_{\bfv_I})}{\theta(\varphi^{2k_2-\al}_J,\sF_{\bfv_J})}(-1)^\al\binom{2k_2}{\al}.
\eeq

In the case $\infty\not\in\frakS$, 
 the local vector in $\wtd\sigma_\infty$ corresponding to  $\sF_{\bfv_I}$ is  $(\bfv_I, 0)$
by the fixed isomorphism $j:\otimes_v\wtd\sigma_v\iso\wtd\sigma$ given in \eqref{E:iso.I}. 
To emphasis this correspondence, we also write 
$\sF_{\bfv_I}$ as $\sF_{(\bfv_I, 0)}$ according to the notion in \eqref{E:iso.I}.

In the case $\infty\in\frakS$, we can decompose $\sF_{\bfv_I}=\sF_{\bfv_I}^++\sF_{\bfv_I}^-$, where
\[\sF_{\bfv_I}^\pm=\onehalf(\sF_{\bfv_I}\pm\wtd\piHo(\bft_\infty)\sF_{\bfv_{I^{\rm sw}}})\in\wtd\piHo^\pm.\]
Here $I^{\rm sw}=(j,i)$ for $I=(i,j)$. The global lift $\theta(\wtd\piHo^-)=0$ by \remref{R:1.I}, so we have \[\theta(\varphi^\al_I,\sF_{\bfv_I})=\theta(\varphi^\al_I,\sF_{\bfv_I}^+).\]
The fixed isomorphism $j:\otimes_v\wtd\sigma_v\iso\wtd\sigma$ given in \eqref{E:iso.I} shows that 
  the local vector in $\wtd\sigma_\infty$ corresponding $\sF^+_{\bfv_I}$ is given by $2^{-1}(\bfv_I, \bfv_{I^{\rm sw}})$. 

Given $0\leq \al,\beta\leq 2k_2$, we consider
\[\pair{\theta(\varphi^\al_I,\sF_{\bfv_I})}{\theta(\varphi^{\beta}_J,\sF_{\bfv_J})}
=\begin{cases}\pair{\theta(\varphi^\al_I,\sF_{(\bfv_I,0)})}{\theta(\varphi^{\beta}_J,\sF_{(\bfv_J,0)})}&\text{ if }\infty\not\in\frakS,\\
\pair{\theta(\varphi^\al_I,\sF^+_{\bfv_I})}{\theta(\varphi^{\beta}_J,\sF^+_{\bfv_J})}&\text{ if }\infty\in\frakS. \end{cases}
\]
By Rallis inner product formula (\propref{P:GIRallis}), we have   
\beq\label{E:6.I}
\pair{\theta(\varphi^\al_I,\sF_{\bfv_I})}{\theta(\varphi^{\beta}_J,\sF_{\bfv_J})}
=
\frac{\pair{\newform}{\newform}_{H^0}}{\dim\cW_\ulk(\C)}\cdot \frac{L(1,{\rm As}^+(\pi))}{\zeta(2)\zeta(4)}\cdot\left(\frac{\zeta_\infty(2)\zeta_\infty(4)}{L(1,{\rm As}^+(\pi_\infty))}\wtd\cZ_{I,J}\right) \prod_{\q<\infty}\wtd\cZ^*_\q,
\eeq
where $\wtd\cZ_{I,J}$ and $\wtd\cZ_\q$ are local zeta integrals defined by 
\beq\label{E:A1}\begin{aligned}
\wtd\cZ_{I,J}=&\begin{cases}\int_{H_1(\R)}\cB_{\om_\infty}(\om_\infty(h_{1,\infty})\varphi_I^\al,\varphi_J^\beta)\cB_{\wtd\piHo_\infty}(\wtd\piHo_\infty(h_{1,\infty})(\bfv_I,0),(\bfv_J,0))\rmd h_{1,\infty}&\text{ if }\infty\not\in\frakS,\\
&\\
4^{-1}\int_{H_1(\R)}\cB_{\om_\infty}(\om_\infty(h_{1,\infty})\varphi_I^\al,\varphi_J^\beta)
      \cB_{\piHo^+_\infty}(\piHo_\infty^+(h_{1,\infty})(\bfv_I, \bfv_{I^{\rm sw}}),(\bfv_J, \bfv_{J^{\rm sw}}))
      \rmd h_{1,\infty}&\text{ if }\infty\in\frakS,\end{cases}\\
\wtd\cZ^*_\q=&\frac{\zeta_\q(2)\zeta_\q(4)}{L(1,{\rm As}^+(\pi_\q))}\int_{H_1(\Q_\q)}\cB_{\om_\q}(\om(h_{1,\q})\test_\q,\test_\q)\cB_{{\wtd\piHo}_\q}({\wtd\piHo}_\q(h_{1,\q})\wtd f^\dagger_\q,\wtd f^\dagger_\q)\rmd h_{1,\q}\quad \text{ if }\q<\infty.
%&\wtd f_\q=\wtd f_\q^0\text{ if }v\not\in\cP,\quad \wtd f_\q=(f^\dagger_\q,f^\dagger_\q)\text{ if }q\in\cP.
\end{aligned}\eeq
For any finite place $\q$, we have
\begin{align*}
\omega_\q(\bft_\q) \test_\q  =\test_\q,   \quad
\wtd\piHo_\q ({\mathbf t}_\q)\wtd f^\dagger_\q = \wtd f^\dagger_\q,
\end{align*}
and hence the local zeta integral $\wtd\cZ_p$ equals 
\beq\label{E:9.I}
\begin{aligned}
& \onehalf\int_{H^0_1(\Q_\q)} \cB_{\omega_\q} (\omega_\q(h_\q)\test_\q, \test_\q)\cB_{\wtd\piHo_\q}( \wtd\piHo_\q(h_\q)\wtd f^\dag_\q,\wtd f^\dag_\q)+\cB_{\omega_\q} (\omega_\q(h_\q\bft_\q)\test_\q, \test_\q)\cdot\cB_{\wtd\piHo_\q}( \wtd\piHo_\q(h_\q\bft_\q)\wtd f^\dag_\q,\wtd f^\dag_\q)\rmd h_\q \\
=&\int_{H^0_1(\Q_\q)} 
        {\mathcal B}_{\omega_\q} (\omega_v(h_\q)\test_\q, \test_\q)
        {\mathcal B}_{\piHo_\q}( \piHo_v(h_\q)f^\dag_\q,f^\dag_\q)\rmd h_v =\cZ_\q. 
\end{aligned}\eeq
To compute the archimedean local zeta integral $\wtd\cZ_{I,J}$, we put
\[\cZ_{I,J}:=\int_{H_1^0(\R)}\cB_{\om_\infty}(\om_\infty(h_\infty)\varphi_I^\al,\varphi_J^\beta)\cB_{\piHo_\infty}(\piHo_\infty(h_\infty)\bfv_I,\bfv_J)\rmd h_\infty.\]
Assume that  $\infty\not\in\frakS$. 
By the definition of ${\cB}_{\wtd{\sigma}_\infty}$, we find that 
\begin{align*}
      {\cB}_{\wtd{\sigma}_\infty}(\wtd{\sigma}_\infty(h_\infty)({\bfv}_I, 0),  ({\bfv}_J, 0) )
      =& \frac{1}{2}  {\cB}_{\sigma_\infty}( \sigma_\infty(h_\infty){\bfv}_I,  {\bfv}_J ),  \\
      {\cB}_{\wtd{\sigma}_\infty}(\wtd{\sigma}_\infty(h_\infty {\bft}_\infty)({\bfv}_I, 0),  ({\bfv}_J, 0) )
      =&       {\cB}_{\wtd{\sigma}_\infty}(\wtd{\sigma}_\infty(h_\infty )(0, {\bfv}_I),  ({\bfv}_J, 0) )
     =0
\end{align*} 
for each $h_\infty\in H^0_1({\R})$.  
Since ${\rm vol}(\mu_2(\R),\rmd\epsilon_\infty)=1$, we obtain 
\begin{align*}
     \wtd{\mathcal Z}_{I,J} 
   =\frac{1}{2} 
         \int_{H^0_1({\mathbf R})} 
           {\mathcal B}_{\omega_\infty}(\omega_\infty(h_{\infty})\varphi^\alpha_I, \varphi^\beta_J ) 
           \frac{1}{2}  {\mathcal B}_{\sigma_\infty}( \sigma_\infty(h_\infty){\mathbf v}_I,  {\mathbf v}_J ) {\rm d}h_\infty   
   = 4^{-1} {\mathcal Z}_{I,J}.       
\end{align*}
If $\infty\in\frakS$, then $k_1=k_2$ and $\om_\infty(\bft_\infty)\varphi_I^\al=\varphi_{I^{\rm sw}}^\al$. 
By the definition of $\cB_{\sigma^+_\infty}$, we have 
\begin{align*}
    {\cB}_{\sigma^+_\infty}(({\bfv}_I, {\bfv}_{I^{\rm sw}}), ({\bfv}_J, {\bfv}_{J^{\rm sw}}))
  =\frac{1}{2}  
      \left\{   {\cB}_{\sigma_\infty}({\bfv}_I,{\bfv}_J) 
                  +  {\cB}_{\sigma_\infty}({\bfv}_{I^{\rm sw}},{\bfv}_{J^{\rm sw}}) 
       \right\}
   = {\cB}_{\sigma_\infty}({\bfv}_I,{\bfv}_J).  
\end{align*}
By using ${\rm vol}(\mu_2(\R),\rmd\epsilon_\infty)=1$ again, we obtain  
\begin{align*}
     \wtd{\cZ}_{I,J} 
  =& 8^{-1}\int_{H^0_1({\R})} 
        {\cB}_{\omega_\infty}(\omega_\infty(h_{\infty})\varphi^\al_I, \varphi^\beta_J ) 
        {\cB}_{\sigma^+_\infty} 
                        (\sigma^+_\infty(h_{\infty}) ({\bfv}_I, {\bfv}_{I^{\rm sw}}  ), 
                        ({\bfv}_J, {\bfv}_{J^{\rm sw}}  ) )                                                                     \\
     &\quad   +     {\cB}_{\omega_\infty}(\omega_\infty(h_{\infty}{\bft}_\infty)\varphi^\al_I, \varphi^\beta_J ) 
              {\cB}_{\sigma^+_\infty} 
                        (\sigma^+_\infty(h_{\infty}{\bft}_\infty) ({\bfv}_I, {\bfv}_{I^{\rm sw}}  ), 
                        ({\bfv}_J, {\bfv}_{J^{\rm sw}}  ) )                         
                        {\rm d}h_\infty   \\
  =& 8^{-1}\int_{H^0_1({\R})} 
        {\cB}_{\omega_\infty}(\omega_\infty(h_{\infty})\varphi^\al_I, \varphi^\beta_J ) 
        {\cB}_{\sigma^+_\infty} 
                        (\sigma^+_\infty(h_{\infty}) ({\bfv}_I, {\bfv}_{I^{\rm sw}}  ), 
                        ({\bfv}_J, {\bfv}_{J^{\rm sw}}  ) )                                                                     \\
     &\quad   +     {\cB}_{\omega_\infty}(\omega_\infty(h_{\infty})\varphi^\al_{I^{\rm sw}}, \varphi^\beta_J ) 
              {\cB}_{\sigma^+_\infty} 
                        (\sigma^+_\infty(h_{\infty} ) ({\bfv}_{I^{\rm sw}}, {\bfv}_I  ), 
                        ({\bfv}_J, {\bfv}_{J^{\rm sw}}  ) )                         
                        {\rm d}h_\infty   \\
  =& 8^{-1}\int_{H^0_1({\R})} 
        {\cB}_{\omega_\infty}(\omega_\infty(h_{\infty})\varphi^\al_I, \varphi^\beta_J ) 
        {\cB}_{\sigma_\infty} 
                        (\sigma_\infty(h_{\infty}) {\bfv}_I, 
                        {\bfv}_J )                                                                     \\
     &\quad   +     {\cB}_{\omega_\infty}(\omega_\infty(h_{\infty})\varphi^\al_{I^{\rm sw}}, \varphi^\beta_J ) 
              {\cB}_{\sigma_\infty} 
                        (\sigma_\infty(h_{\infty} ) {\bfv}_{I^{\rm sw}}, 
                        {\bfv}_J )                         
                        {\rm d}h_\infty   \\
   =& 8^{-1} ({\cZ}_{I,J} + {\cZ}_{I^{\rm sw}, J}). 
\end{align*}

To simply $\cZ_{I,J}$, we note that by the definition of $\test_\infty$ we have
\[\ol{\test_\infty(x)}=\tau_\ulk(\cJ)\test_\infty(x).\]
This implies that $\ol{\varphi^\al_{I,\infty}(x)}=(-1)^I\varphi_{2\ulk-I,\infty}^\al(x)$. We have
\begin{align*}
\cZ_{I,J}=&\int_{H_1^0(\R)}\int_{\bfX_\infty}\pair{\varphi^\al_\infty(x)}{\tau_{\ul{k}}(h_\infty)\bfv_I^*}_\cW\cdot \ol{\varphi^\beta_{J,\infty}(x)}\cdot\pair{\tau_{\ul{k}}(h_\infty)\bfv_I}{\bfv_{2\ulk-J}}_\cW\cdot(-1)^J\rmd x\rmd h_\infty\\
=&\frac{\pair{\bfv_I^*}{\bfv_I}_\ulk\vol(H^0_1(\R))}{\dim\cW_\ulk(\C)}\cdot \int_{\bfX_\infty}\pair{\varphi^\al_\infty(x)}{\bfv_{2\ulk-J}}_\cW\cdot \ol{\varphi^\beta_{J,\infty}(x)}(-1)^J\rmd x\\
=&\frac{\vol(H_1(\R))}{\dim\cW_\ulk(\C)}\cdot \int_{\bfX_\infty}\pair{\varphi^\al_\infty(x)}{\bfv_{2\ulk-J}}_\cW\cdot\varphi^\beta_{2\ulk-J, \infty}(x)\rmd x.
\end{align*}
In particular, $\cZ_{I,J}$ is independent of $I$. Therefore, we obtain
\[\wtd\cZ_{I,J}=4^{-1}\cZ_{I,J}=\frac{\vol(H_1(\R))}{4\dim\cW_\ulk(\C)}\cdot \int_{\bfX_\infty}\pair{\varphi^\al_\infty(x)}{\bfv_{2\ulk-J}}_\cW\cdot\varphi^\beta_{2\ulk-J,\infty}(x)\rmd x.
\]
\begin{align*}
\sum_{\al=0}^{2k_2}\sum_{I,J\in\bfB}\cZ_{I,J}\cdot(-1)^\al\binom{2k_2}{\al}
=&\sum_{\al=0}^{2k_2}\vol(H_1^0(\R))\cdot \int_{\bfX_\infty}\sum_{J\in\bfB}\pair{\varphi^\al_\infty(x)}{\bfv_{2\ulk-J}}_\cW\cdot \varphi^{2k_2-\al}_{2\ulk-J,\infty}(x)\cdot(-1)^\al\binom{2k_2}{\al}\rmd x\\
=&\vol(H_1^0(\R))\cdot\int_{\bfX_\infty}\sum_{\al=0}^{2k_2}\pair{\varphi^\al_\infty(x)}{\varphi^{2k_2-\al}_\infty(x)}_\cW\cdot(-1)^\al\binom{2k_2}{\al}\rmd x\\
=&\vol(H_1^0(\R))\cdot\int_{\bfX_\infty}\pair{\test_\infty(x)}{\test_\infty(x)}_{\cW\ot\cL}\rmd x
\end{align*}
Combined with \eqref{E:5.I}, \eqref{E:6.I} and \eqref{E:9.I}, the above equation yields the proposition.
%we have  \begin{align*}
%\rpair{\theta(\test,\bff)}{\theta(\test,\bff)}_G&=\frac{\vol(H_1^0(\R))\pair{\newform}{\newform}_{H^0}}{(\dim\cW_\ulk(\C))^2}\cdot\frac{L^S(1,{\rm As}^+(\pi))}{\zeta^S(2)\zeta^S(4)}\cdot\int_{\bfX_\infty}\pair{\test_\infty(x)}{\test_\infty(x)}_{\cW\ot\cL}\rmd x\cdot \prod_{q\divides N_F}\cZ_\q.\end{align*}
\end{proof}

The explicit calculations of local integrals $\cI(\test_\infty)$ and $\cZ^0_p$ will be postponed to the next section.  
\begin{cor}\label{C:formulaa}
Assume that $ \Delta_F$ and $N^+N^-$ are coprime and that $\pi$ is not Galois self-dual. For $p\divides \frakN$, put \[\ep_p=\begin{cases}
\ep_{\frakp}\ep_{\frakp^c}&\text{ if }p=\frakp\frakp^c\text{ is split in $F$},\\
\ep_\frakp&\text{ if }p=\frakp\text{ is inert in $F$}. \end{cases}\]Then we have
\begin{align*}
  \frac{ \rpair{\theta(\test, {\mathbf f}^\dag)}{\theta(\test, {\mathbf f}^\dag)}_G}{ \pair{\bff^\dag}{\bff^\dag}_{H^0} }  
= &      \frac{ L(1, {\rm As}^+(\pi) )}{ \zeta(2) \zeta(4)} 
\cdot  \frac{(-1)^{k_2}\vol(\cU,\rmd h)2^{\#\cP}}{2^{2k_1+7}(2k_1+1)(2k_2+1)^2N^2\Delta_F^3}\\
&\times\frac{ \zeta_{N_F}(4) }{ \zeta_{N_F}(1) }\cdot        \prod_{p\divides\frakN}(1+\ep_p)\cdot \prod_{  p\mid \Delta_F  } (1+p^{-1}).
\end{align*}
\end{cor}
\begin{proof}
Recall that  $L(s,{\rm As}^+(\pi_\infty))=\Gamma_\C(s+k_1+k_2+1) \cdot  \Gamma_\C(s+k_1-k_2)$. By \propref{locinfty}, we have
\[\cI(\test_\infty)=\frac{(-1)^{k_2}(2k_1+1)}{2^{2k_1+7}}\cdot\frac{L(1,{\rm As}^+(\pi_\infty))}{\zeta_\infty(2)\zeta_\infty(4)}.\]
On the other hand, by the formulas of the local zeta integrals $\cZ^0_p(\test_p,f^\dag_p)$ in \propref{P:splita.loc}, \ref{P:splitb.loc}, \ref{P:inert.loc}, and \ref{P:ramified.loc}, we find that\begin{align*}
 \prod_{p}\cZ^*_p(\test_p,f_p^\dag)=  &  \frac{L(1,{\rm As}^+(\pi))}{ \zeta(2) \zeta^{N_F}(4) \zeta_{N_F}(1)} \cdot \frac{\vol(\cU,\rmd h)}{N_F^2\Delta_F}\cdot
 \prod_{p\in \cP}
                 2\cB_{\piHo_p}(f_p^\dag,f_p^\dag)
   \cdot \prod_{p\divides\frakN }   
                 (1+ \ep_p)\cdot\prod_{ p\divides \Delta_F }  
                     (1+p^{-1}).  
\end{align*}
The corollary follows from \propref{RallisToInt} and \lmref{L:2.I} (1).
\end{proof}

\subsection{The Petersson norm of classical Yoshida lifts}\label{S:PetYL}
Note that the pairing $\rpairing_G$ may not be positive definite unless $k_2=0$. In this subsection, we introduce a positive definite Hermitian pairing on the space of Siegel modular forms and rephrase \corref{C:formulaa} in terms of classical Siegel cusp forms of genus two. Define the Hermitian pairing $\cB_\cL:\cL_\kappa(\C)\ot\cL_\kappa(\C)\to\C$ by 
\beq\label{E:pairBL}\cB_\cL(v_1,v_2):=\pair{v_1}{\rho_\kappa(w_0)\ol{v_2}}_\cL,\quad w_0=\pMX{0}{1}{-1}{0}.\eeq
Then it is easy to see that $\cB_\cL$ is an $\SU_2(\R)$-invariant and positive definite Hermitian pairing. 
\begin{lm}\label{L:5.I}For $v_1,v_2\in\cL_\kappa(\C)$, we have \[\bbpair{v_1}{v_2}=\frac{(-1)^{k_2}}{2k_2+1}\cdot \cB_\cL(v_1,v_2).\]
\end{lm}
\begin{proof}
%Recall that $U(2)(\R)\hookto \GSp_4(\R),\, A+iB\mapsto \pMX{A}{B}{-B}{A}$. Let $\theta(g)=\theta(g;\varphi,\bff)$. Then 
%\begin{align*}
%\int_{[\GSp_4]}\pair{\theta(g)}{\ol{\theta(g)}}\rmd g&=\int_{[\GSp_4]}\int_{u\in \SU(2)(\R)}\pair{\theta(gu)}{\ol{\theta(gu)}}\rmd u\rmd g\\
%=&\int_{[\GSp_4]}\int_{u\in \SU(2)(\R)}\pair{\rho_\lam(u)\theta(g)}{\ol{\rho_\lam(u)\theta(g)}}\rmd u\rmd g\\
%=&\int_{[\GSp_4]}\bbpair{\theta(g)}{\theta(g)}\rmd g\\
%=&C\cdot\vol(\SU(2)(\R),\rmd u))\int_{[\GSp_4]}\pair{\theta(g)}{\rho_\lam(w_0)\ol{\theta(g)}}_\cL\rmd g
%\end{align*}
Since $\bbpairing$ and $\cB_\cL$ are both $\SU_2(\R)$-invariant Hermitian pairing 
   and $\cL_\kappa(\C)$ is irreducible, we have \[\bbpair{v_1}{v_2}=C\cdot\cB_\cL(v_1,v_2)=C\cdot \pair{v_1}{\rho_\kappa(w_0)\ol{v_2}}_\cL\]
for some constant $C$. Letting $v_1=v_2=X^{2k_2}$, we have \begin{align*}C=&\int_{\SU_2(\R)}\pair{\rho_\kappa(u)X^{2k_2}}{\rho_\kappa(\ol{u})X^{2k_2}}_\cL\rmd^* u\\
=&\int_{\SU_2(\R)}\sum_{a}(-1)^a\binom{2k_2}{a}\al^a\ol{\al}^{2k_2-a}\beta^a\ol{\beta}^{2k_2-a}\rmd^* u,\,u=\pMX{\al}{\beta}{-\ol{\beta}}{\ol{\al}}\in\SU_2(\R).
\end{align*}
For $u\in  \SU_2(\R)$, we introduce the coordinates $u=u(\psi,\theta,\varphi)$: 
\[u=\pMX{\al}{\beta}{-\ol{\beta}}{\ol{\al}},\,\al=\cos\psi \cdot e^{\sqrt{-1}\theta},\,\beta=\sin\psi\cdot  e^{\sqrt{-1}\varphi},\,0\leq \theta,\varphi\leq 2\pi,0\leq \psi\leq \pi/2. \]
Then the Haar measure $\rmd^* u$ is given by 
\[\rmd^* u=(4\pi)^{-2}\sin 2\psi \rmd \psi\rmd \theta\rmd \varphi.\]
We thus find that
\begin{align*}C=&(-1)^{k_2}\binom{2k_2}{k_2}2^{-2k_2}\int_0^{\pi/2}(\sin 2\psi)^{2k_2+1}\rmd \psi\\
%=&(-1)^{k_2}\binom{2k_2}{k_2}2^{-2k_2}\int_0^{\pi/2}(\sin \psi)^{2k_2+1}\rmd \psi\\
=&(-1)^{k_2}\binom{2k_2}{k_2}2^{-2k_2}\cdot 2^{2k_2}\frac{(k_2!)^2}{(2k_2+1)!}=\frac{(-1)^{k_2}}{2k_2+1}.\qedhere
\end{align*}
\end{proof}

%To a Siegel modular form $F:{\rm GSp}_4({\mathbf A}) \to {\mathcal L}_\kappa(\C)$ of genus 2,  we associate a classical Siegel modular form $F^\ast: {\mathcal H}_2 \times {\rm GSp}_4({\mathbf A}_f) \to {\mathcal L}_\kappa(\C)$ by
%\begin{align*}
 %F^\ast(Z, g_f) = \rho_\kappa(J(g_\infty, {\mathbf i})) F(g_\infty g_f), \quad (g_\infty\in {\rm Sp}_4(\R), g_\infty\cdot{\mathbf i} = Z).   
%\end{align*}
Define the classical normalized Yoshida lift $\theta^*_{\bff^\dag}: {\frakH}_2  \to \cL_\kappa(\C)$ by \begin{align*}
  \theta^*_{\bff^\dag}(Z)=&\frac{1}{\vol(\cU,\rmd h)} \rho_\kappa(J(g_\infty,\bfi))\theta(\test,\bff^\dag)(g_\infty)\\
  &\quad (g_\infty\in\Sp_4(\R), g_\infty\cdot \bfi=Z).
\end{align*}
Applying the proof of \cite[Proposition 3.6]{hn15} verbatim, one can show that $\theta_{\bff^\dag}^*$ is a holomorphic vector-valued Siegel modular from of weight $\Sym^{2k_2}(\C)\ot\det^{k_1-k_2+2}$ and level $\Gamma_0^{(2)}(N_F)$ and has $\ell$-adic integral Fourier coefficients if ${\mathbf f}$ is normalized so that the values of $\bff$ on $\wh D^\x$ are all $\ell$-adically integral. 

Define the Petersson norm of $\theta^*_{\bff^\dag}$ by 
\[
     \pair{\theta^*_{\bff^\dag}}{\theta^*_{\bff^\dag}}_{\frakH_2} 
  =  \int_{\Gamma_0^{(2)}(N_F) \backslash \frakH_2} \cB_\cL(\theta^*_{\bff^\dag}(Z),\theta^*_{\bff^\dag}(Z))  (\det Y)^{k_1+2} \frac{\rmd X\rmd Y}{(\det Y)^{3} }, 
\]
where $Z=X+\sqrt{-1}Y\in {\frakH}_2$ and $\rmd X = \prod_{j\leq l}\rmd x_{jl}, \rmd Y=\prod_{j\leq l} \rmd y_{jl}$ for $X=(x_{jl})$ and $Y=(y_{jl})$. Recall that $R$ is the Eichler order of level $N^+\cO_F$ contained in $R^0$. For $\bff_1,\bff_2\in \cM_\ulk(D_\A^\x,N^+\cO_F^\x)$, put\beq\label{E:fnorm}\begin{aligned}\pair{\bff_1}{\bff_2}_{R}:=&\frac{1}{\vol(\bar U_R,\rmd h_0)}\pair{\bff_1}{\bff_2}_{H^0}\\=&\sum_{[a]\in D^\x\bksl \wh D^\x/\wh R^\x}\pair{\bff_1(a)}{\ol{\tau_\ulk(\cJ)\bff_2(a)}}_\cW\cdot\frac{1}{\#\Gamma_a},\end{aligned}\eeq
where $\bar U_R$ is the image of $U_R$ in $H^0(\A)/Z_{H^0}(\A)$ and $\Gamma_a=(a \wh R^\x a^{-1}\cap D^\x)/\stt{\pm 1}$.
\begin{thm}\label{T:ClaInnPrd}
Let $r_F$ be the number of primes ramified in $F$. Put \[r_{F,2}=\begin{cases}1&\text{ if }2\divides \Delta_F,\\
0&\text{ if } 2\ndivides \Delta_F.\end{cases}\] We have 
\begin{align*}
    \frac{\pair{\theta^*_{\bff^\dag}}{\theta^*_{\bff^\dag}}_{\frakH_2}}{\pair{\bff^\dag}{\bff^\dag}_{R}}  
=&  \frac{ 2^\beta  N}{(2k_1+1)(2k_2+1)}\cdot L(1, {\rm As}^+(\pi) )\cdot \prod_{p\divides \frakN}(1+\ep_p)\cdot \prod_{  p\mid \Delta_F  } (1+p^{-1}),
\end{align*}
where\[\beta=\#\cP+4r_{F,2}-2k_1-7-r_F \]
and $\cP$ is the finite set defined in \eqref{E:badprime}.
\end{thm}
\begin{proof}We recall some facts:
\begin{itemize}
\item the Tamagwa number $\tau({\rm PGSp}_4)=\tau(\SO(3,2))=2$, \item $\vol(\Sp_4(\Z)\backslash {\frakH}_2,\frac{\rmd X\rmd Y}{(\det Y)^{3}}) = 2\zeta(2)\zeta(4)$ (\cite[Theorem 11]{si43}),
 \item $ [{\rm Sp}_4(\Z):\Gamma_0(N_F)] = N_F^3 \prod_{p\mid N_F}\frac{1-p^{-4}}{1-p^{-1}}$ (\cite[p114, (1)]{kl59}).\end{itemize}
The above combined with \lmref{L:5.I} yield 
\begin{align}\label{GAdtoCla}
  \frac{1}{\vol(\cU,\rmd h)^2}\cdot \rpair{\theta(\test,\bff^\dag)}{\theta(\test,\bff^\dag)}_G = \frac{(-1)^{k_2}}{2k_2+1}\cdot N_F^{-3} \prod_{p\mid N_F}\frac{1-p^{-1}}{1-p^{-4}} 
                        \cdot \frac{\pair{\theta^*_{\bff^\dag}}{\theta^*_{\bff^\dag}}_{\frakH_2}}{\zeta(2)\zeta(4)}.
\end{align}
Then we have
\begin{align*}
\pair{\bff^\dag}{\bff^\dag}_{H^0}&=\pair{\bff^\dag}{\bff^\dag}_{R}\cdot \vol(\bar U_R,\rmd h_0),
\end{align*}
so by \corref{C:formulaa}, we find that 
\begin{align*}
\frac{\pair{\theta_{\bff^\dag}^*}{\theta_{\bff^\dag}^*}_{\frakH_2}}{\pair{\bff^\dag}{\bff^\dag}_{R}}\cdot N_F^{-3}&=\frac{\vol(\bar U_R,\rmd h_0)\cdot 2^{\#\cP}\cdot L(1,{\rm As}^+(\pi))}{\vol(\cU,\rmd h)2^{2k_1+7}(2k_1+1)(2k_2+1)N_F^2 \Delta_F2^{-4r_{F,2}}}\cdot\prod_{p\divides \frakN}(1+\ep_p)\cdot\prod_{p\divides \Delta_F}(1+p^{-1}).
\end{align*}
Therefore, it remains to show that 
\[\vol(\bar U_R,\rmd h_0)=2^{-r_F}\cdot \vol(\cU,\rmd h)\]
for Tamagawa measures $\rmd h_0$ and $\rmd h$.

Following \cite[\S 8, p.279]{gi11}, let $\om_{H^0}$be a rational invariant differential top form on $H^0/Z_{H^0}$ and $\om_{H^0_1}$ be the pull-back of $\om_{H^0}$ by the natural isogeny $H^0_1\to H^0/Z_{H^0}$. For each place $v\in\Sigma_\Q$, let $\rmd^t h_{0,v}$ and $\rmd^t h_v$ be the measures on $H^0(\Q_v)$ and $H^0_1(\Q_v)$ induced by $\om_{H^0}$ and $\om_{H^0_1}$. Then $\rmd h_0=\prod_v\rmd^t h_{0,v}$ and $\rmd h=\prod_v \rmd^t h_v$ are Tamagawa measures on $H^0$ and $H^0_1$. We have
\[\frac{\vol(\bar U_{R_v},\rmd h^t_{0,v})}{\vol(\cU_v,\rmd h^t_v)}=2^{-1}[\rmN_{F/\Q}({\rm n}(R_v^\x)):(\Z_v^\x)^2]=\begin{cases}
1/2&\text{ if }v\ndivides 2,\,v\divides \infty \Delta_F,\\
2&\text{ if }v=2\ndivides \Delta_F,\\
1&\text{ otherwise.}
\end{cases}\]
 We see that $\vol(\bar U_R,\rmd h_0)=\vol(\cU,\rmd h)\cdot 2^{-r_F}$. 
\end{proof}
% !TEX root = YPNorm.tex
\section{The calculations of the local integrals}\label{LocComp}
\subsection{The local integral at the infinite place}
In this subsection, we evaluate the integral $\cI(\test_\infty)$ in \eqref{E:inftyloca}. Recall that if we define $P_\ulk:\bbH^{\oplus 2}\to\cW_\ulk(\C)\ot\cL_\kappa(\C)$ by
\[P_\ulk(x_1,x_2)=\sum_{\al=0}^{2k_2}P^\al_\ulk(x_1,x_2)\binom{2k_2}{\al}X^\al Y^{2k_2-\al},\]
where $P^\al_\ulk$ is the polynomial introduced in \eqref{E:vpalpha.def}, then 
\[\test_\infty(x_1,x_2)=e^{-2\pi({\rm n}(x_1)+{\rm n}(x_2))}\cdot P_\ulk(x_1,x_2).\]
\begin{lm}\label{L:1.loc} We have
\begin{align*}
        \int_{ {\rm SU}_2(\R)^2} 
          \langle P_\ulk(u),  P_\ulk(u)  \rangle_{\cW\ot\cL}\,   \rmd^* u
 =  &(-1)^{k_2} (2k_1+1)\cdot\frac{\Gamma(k_1+k_2+2)\Gamma(k_1-k_2+1)}{\Gamma(k_1+2)^2},
\end{align*}
where $\rmd^*u$ is the Haar measure on $\SU_2(\R)^2$ with $\vol(\SU_2(\R)^2)=1$.
\end{lm}
\begin{proof}Set $\Psi(u):=\pair{P_\ulk(u)}{P_\ulk(u)}_{\cW\ot\cL}$. We have
\[    \int_{ {\rm SU}_2(\R)^2}  \Psi(u)\rmd^* u 
    =\int_{\SU_2(\R)}\int_{\SU_2(\R)}\Psi(u_1,u_2) \rmd u_1\rmd u_2,\]
 where $\rmd u_1,\rmd u_2$ are the Haar measure on $\SU_2(\R)$ with $\vol(\SU_2(\R))=1$. 
 By \cite[Lemma 3.2]{hn15}, we find that 
\[\Psi(u_1,u_2)=\Psi(u_2^{-1}u_1,\bfone_2),\quad\Psi(u_1,\bfone_2)=\Psi(u_2^{-1}u_1u_2,\bfone_2).\]
It follows that \begin{align*}
 \int_{ {\rm SU}_2(\R)^2}  \Psi(u)\rmd^* u = & \int_{\SU_2(\R)}\int_{\SU_2(\R)} 
       \Psi(u_2^{-1}u_1,\bfone_2)\rmd u_1\rmd u_2  \\
 = &  \int_{ {\rm SU}_2(\R)}
           \Psi(u_1,1)\rmd u_1.
\end{align*}
Moreover, the function $u_1\mapsto \Psi(u_1,\bfone_2)$ is a class function on $\SU_2(\R)$, so by Weyl's integral formula, we obtain
\begin{align*}
 \int_{ {\rm SU}_2(\R)^2}  \Psi(u)\rmd^* u
=  \frac{1}{4\pi} \int^{2\pi}_0 
         \abs{e^{\sqrt{-1}\theta}-e^{-\sqrt{-1}\theta}}^2  \Psi(\pDII{e^{\sqrt{-1}\theta}}{e^{-\sqrt{-1}\theta}},\bfone_2)\rmd \theta  .
\end{align*}
By definition \eqref{E:vpalpha.def}, we have\begin{align*}
    &P_\ulk^\al(\begin{pmatrix} e^{\sqrt{-1}\theta}  & \\ & e^{-\sqrt{-1}\theta}  \end{pmatrix}, \bfone_2 )  \\  \displaybreak[0]
 = & ( (e^{\sqrt{-1}\theta}  - e^{-\sqrt{-1}\theta }) X_1Y_1  )^{k_1-k_2}\cdot ( e^{-\sqrt{-1}\theta} Y_1\otimes X_2  - e^{\sqrt{-1}\theta} X_1\otimes Y_2 )^\alpha 
                                                 (  Y_1\otimes X_2  -  X_1\otimes Y_2 )^{2k_2-\alpha}  \\  \displaybreak[0]
 = & (e^{\sqrt{-1}\theta}  - e^{-\sqrt{-1}\theta})^{k_1-k_2}  (X_1Y_1  )^{k_1-k_2}  
\cdot\sum^\alpha_{a=0} \binom{\alpha}{a} (-1)^a e^{\sqrt{-1}\theta(a-\alpha+a)}    X^a_1Y^{\alpha-a}_1 \otimes X^{\alpha-a}_2Y^a_2   \\  \displaybreak[0]
    & \qquad \qquad \times   \sum^{2k_2-\alpha}_{b=0} \binom{2k_2-\alpha}{b} (-1)^b    X^b_1Y^{2k_2-\alpha-b}_1 \otimes X^{2k_2-\alpha-b}_2Y^b_2   \\  \displaybreak[0]
 = &  (e^{\sqrt{-1}\theta}  - e^{-\sqrt{-1}\theta})^{k_1-k_2}     \sum^\alpha_{a=0}  \sum^{2k_2-\alpha}_{b=0} 
             \binom{\alpha}{a}  \binom{2k_2-\alpha}{b} (-1)^{a+b}
       e^{\sqrt{-1}\theta(2a-\alpha)} X^{k_1-k_2+a+b}_1Y^{k_1+k_2-(a+b)}_1 \otimes X^{2k_2-(a+b)}_2Y^{a+b}_2.
 \end{align*}
From the above equation, we see that  \begin{align*}
 &  \Psi(\pDII{e^{\sqrt{-1}\theta}}{e^{-\sqrt{-1}\theta}},\bfone_2)\\
   = & (e^{\sqrt{-1}\theta}  - e^{-\sqrt{-1}\theta})^{2k_1-2k_2}     
     \sum^{2k_2}_{\alpha=0}  (-1)^\alpha \binom{2k_2}{\alpha} 
\sum^\alpha_{a=0}  \sum^{2k_2-\alpha}_{b=0}  \binom{\alpha}{a}  \binom{2k_2-\alpha}{b} (-1)^{a+b}  
                e^{\sqrt{-1}\theta(2a-\alpha) }    \\
     & \times \sum_{c+d=2k_2-a-b}   \binom{2k_2-\alpha}{c}  \binom{\alpha}{d} (-1)^{c+d}
                e^{\sqrt{-1}\theta(2c-(2k_2-\alpha)) }\cdot  \frac{(-1)^{k_1-k_2+a+b}}{\binom{2k_1}{k_1-k_2+a+b} }   \frac{(-1)^{a+b}}{\binom{2k_2}{a+b} }             \\  
 = &(-1)^{k_1-k_2}(e^{\sqrt{-1}\theta}  - e^{-\sqrt{-1}\theta})^{2k_1-2k_2}     
     \times \sum^{2k_2}_{\alpha=0}  (-1)^\alpha \binom{2k_2}{\alpha}   \\  \displaybreak[0]
   & \times \sum_{a,b,c}    \binom{\alpha}{a}  \binom{2k_2-\alpha}{b}  \binom{2k_2-\alpha}{c}  \binom{\alpha}{2k_2-a-b-c}
                                    \binom{2k_1}{k_1-k_2+a+b}^{-1}  \binom{2k_2}{a+b}^{-1}   
      \cdot e^{\sqrt{-1}\theta(2a+2c-2k_2) } .
\end{align*}
By the formula
\begin{align*}
    \int^{2\pi}_0\abs{e^{\sqrt{-1}\theta}-e^{-\sqrt{-1}\theta}}^2
       (e^{\sqrt{-1}\theta}  - e^{-\sqrt{-1}\theta})^K e^{ \sqrt{-1}\theta A}  \rmd \theta
=  & 2\pi (-1)^\frac{K+A}{2} \binom{K+2}{\frac{K+A+2}{2}} 
\end{align*}
for even integers $K$ and $A$, we find that
\begin{align*}
&    \frac{1}{2\pi} \int^{2\pi}_0 
         \abs{e^{\sqrt{-1}\theta}-e^{-\sqrt{-1}\theta}}^2\Psi(\pDII{e^{\sqrt{-1}\theta}}{e^{-\sqrt{-1}\theta}},\bfone_2)\rmd \theta\\
         = &(-1)^{k_1-k_2}      
     \sum^{2k_2}_{\alpha=0}  (-1)^\alpha \binom{2k_2}{\alpha}  \sum_{a,b,c}    \binom{\alpha}{a}  \binom{2k_2-\alpha}{b}  \binom{2k_2-\alpha}{c}  \binom{\alpha}{2k_2-a-b-c}\\
     &\times
  \binom{2k_1}{k_1-k_2+a+b}^{-1}  \binom{2k_2}{a+b}^{-1}  (-1)^{k_1+a+c} \binom{2k_1-2k_2+2}{k_1-2k_2+a+c+1}\\
%= &(-1)^{k_2} e^{-8\pi}      
   %  \times \sum^{2k_2}_{\alpha=0}  (-1)^\alpha \binom{2k_2}{\alpha}   \\  \displaybreak[0]
  % & \times \sum_{a,b,c} (-1)^{a+c}   \binom{\alpha}{a}  \binom{2k_2-\alpha}{b}  \binom{2k_2-\alpha}{c}  \binom{\alpha}{2k_2-a-b-c}  \\  \displaybreak[0]
 %  & \times  \binom{2k_1}{k_1-k_2+a+b}^{-1}  \binom{2k_2}{a+b}^{-1}  \binom{2k_1-2k_2+2}{k_1-2k_2+a+c+1}. 
%\end{align*}We note that  
%\begin{align*}
 % & \sum_{a,b,c} (-1)^{a+c}   \binom{\alpha}{a}  \binom{2k_2-\alpha}{b}  \binom{2k_2-\alpha}{c}  \binom{\alpha}{2k_2-a-b-c}  \\  \displaybreak[0]
  % & \times  \binom{2k_1}{k_1-k_2+a+b}^{-1}  \binom{2k_2}{a+b}^{-1}  \binom{2k_1-2k_2+2}{k_1-2k_2+a+c+1}   \\  \displaybreak[0]
\stackrel{b\mapsto b-a+k_2}{=} &(-1)^{k_1-k_2}\sum^{2k_2}_{\alpha=0}  (-1)^\alpha \binom{2k_2}{\alpha} \sum_{b=-k_2}^{k_2}\sum_{a,c} (-1)^{a+c}   \binom{\alpha}{a}  \binom{2k_2-\alpha}{k_2+b-a}  \binom{2k_2-\alpha}{c}  \binom{\alpha}{k_2-b-c}  \\  \displaybreak[0]
 %  & \times  \binom{2k_1}{k_1+b}^{-1}  \binom{2k_2}{k_2+b}^{-1}  \binom{2k_1-2k_2+2}{k_1-2k_2+a+c+1}   \\  \displaybreak[0]
%\stackrel{c\mapsto 2k_2-c}{=}
%   & \sum_{a,b,c} (-1)^{a+c}   \binom{\alpha}{a}  \binom{2k_2-\alpha}{k_2+b-a}  \binom{2k_2-\alpha}{2k_2-c}  \binom{\alpha}{-k_2-b+c}  \\  \displaybreak[0]
   %& \times  \binom{2k_1}{k_1+b}^{-1}  \binom{2k_2}{k_2+b}^{-1}  \binom{2k_1-2k_2+2}{k_1+a-c+1}   \\  \displaybreak[0]
%= & \sum_{a,b,c} (-1)^{a+c}   \binom{\alpha}{a}  \binom{2k_2-\alpha}{k_2-\alpha-b+a}  \binom{2k_2-\alpha}{c-\alpha}  \binom{\alpha}{\alpha+k_2+b-c}  \\  \displaybreak[0]
  % & \times  \binom{2k_1}{k_1+b}^{-1}  \binom{2k_2}{k_2+b}^{-1}  \binom{2k_1-2k_2+2}{k_1+a-c+1}   \\  \displaybreak[0]
%\stackrel{ \substack{ a\mapsto -a+\alpha  \\ c\mapsto c+\alpha  } }{=} 
%& \sum_{a,b,c} (-1)^{a+c}   \binom{\alpha}{a}  \binom{2k_2-\alpha}{k_2-b-a}  \binom{2k_2-\alpha}{c}  \binom{\alpha}{k_2+b-c}  \\  \displaybreak[0]
   & \times  \binom{2k_1}{k_1+b}^{-1}  \binom{2k_2}{k_2+b}^{-1}  \binom{2k_1-2k_2+2}{k_1-2k_2+a+c+1}\\
   =&(-1)^{k_2}\sum_{b=-k_2}^{k_2} \binom{2k_1}{k_1+b}^{-1}  \binom{2k_2}{k_2+b}^{-1} \sum^{2k_2}_{\alpha=0} (-1)^\alpha\binom{2k_2}{\alpha}T^\al_b,
\end{align*}
where 
\[  T^\alpha_b =\sum_{a,c} (-1)^{a+c} \binom{\alpha}{a} \binom{2k_2-\alpha}{k_2+b-a}\binom{2k_2-\alpha}{c}\binom{\alpha}{k_2-b-c}\binom{2k_1-2k_2+2}{k_1-2k_2+a+c+1}.\]
%\begin{lem}
%{\itshape We have 
%\begin{align*}
   % {\mathpzc B}(k_1,k_2) 
 %:=   & \sum^{2k_2}_{\alpha=0}  (-1)^\alpha \binom{2k_2}{\alpha}   
    %   \times  \sum_{a,b,c} (-1)^{a+c}   \binom{\alpha}{a}  \binom{2k_2-\alpha}{k_2-b-a}  \binom{2k_2-\alpha}{c}  \binom{\alpha}{k_2+b-c}  \\
    %& \times  \binom{2k_1}{k_1+b}^{-1}  \binom{2k_2}{k_2+b}^{-1}  \binom{2k_1-2k_2+2}{k_1-a-c+1}   \\
 %= & 2 (2k_1+1)  \times  \frac{ 1 }{k_1+1}  \binom{k_1+k_2+1}{k_2} \binom{k_1}{k_2}^{-1}.  
%\end{align*}}
%\end{lem}
%\begin{proof}
%For each $-k_2\leq b \leq k_2$, put
%\begin{align*}
  %T^\alpha(b) =& \sum_{a,c} (-1)^{a+c} \binom{\alpha}{a} \binom{2k_2-\alpha}{k_2-b-a}\binom{2k_2-\alpha}{c}\binom{\alpha}{k_2+b-c}\binom{2k_1-2k_2+2}{k_1-a-c+1},  \\
 % S^\alpha(k_1,k_2) =& (-1)^\alpha\binom{2k_2}{\alpha} \sum^{k_2}_{b=-k_2}\binom{2k_2}{k_1+b}^{-1} \binom{2k_2}{k_2+b}^{-1} \cdot T^\alpha(b).
%\end{align*}
%Then, it suffices to prove that
%\begin{align}
  %{\mathpzc B}(k_1, k_2) = \binom{2k_2+2}{k_1+1} \cdot \frac{2k_1+1}{k_1+1} \cdot \binom{2k_1+1}{k_1-k_2}^{-1}.  \label{E:arch}
%\end{align}
Note that $T^\alpha_b$ is equal to the coefficient  of $X^{k_1+1}Y^{k_2-b}Z^{k_2+b}$ of the following polynomial \begin{align*}
  F^\alpha(X,Y,Z) :=(1-Z)^\alpha(1+XZ)^{2k_2-\alpha}   (1-Y)^{2k_2-\alpha}(1+XY)^\alpha  (1+X)^{2k_1-2k_2+2},
\end{align*}
so we obtain the identity 
\[\sum^{2k_2}_{\alpha=0} (-1)^\alpha\binom{2k_2}{\alpha}T^\al_b=\binom{2k_1+2}{k_1+1}\binom{2k_2}{k_1+b}\]
by looking at the coefficient of $X^{k_1+1}Y^{k_2-b}Z^{k_2+b}$ of the polynomial \begin{align*}
  &\sum^{2k_2}_{\alpha=0} (-1)^\alpha \binom{2k_2}{\alpha} F^\alpha(X,Y,Z)  \\
 =& (1+X)^{2k_1-2k_2+2} \sum^{2k_2}_{\alpha=1} (-1)^\alpha\binom{2k_2}{\alpha} (1+XY-Z-ZXY)^\alpha (1+XZ-Y-XYZ)^{2k_2-\alpha}  \\  \displaybreak[0]
 =& (1+X)^{2k_1+2} (Y-Z)^{2k_2}.
\end{align*}
Summarizing the above calculations, we obtain 
\beq\label{E:7.I}\begin{aligned}\int_{ {\rm SU}_2(\R)^2}  \Psi(u)\rmd^* u =&(-1)^{k_2} 2^{-1}\sum_{b=-k_2}^{k_2} \binom{2k_1}{k_1+b}^{-1}  \binom{2k_2}{k_2+b}^{-1} \sum^{2k_2}_{\alpha=0} (-1)^\alpha\binom{2k_2}{\alpha}T^\al_b\\
=& (-1)^{k_2}2^{-1}\binom{2k_1+2}{k_1+1}\cdot\sum_{b=-k_2}^{k_2} (-1)^{k_2+b} \binom{2k_1}{k_1+b}^{-1}  
\end{aligned}\eeq
A simple induction argument shows that  for any $0\leq k_2\leq k_1$, we have
\beq\label{E:8.I}\sum^{k_2}_{b=-k_2} (-1)^{k_2+b} \binom{2k_1}{k_1+b}^{-1}= \frac{2k_1+1}{k_1+1} \cdot \binom{2k_1+1}{k_1-k_2}^{-1}.
\eeq
It is clear that \eqref{E:7.I} and \eqref{E:8.I} yield the lemma.
%We proceed to prove by induction on $0\leq k_2\leq k_1$ that
%\begin{align}
 % G(k_2) = \frac{2k_1+1}{k_1+1} \cdot \binom{2k_1+1}{k_1-k_2}^{-1}.  \label{E:arch2}
%\end{align}
%If $k_2=0$, (\ref{E:arch2}) is clear. 
%If $k_2>0$, by induction hypothesis \begin{align*}     G(k_2+1) 
  %=&  \frac{2(k_1-k_2-1)! (k_1+k_2+1)! }{(2k_2)!}   - G(k_2)    \\  \displaybreak[0]
  %=&  \frac{2(k_1-k_2-1)! (k_1+k_2+1)! }{(2k_2)!}   - \frac{2k_1+1}{k_1+1} \cdot \frac{(k_1-k_2)!(k_1+k_2+1)!}{(2k_1+1)!}    \\  %\displaybreak[0]
 % =&  \binom{2k_1+1}{k_1-k_2-1}^{-1} \cdot (2k_1+1) \cdot \left( \frac{2}{k_1+k_2+2}  -   \frac{k_1-k_2}{(k_1+1)(k_1+k_2+2)}  \right)  \\   \displaybreak[0]
  %=&  \binom{2k_1+1}{k_1-k_2-1}^{-1} \cdot \frac{2k_1+1}{k_1+1}. 
%\end{align*}
\end{proof}

\begin{prop}\label{locinfty}
We have
\begin{align*}
     \cI(\test_\infty)=&  \int_{\bfX_\infty} \pair{\test_\infty(x)}{\test_\infty(x)}_{\cW\ot\cL}   \rmd x\\
     =&  \frac{(-1)^{k_2} (2k_1+1)}{2^{2k_1+7}  } \cdot \frac{ \Gamma_\C(k_1+k_2+2) \cdot  \Gamma_\C(k_1-k_2+1)}{\Gamma_\R(2)\Gamma_\R(4)}
%= & e^{-8\pi} \cdot (-1)^{k_2}\cdot  4^{-k_1-3} (2k_1+1)\cdot \frac{L(1, {\rm As}^+(\pi_\infty))}{\Gamma_\R(2)\Gamma_\R(4)}%\\ =   &(-1)^{k_2} \frac{1}{2^3} e^{-8\pi} \times (2k_1+1) \times L(1, {\rm As}^+(\pi_\infty)).
\end{align*}
\end{prop}
\begin{proof}
For $x\in (D_\infty^\x)^2$, we write $x=r\cdot u$ with $r=(r_1,r_2)\in(\R_+)^2$ and $u=(u_1,u_2)\in \SU_2(\R)^2$. Then the Haar measue $\rmd x$ is given by $(r_1r_2)^3\rmd r\rmd u$, where $\rmd r=\rmd r_1\rmd r_2$ is the Lebesque measure on $\R_+\x\R_+$ and $\rmd u=\rmd u_1\rmd u_2$ is the Haar measure on $\SU_2(\R)^2$ with $\vol(\SU_2(\R)^2)=4\pi^4$. We have \begin{align*}
      \int_{\bfX_\infty}  \langle \test_\infty(x), \test_\infty(x)  \rangle_{\cW\ot\cL} \rmd x 
= &  \int_{D^{\times 2}_\infty} \langle \test_\infty(x), \test_\infty(x)  \rangle_{\cW\ot\cL}  \rmd x\\ 
= &  \int_0^\infty\int_0^\infty\int_{\SU_2(\R)^2}  \pair{\test_\infty(r\cdot u)}{\test_\infty(r\cdot u )}(r_1r_2)^3 \rmd r\rmd u.
\end{align*}
Note that \begin{align*}
\pair{ \test_\infty(r\cdot u)}{ \test_\infty(r\cdot u)}_{ \cW\ot\cL }
=& \pair{\rho_{(2k_2,0)}(\pDII{r_1}{r_2})\test_\infty(u)}{\rho_{(2k_2,0)}(\pDII{r_1}{r_2})\test_\infty(u)}_{\cW\ot\cL}\\
=& (r_1r_2)^{2k_1} e^{-4\pi (r_1^2+r_2^2)} \pair{P_\ulk(u)}{P_\ulk(u)}_{\cW\ot\cL}.
\end{align*}
Hence by \lmref{L:1.loc}, we see that\begin{align*}
&\int_{\bfX_\infty}  \langle \test_\infty(x), \test_\infty(x)  \rangle_{\cW\ot\cL} \rmd x \\
=&  \left(\int^\infty_{0} r_1^{2k_1+3} e^{-4\pi r_1^2}\rmd r_1\right)^2\cdot (4\pi^4)\cdot \int_{{\rm SU}_2(\R)^2} \pair{P_\ulk(u)}{P_\ulk(u)}_{\cW\ot\cL} \rmd^* u\\ 
=&\left( 1/2\cdot(4\pi)^{-k_1-2}\Gamma(k_1+2)\right)^2\cdot(4\pi^4)\cdot  (-1)^{k_2} \cdot   (2k_1+1) \cdot\frac{\Gamma(k_1+k_2+2)\Gamma(k_1-k_2+1)}{\Gamma(k_1+2)^2}\\
= & (-1)^{k_2} \frac{\pi^3 }{2^{2k_1+7} }\cdot (2k_1+1) \cdot 2(2\pi)^{-k_1-k_2-2} \Gamma(k_1+k_2+2) \cdot 2(2\pi)^{-k_1+k_2-1} \Gamma(k_1-k_2+1)  \\  \displaybreak[0]
= & (-1)^{k_2} \frac{1}{2^{2k_1+7}  } \cdot (2k_1+1) \cdot \frac{ \Gamma_\C(k_1+k_2+2) \cdot  \Gamma_\C(k_1-k_2+1)}{\Gamma_\R(2)\Gamma_\R(4)}.   
\end{align*}
This finishes the proof of the proposition.\end{proof}

\subsection{Local integrals at finite places: preliminary}
In the following two subsections 6.3 and 6.4, we let $\q$ be a rational prime and calculate the local zeta integral 
\[ \cZ_\q(\test_\q,f_\q^\dagger)=\int_{H^0_1(\Q_\q)} 
         {\mathcal B}_{\omega_\q}(\omega_\q(h_\q)\test_\q, \test_\q)
         {\mathcal B}_{\piHo_\q}( \piHo_\q(h_\q)f_\q^\dag, f_\q^\dag) \rmd h_\q.\]
 To simply the notation, we often omit the subscript $\q$. For example, we write $D_0$, $F$, $\om$, $\test$, $f^\dagger$, $h$ for $D_0\ot\Qp$, $F\ot\Qp$, $\om_\q$, $\test_\q$, $f^\dagger_\q$ and $h_\q$. Let $\cU_p=H^0_1(\Qp)\cap ((R\ot\Zp)^\x\times \Zp^\x)/\cO_{F_p}^\x)$ be the local component of the open-compact subgroup $\cU$ defined in \eqref{E:opcpt}. One verifies that $\test$ and $f^\dagger$ are $\cU_p$-invariant, and hence we have \beq\label{E:1.loc}\cZ_p(\test,f^\dag)=\sum_{h}\cB_\om(\om(h)\test,\test)\cB_\piHo(\piHo(h)f^\dag,f^\dag)\cdot\vol(\cU_p h\cU_p),
\eeq
where $h$ runs over a complete set of representatives of the double coset space $\cU_p\bksl H^0_1(\Qp)/\cU_p$.
\subsection{Local integrals at finite places: the split case}
In this subsection, we suppose that $\q=\frakp\frakp^c$ is split in $F$. We shall identity $H^0(\Q_\q)$ with $(D_0^\x\times D_0^\x)/\Q_\q^\x$ with respect to $\frakp$ as in \remref{R:split}.  
First we treat the case $\q\ndivides N^-$. Then $D_0={\rm M}_2(\Q_\q)$ and $H^0(\Q_\q)=(\GL_2(\Q_\q)\times\GL_2(\Q_\q))/\Q_\q^\x$. We have $\piHo=\pi_\frakp\boxtimes\pi_{\frakp^c}$, where $\pi_\frakp$ and $\pi_{\frakp^c}$ are admissible and irreducible representations of $\PGL_2(\Q_\q)$. Then $f^0=f_\frakp^0\ot f_{\frakp^c}^0$, where $f_\frakp^0$ and $f_{\frakp^c}^0$ are new vectors of $\pi_\frakp$ and $\pi_{\frakp^c}$. For $\pi=\pi_\frakp$ or $\pi_{\frakp^c}$ and $f_?^0=f^0_{\frakp}$ or $f^0_{\frakp^c}$, let $\cB_{\pi}:\pi\ot\bar\pi\to\C$ be the $\GL_2(\Qp)$-invariant pairing such that $\cB_{\pi}(f_?^0,f_?^0)=1$. We thus have
\[\cB_{\piHo}(a_1\ot a_2,b_1\ot b_2)=\cB_{\pi_\frakp}(a_1,b_1)\cB_{\pi_{\frakp^c}}(a_2,b_2).\]
In the case $p\divides N^+$,  we shall assume $\frakp^c\divides\frakN^+$. Thus $\pi_{\frakp^c}\iso{\rm St}\ot(\chi_2\circ\det)$ is a special representation associated with a unramified quadratic character $\chi_2:\Q_\q^\x\to\stt{\pm 1}$ and let $\e\in\stt{\pm 1}$ be the sign given by
\[\e:=\begin{cases}1&\text{ if }\pi_\frakp\text{ is spherical},\\
\e_\frakp\e_{\frakp^c}&\text{ if }\pi_\frakp\text{ is special}.\end{cases}\]
Here we recall that $\e_\frakp$ and $\e_{\frakp^c}$ are the Atkin-Lehner eigenvalues of $\newform$ at $\frakp$ and $\frakp^c$ in \eqref{E:Atkineigenvalue}. Note that $\chi_2(p)=-\e_{\frakp^c}$ (\cf\cite[Proposition 3.1.2]{schmidt02}).
\begin{prop}\label{P:splita.loc}Suppose that $p\ndivides N^-$ is split in $F$. If $p\ndivides N^+$, then 
\begin{align*}\cZ_p(\test,f^\dag)=&\vol(\cU_p)\cdot\frac{L(1,{\rm As}^+(\pi))}{\zeta_p(2)\zeta_p(4)},\\
\intertext{ and if $p\divides N^+$}
\cZ_p(\test,f^\dag)=&\vol(\cU_p)\cdot p^{-2}(1+\e)\cdot \frac{ L(1, {\rm As}^+(\pi_p) )}{ \zeta_\q(1)\zeta_\q(2)}\cdot \cB_{\piHo}(f^\dag,f^\dag).
\end{align*}
\end{prop}
\begin{proof} For $n,a\in\Z$, put
\[h_{n,a}=(\pDII{p^{n+a}}{1},\pDII{p^n}{p^a})\in H^0_1(\Qp).\]
Let \[t=p^{-1}.\] For $m\in\Z$, put
\[\bfc_1(m)=\cB_{\pi_\frakp}(\pi_{\frakp}(\pDII{p^m}{1}f^0_\frakp,f^0_\frakp);\quad 
\bfc_2(m)=\cB_{\pi_{\frakp^c}}(\pi_{\frakp^c}(\pDII{p^m}{1}f^0_{\frakp^c},f^0_{\frakp^c}).\]
It is well-known that  for any $\epsilon>0$, there exists a constant $C_\epsilon$ such that \beq\label{E:Rbound1}\abs{\bfc_i(m)}\leq C_\epsilon\cdot t^{\abs{m}(1/2-\epsilon)}\text{ for }i=1,2\eeq
by the Ramanujan conjecture.

\ul{Case (i) $p\ndivides N_F$:} In this case, $\pi_\frakp$ and $\pi_{\frakp^c}$ are both spherical. 
Let $R_0={\rm M}_2(\Zp)$. Then $\cU_p=(R_0^\x\times R_0^\x)/\Zp^\x$, $\test$ is the characteristic function of $R_0\oplus R_0$ and $f^\dag=f^0=f_\frakp^0\ot f_{\frakp^c}^0$ is the fixed new vector. By Cartan decomposition, the set 
\[\stt{h_{n,a}\mid n+a\geq 0,\,n-a\geq 0}\]
is a complete set of representatives of $\cU_p\bksl H_1^0(\Qp)/\cU_p$. One can verify that \begin{align*}\cB_{\om}(\om(h_{n,a})\test,\test)=&t^{2(\abs{n}+\abs{a})},\\
\cB_\piHo(\piHo(h_{n,a})f^0,f^0)=&\bfc_1(n+a)\bfc_2(n-a),\\ 
\#(R_0^\x\pDII{p^m}{1} R_0^\x/R_0^\x)=&\begin{cases}1&\text{ if }m=0,\\
t^{-\abs{m}}(1+t)&\text{ if }m\not =0.\end{cases}\end{align*}
From \eqref{E:1.loc} together with the above equations, we see that $\vol(\cU_p)^{-1}\cZ^0(\test,f^\dag)$ equals
\begin{align*}
&\sum_{n+a\geq 0,n-a\geq 0}t^{2(\abs{n}+\abs{a})}\cB_{\piHo}(\piHo(h_{n,a})f^0),f^0)\#(\cU_ph_{n,a}\cU_p/\cU_p)\\
=&1+\sum_{n\geq 1}t^{4n}(\bfc_1(2n)+\bfc_2(2n))t^{-2n}(1+t)\\
+&\sum_{n\geq a+1,a\geq 0}t^{2n+2a}\bfc_1(n+a)\bfc_2(n-a)t^{-2n}(1+t)^2
+\sum_{n\geq a+1,a\geq 1}t^{2n+2a}\bfc_1(n-a)\bfc_2(n+a)t^{-2n}(1+t)^2.
\end{align*}
Here the above series converges absolutely by \eqref{E:Rbound1}.
Suppose that $\pi_\frakp=\pi(\mu_1,\mu_1^{-1})$ and $\pi_{\frakp^c}=\pi(\mu_2,\mu_2^{-1})$. By Macdonald's formula (\cf\cite[Theorem 4.6.6]{bump97Grey}), for $i=1,2$ letting $\al_i=\mu_i(p)$ and $\beta_i=\mu_i(p)^{-1}=\al_i^{-1}$, we have
\beq\label{E:MC.loc}\bfc_i(m)=\frac{t^{\frac{\abs{m}}{2}}}{1+t}(\al_i^{\abs{m}}A_i-\beta_i^{\abs{m}}B_i),\eeq
where \[A_i=\frac{\al_i-\beta_i t}{\al_i-\beta_i};\quad B_1=\frac{\beta_i-\al_i t}{\al_i-\beta_i}.\]
Note that $\abs{\al_i}=\abs{\beta_i}=1$. Therefore, we obtain \begin{align*}
&\vol(\cU_p)^{-1}\cZ_p(\test,f^\dag)\\
=1+&\sum_{n\geq 1}t^{3n}(\al_1^{2n}A_1-\beta_1^{2n}B_1+\al_2^{2n}A_2-\beta_2^{2n}B_2)+\\
&+\sum_{n\geq 1,a\geq 0} t^{n+3a}(\al_1^{n+2a}\al_2^nA_1A_2-\al_2^n\beta_1^{n+2a}A_2B_1-\al_1^{n+2a}\beta_2^nA_1B_2+\beta_1^{n+2a}\beta_2^nB_1B_2)\\
&+\sum_{n\geq 1,a\geq 1} t^{n+3a}(\al_1^n\al_2^{n+2a}A_1A_2-\al_2^{n+2a}\beta_1^{n}A_2B_1-\al_1^{n}\beta_2^{n+2a}A_1B_2+\beta_1^{n}\beta_2^{n+2a}B_1B_2)\\
=1+&\frac{A_1\al_1^2t^3}{1-t^3\al_1^2}-\frac{B_1\beta_1^2t^3}{1-t^3\beta_1^2}+\frac{A_2\al_2^2t^3}{1-t^3\al_2^2}-\frac{B_2\beta_2^2t^3}{1-t^3\beta_2^2}\\
+&\frac{A_1A_2\al_1\al_2 t}{(1-\al_1^2 t^3)(1-\al_2\al_1t)}-\frac{A_2B_1\beta_1\al_2 t}{(1-\beta_1^2 t^3)(1-\beta_1\al_2t)}-\frac{A_1B_2\al_1\beta_2 t}{(1-\al_1^2 t^3)(1-\al_1\beta_2t)}+\frac{B_1B_2\beta_1\beta_2 t}{(1-\beta_1^2 t^3)(1-\beta_1\beta_2t)}\\
+&\frac{A_1A_2\al_1\al_2^3 t^4}{(1-\al_2^2 t^3)(1-\al_1\al_2t)}-\frac{A_2B_1\beta_1\al_2^3 t^4}{(1-\al_2^2 t^3)(1-\beta_1\al_2t)}-\frac{A_1B_2\al_1\beta_2^3 t^4}{(1-\beta_2^2 t^3)(1-\al_1\beta_2t)}+\frac{B_1B_2\beta_1\beta_2^3 t^4}{(1-\beta_2^2 t^3)(1-\beta_1\beta_2t)}.
\end{align*}
We use \emph{Mathematica} to factor the above rational expression and find that
\[\vol(\cU_p)^{-1}\cdot\cZ_p(\test,f^\dag)=\frac{(1-t^4)(1-t^2)}{(1-\al_1\al_2t)(1-\beta_1\al_2t)(1-\al_1\beta_2 t)(1-\beta_1\beta_2 t)}=\frac{L(1,\pi_\frakp\ot\pi_{\frakp^c})}{\zeta_p(2)\zeta_p(4)}.\]
This completes the proof of the case (i).

\ul{Case (ii) $p\divides N^+$:} In this case, $\pi_{\frakp^c}={\rm St}\ot(\chi_2\circ\det)$ is special. Let $R_0$ be the standard Eichler order of level $p$ in ${\rm M}_2(\Zp)$ given by 
\[R_0=\stt{g\in{\rm M}_2(\Zp)\mid g\con\pMX{*}{*}{0}{*}\pmod{p\Zp}}.\]
Then $\cU_p=H^0_1(\Qp)\cap (R_0^\x\x R_0^\x)/\Zp^\x$, $\test$ is the characteristic function of $R_0\oplus R_0$ and $f^\dag=f_\frakp^\dag\ot f_{\frakp^c}^0$, where 
\beq\label{E:deffdag}f_\frakp^\dagger=\begin{cases}f_\frakp^0-\chi_2(p)\cdot \pi_\frakp(\pMX{0}{1}{p}{0})f_\frakp^0&\text{ if }\frakp\ndivides \frakN^+\iff \q\in\cP,\\
f_\frakp^0&\text{ if }\frakp\divides\frakN^+\iff\q\not\in\cP.
\end{cases}\eeq
Here $\cP$ is the set defined in \eqref{E:badprime}. Let $w\in\GL_2(\Qp)$, $w_1, w_2$ in $H_1^0(\Qp)$ be given by  
\begin{align*}
  w = \begin{pmatrix} & -1 \\ 1 &  \end{pmatrix}, \quad 
  w_1 = (w, \bfone_2);  \quad
  w_2 = (\bfone_2, w).  
\end{align*}
Then one can verify directly that the set
\begin{align*}
\Xi:=\stt{\,h_{n,a},\,w_1 h_{n,a},\,w_2 h_{n,a},\,w_1w_2h_{n,a}\,}_{n,a\in\Z} 
\end{align*}
is a complete set of representatives of $\cU_p\bksl H_1^0(\Qp)/\cU_p$. For integers $a,b,c,d$, put
\[\bfS_{a,b,c,d}:=\stt{\pMX{x}{y}{z}{w}\in{\rm M}_2(\Zp)\mid x\in p^a\Zp,y\in p^b\Zp, z\in p\Zp\cap p^c\Zp,w\in p^d\Zp}.\]
A direct computation shows that

\beq\label{E:2.loc}\begin{aligned}
{\mathcal B}_\om( \om(h_{n,a})\test,\test)=&\vol(\bfS_{a,n,1-n,-a})^2= t^{2|n|+ 2|a| + 2 },\\
{\mathcal B}_\om(\om( w_1 h_{n,a)}\test,\test) =& \vol(\bfS_{1-n,-a,a,n})^2=t^{2|n-\frac{1}{2}| +2|a-\frac{1}{2}| +2 },\\
{\mathcal B}_\om( \om(w_2 h_{n,a} )\test,\test) =&\vol(\bfS_{n,a,-a,1-n})^2 = t^{2|n-\frac{1}{2}| +2|a+\frac{1}{2}| +2 },\\
{\mathcal B}_\om( \om(w_1w_2 h_{n,a})\test,\test) =& \vol(\bfS_{-a,1-n,n,a})^2=t^{2|n-1| +2|a| + 2 }.
\end{aligned}\eeq
%\end{proof}
Next we consider $\cB_{\piHo}(\piHo(h_{n,a})f^\dag,f^\dag)$. For $\pi=\pi_\frakp$ or $\pi_{\frakp^c}$ and any $f\in\pi$, we put 
\begin{align*}%C_{\pi}^0(m)=&\cB_{\pi}(\pi_\frakp(\pDII{p^m}{1})f^0,f^0)\text{ for }m\in\Z;\\
\Phi_{f}(g)=&\cB_{\pi}(\pi(g)f,f)\cdot\#(\cU_p g \cU_p/\cU_p).\end{align*}
For $h=(g_1,g_2)\in H^0_1(\Qp)$, we have
\[\cB_\piHo(\piHo(h)f^\dag,f^\dag)\cdot \vol(\cU_p h \cU_p)=\vol(\cU_p)\cdot\Phi_{f_\frakp^\dagger}(g_1)\cdot\Phi_{f_{\frakp^c}^0}(g_2).\]
If $\pi={\rm St}\ot(\chi\circ\det)$ is special, it is well known that \begin{align}\label{matcoeffSpSt}
    \Phi_{f^0}( \begin{pmatrix} p^m & \\  &  1  \end{pmatrix})  = %C_\pi^0(m) t^{-m}=
    \chi(p)^m;\quad 
     \Phi_{f^0}( w\begin{pmatrix} p^m & \\  &  1  \end{pmatrix}) = %(-\chi(p))C_\pi^0(m-1)t^{-m}=
     -\chi(p)^m\quad (w=\pMX{0}{1}{-1}{0}). 
\end{align}
By the definition of $f_\frakp^\dag$ in \eqref{E:deffdag}, it is straightforward to verify that \beq\label{E:3.loc}\begin{aligned}
%\Phi_{f_\frakp^\dagger}(\pDII{p^{ m }}{1})
 %= &t^{-|m|}\cdot \left\{  2C(m) - \alpha_2C(m+1)-\al_2 C(m-1)  \right\}\\
  \Phi_{f_\frakp^\dagger}(\pDII{p^m}{1})=& \Phi_{f_\frakp^\dagger}(\pDII{p^{\abs{m}}}{1}),\\
 \Phi_{f_\frakp^\dagger}(w\pDII{p^m}{1})=&\e\cdot (-\chi_2(p))\cdot\Phi_{f_\frakp^\dagger}(\pDII{p^{m-1}}{1}) .
\end{aligned}\eeq

Let $\al_2:=\chi_2(p)$. By \eqref{E:1.loc} and \eqref{E:2.loc}, we have
\begin{align*}
&\vol(\cU_p)^{-1}\cdot\cZ_v^0(\test,f^\dagger)\\
=  & \sum_{\ep_1,\ep_2\in \stt{0,1} } 
     \sum_{n,a\in \Z} \cB_\om(\om(w^{\ep_1}, w^{\ep_2})h_{n,a})\test, \test) 
                           \Phi_{f_1^\dagger}(w^{\ep_1} \begin{pmatrix} p^{n + a} & \\ & 1 \end{pmatrix})  
                           \Phi_{f^0_2}(w^{\ep_2} \begin{pmatrix} p^{n} & \\ & p^{a} \end{pmatrix})    \\
=  &\sum_{\ep_1\in\stt{0,1}}\sum_{  n,a\in \Z} 
    ( \cB_\om(\om((w^{\ep_1}, \bfone_2)h_{n,a})\test, \test) - \cB_\om((\om(w^{\ep_1}, w)h_{n,a})\varphi, \varphi) )
                           \Phi_{f_1^\dagger}(w^{\ep_1}\pDII{p^{n+a}}{1})  
                           \alpha^{n-a}_2   \\
= &  \sum_{  n,a\in \Z} 
    ( t^{2|n|+2|a|+2} - t^{2|n-\frac{1}{2}| +2|a+\frac{1}{2}|  +2} )
                           \Phi_{f_1^\dagger}( \pDII{p^{n + a}}{1})  
                           \alpha^{n-a}_2   \\
   &  +  \sum_{  n,a\in \Z} 
    ( t^{2|n-\frac{1}{2}|+2|a-\frac{1}{2}|+2} - t^{2|n-1| +2|a|  +2} )
                          \Phi_{f_1^\dagger}(w \pDII{p^{n + a}}{1})  
                           \alpha^{n-a}_2. 
\end{align*}
Note that the absolute convergence of the above series follows from the Ramanujan conjecture. By \eqref{E:3.loc}, the first summation of the above equation is given by  
\begin{align*}
 &  \sum_{  n , a\in \Z} 
    ( t^{2|n|+2|a|+2} - t^{2|n-\frac{1}{2}| +2|a+\frac{1}{2}|  +2} )
                           \Phi_{f_1^\dagger}(\pDII{p^{n + a}}{1})  
                           \alpha^{n-a}_2   \\  \displaybreak[0]
= &  \sum_{  n,a\geq 0} 
    ( t^{2n+2a+2} - t^{2n+2a  +4} )
                          \Phi_{f_1^\dagger}( \pDII{p^{-n + a}}{1})  
                           \alpha^{-n-a}_2   \\   \displaybreak[0]
 & +  \sum_{  n,a\geq 1} 
        ( t^{2n+2a+2} - t^{2n+2a} )
                           \Phi_{f_1^\dagger}( \pDII{p^{n - a}}{1})  
                           \alpha^{n+a}_2   \\   \displaybreak[0]
= & t^2(1-t^2) \sum_{  0\leq n, a} 
                     \alpha^{n+a}_2 t^{2n+2a }   \Phi_{f_1^\dagger}( \pDII{p^{-n + a}}{1})   \\   \displaybreak[0]
 & + t^4(t^2-1) \sum_{  0\leq n , a} 
                          \alpha^{n+a}_2 t^{2n+2a}   \Phi_{f_1^\dagger}( \begin{pmatrix} p^{n - a} & \\ & 1 \end{pmatrix})  \\   \displaybreak[0]
= & t^2(1-t^2)^2 \sum_{  0\leq n, a} 
                     \alpha^{n+a}_2 t^{2n+2a }    \Phi_{f_1^\dagger}( \begin{pmatrix} p^{n - a} & \\ & 1 \end{pmatrix}),
\end{align*}
and  the second summation is given by 
\begin{align*}
 &  \sum_{  n,a\in \Z} 
    ( t^{2|n-\frac{1}{2}|+2|a-\frac{1}{2}|+2} - t^{2|n-1| +2|a|  +2} )
                             \Phi_{f_1^\dagger}(w \begin{pmatrix} p^{n + a} & \\ & 1 \end{pmatrix})  
                           \alpha^{n-a}_2  \\   \displaybreak[0]
= &  \sum_{  n\geq 0 , a\geq 1} 
    ( t^{2n+2a+2} - t^{2n+2a  +4} )
                            \Phi_{f_1^\dagger}( w \begin{pmatrix} p^{-n + a} & \\ & 1 \end{pmatrix})  
                           \alpha^{-n-a}_2   \\   \displaybreak[0]
 & +  \sum_{  n\geq 1 , a\geq 0} 
        ( t^{2n+2a+2} - t^{2n+2a} )
                             \Phi_{f_1^\dagger}( w \begin{pmatrix} p^{n -a} & \\ & 1 \end{pmatrix})  
                           \alpha^{n+a}_2   \\   \displaybreak[0]
= & \alpha_2 t^4(1-t^2) \sum_{n, a\geq 0} 
                     \alpha^{n+a}_2 t^{2n+2a }    \Phi_{f_1^\dagger}( w\begin{pmatrix} p^{-n + a+1} & \\ & 1 \end{pmatrix})   \\    \displaybreak[0]
 & + \alpha_2 t^2(t^2-1) \sum_{n,a\geq 0} 
                          \alpha^{n+a}_2 t^{2n+2a}   \Phi_{f_1^\dagger}( w\begin{pmatrix} p^{n - a+1} & \\ & 1 \end{pmatrix})  \\   \displaybreak[0]
= &  \e\cdot t^2(1-t^2)^2 \sum_{  n, a\geq  0} 
                     \alpha^{n+a}_2 t^{2n+2a }   \Phi_{f_1^\dagger}( \begin{pmatrix} p^{ n - a} & \\ & 1 \end{pmatrix}).
\end{align*}
We thus conclude that
\beq\label{E:4.loc}
\vol(\cU_p)^{-1}\cdot\cZ_p(\test,f^\dagger)=(1+\e)\cdot t^2(1-t^2)^2 \sum_{  n, a\geq  0} 
                     (\alpha_2t^2)^{n+a }\cdot\Phi_{f_1^\dagger}( \pDII{p^{ n - a}}{1}).
\eeq
To evaluate the last infinite series in the above equation, we will use the following elementary identity:\begin{lm}\label{formalsum}
For indeterminates $T, X, Y$, we have 
\begin{align*}
%\sum_{0\leq n,a} X^{n+a}  T^{|n-a|}  
   % =& \frac{1+XT }{(1-X^2)(1-XT)},  \\
\sum_{n,a\geq 0} X^{n+a} Y^{|n-a|} T^{|n-a-1|}  
    =& \frac{XY+T}{(1-X^2)(1-XYT)}.
\end{align*}
\end{lm}
Suppose that $\pi_\frakp=\pi(\mu_1,\mu_1^{-1})$ is spherical with Satake parameters  $\al_1=\mu(p)$ and $\beta_1=\mu_1(p)^{-1}$. %then by Macdonald formula (\cf\cite[Theorem 4.6.6]{bump97Grey}),
%\[\bfc(m):=\cB_{\pi_\frakp}(\pi_\frakp(\pDII{p^m}{1}f_\frakp^0,f_\frakp^0)= \frac{t^{\frac{|m|}{2}}    }{1+t} \cdot (\alpha^{|m|}_1 A - \beta^{|m|}_1 B),\] 
%where 
 %\begin{align*}
 % A= \frac{\alpha_1-\beta_1 t}{\alpha_1-\beta_1}; \quad   B= \frac{\beta_1-\alpha_1 t}{\alpha_1-\beta_1}.
 % \end{align*}
We have $\#(\cU_p\pDII{p^m}{1}\cU_p/\cU_p)=t^{-\abs{m}}$ and  
  \begin{align*} \Phi_{f_\frakp^\dagger}( \pDII{p^m}{1})=& t^{-\abs{m}}\cdot\cB_{\pi_\frakp}(\pi_\frakp(\pDII{p^m}{1})f_\frakp^0-\al_2\pi_\frakp(\pDII{p^{m-1}}{1})f_\frakp^0,f_\frakp^0-\al_2\pi_\frakp(\pDII{1}{p})f_\frakp^0 )\\
=&  t^{-\abs{m}}\cdot(2\bfc_1(m)-\al_2\bfc_1(m+1)-\al_2\bfc_1(m-1)). \end{align*}
%We note that 
%\begin{align*}
   %  & {\mathcal B}(\phi^\dag_1\otimes \phi_2, \phi^\dag_1\otimes \phi_2)  \\
  %= & 2 {\mathcal B}(\phi_1\otimes \phi_2, \phi_1\otimes \phi_2) 
     %    - \alpha_2 {\mathcal B}( \pi_1( \begin{pmatrix} 1 & \\ & p \end{pmatrix})(\phi_1)\otimes \phi_2, \phi_1\otimes \phi_2) 
        % - \alpha_2 {\mathcal B}(\phi_1\otimes \phi_2, \pi_1( \begin{pmatrix} 1 & \\ & p \end{pmatrix})(\phi_1)\otimes \phi_2)  \\  
 % = & 2 -2\alpha_2 \times \frac{(\alpha_1+\beta_1)t^{\frac{1}{2}}}{1+t}  \\
  %= & 2 \times \frac{1+t - (\alpha_1+\beta_1)\alpha_2t^{\frac{1}{2}}}{1+t}.
%\end{align*}
By Macdonald's formula \eqref{E:MC.loc}, we thus find that the last inifnite series in \eqref{E:4.loc} equals\begin{align*}
 & \sum_{n,a\geq 0}  (\alpha_2 t^2)^{n+a }  
        \Phi_{f_\frakp^\dagger}( \pDII{p^{n-a}}{1})  \\
 = & \sum_{n, a\geq 0}   (\alpha_2 t^2)^{n+a } 
       \cdot t^{-|n-a|} 
       \cdot ( 2\bfc(n-a) - \alpha_2\bfc(n-a+1)-\al_2\bfc(n-a-1))\\
 = & 2 \sum_{ n, a\geq 0}  (\alpha_2 t^2)^{n+a } 
       \cdot t^{-|n-a|} 
      \cdot(\bfc(n-a)   - \alpha_2\bfc(n-a-1)   )  \\
 = & \frac{2}{1+t} \sum_{n, a\geq 0}  (\alpha_2 t^{2})^{n+a}  
        \left\{   (\alpha_1t^{-\frac{1}{2}})^{|n-a|} A_1 + (\beta_1t^{-\frac{1}{2}})^{|n-a|} B_1   - \alpha_2 t^{-|n-a|} (\alpha_1t^{\frac{1}{2}})^{|n-a-1|} A_1 + (\beta_1t^{\frac{1}{2}})^{|n-a-1|} B)   \right\}.
\end{align*} 
Applying \lmref{formalsum}, the above equation equals
\begin{align*}
 & \frac{2A_1}{(1+t)(1-t^4)(1-\alpha_1\alpha_2 t^{\frac{3}{2}} )}  
        \left\{  (1+\alpha_1 \alpha_2 t^{\frac{3}{2}}) -\alpha_2 (\alpha_2 t +\alpha_1t^{\frac{1}{2}})     \right\}   \\ \displaybreak[0]
 & + \frac{ 2B_1 }{(1+t)(1-t^4)(1-\beta_1\alpha_2 t^{\frac{3}{2}} )}  
        \left\{  (1+\beta_1 \alpha_2 t^{\frac{3}{2}}) -\alpha_2 (\alpha_2 t +\beta_1t^{\frac{1}{2}})     \right\} \\  \displaybreak[0]
=  & \frac{2A_1(1-t)(1-\alpha_1\alpha_2 t^{\frac{1}{2}})}{(1+t)(1-t^4)(1-\alpha_1\alpha_2 t^{\frac{3}{2}} )}  
    + \frac{ 2B_1(1-t)(1-\beta_1\alpha_2 t^{\frac{1}{2}}) }{(1+t)(1-t^4)(1-\beta_1\alpha_2 t^{\frac{3}{2}} )}    \\    \displaybreak[0]
=  & \frac{2(1-t)}{(1+t)(1-t^4)}  
      \left\{ \frac{A_1(1-\alpha_1\alpha_2 t^{\frac{1}{2}})}{1-\alpha_1\alpha_2 t^{\frac{3}{2}}} 
                 + \frac{B_1(1-\beta_1\alpha_2 t^{\frac{1}{2}})}{1-\beta_1\alpha_2 t^{\frac{3}{2}}}   \right\}  \\   \displaybreak[0]
=  & \frac{2(1-t)}{(1+t)(1-t^4)}
       \cdot  \frac{    (\alpha_1-\beta_1 t)(1-\alpha_1\alpha_2 t^{\frac{1}{2}})(1-\beta_1\alpha_2 t^{\frac{3}{2}})   
                           -   (\beta_1 -\alpha_1 t)(1-\beta_1\alpha_2 t^{\frac{1}{2}})(1-\alpha_1\alpha_2 t^{\frac{3}{2}})      
                          }{  (\alpha_1 -\beta_1) (1-\alpha_1\alpha_2 t^{\frac{3}{2}})  (1-\beta_1\alpha_2 t^{\frac{3}{2}}) }  \\   \displaybreak[0]
%=  & \frac{2(1-t)}{(1+t)(1-t^4)}  
   %   \cdot\frac{    (\alpha_1-\beta_1 )(1+t) -(\alpha^2_1-\beta^2_1)\alpha_2 t^{\frac{1}{2}} -(\alpha^2_1-\beta^2_1)\alpha_2 t^{\frac{5}{2}}  
      %                        +(\alpha_1-\beta_1) t^2 + (\alpha_1-\beta_1) t^3     
         %                 }{  (\alpha_1 -\beta_1) (1-\alpha_1\alpha_2 t^{\frac{3}{2}})  (1-\beta_1\alpha_2 t^{\frac{3}{2}}) }  \\   \displaybreak[0]  
=  & \frac{2(1-t)}{(1+t)(1-t^4)}  
       \cdot \frac{    1+t - (\alpha_1+\beta_1)\alpha_2 t^{\frac{1}{2}}(1+t^2) 
                              + t^2+t^3     
                          }{   (1-\alpha_1\alpha_2 t^{\frac{3}{2}})  (1-\beta_1\alpha_2 t^{\frac{3}{2}}) }  \\  \displaybreak[0]
%=  & \frac{2(1-t)}{(1+t)(1-t^4)}  
   %    \cdot  \frac{    (1+t)(1+t^2) - (\alpha_1+\beta_1)\alpha_2 t^{\frac{1}{2}}(1+t^2) 
      %                     }{   (1-\alpha_1\alpha_2 t^{\frac{3}{2}})  (1-\beta_1\alpha_2 t^{\frac{3}{2}}) }  \\   \displaybreak[0]
=  & \frac{ 2}{(1+t)^2}
       \cdot\frac{   (1-\al_1\al_2t^\onehalf)(1-\beta_1\al_2 t^\onehalf)
               }{   (1-\alpha_1\alpha_2 t^{\frac{3}{2}})  (1-\beta_1\alpha_2 t^{\frac{3}{2}}) }.  
\end{align*}
Therefore, \eqref{E:4.loc} yields 
\[\cZ_p(\test,f^\dag)=\vol(\cU_p)t^2\cdot\frac{ L(1,{\rm As}^+(\pi))}{\zeta_\q(1)\zeta_\q(2)}\cdot \frac{4(1-\al_1\al_2t^\onehalf)(1-\beta_1\al_2 t^\onehalf)}{1+t},\]
and we complete the proof in this case by noting that 
\[\cB_{\piHo}(f^\dag,f^\dag)=\cB_{\pi_\frakp}(f^\dag_\frakp,f^\dag_\frakp)=
2(1-\al_2\cdot \cB_{\pi_\frakp}(\pi_\frakp(\pDII{1}{p})f_\frakp^0,f_\frakp^0))=\frac{2(1-t^\onehalf\al_2(\al_1+\beta_1))}{1+t}.\]

Suppose that $\pi_\frakp={\rm St}\ot(\chi_1\circ\det)$ is special. Let $\al_1=\chi_1(p)$. Then $\e=\e_\frakp\e_{\frakp^c}=\al_1\al_2$. We have 
\begin{align*}
    \sum_{n,a\geq 0}  (\alpha_2 t^2)^{n+a }  
        \Phi_{f_1^\dagger}( \pDII{p^{n-a}}{1}))  = & \sum_{n,a\geq 0}(\alpha_2 t^2)^{n+a } \al_1^{\abs{n-a}} 
= \frac{1+\al_1\al_2t^2}{(1-t^4)(1-\al_1\al_2 t^2)}.
\end{align*}
Recall that \[L(s,{\rm As}^+(\pi))=L(1,\pi_\frakp\ot\pi_{\frakp^c})=(1-\al_1\al_2 t^s)^{-1}(1-\al_1\al_2 t^{s+1})^{-1}.\]
We find that 
\begin{align*}
\cZ_p(\test,f^\dag)&=\vol(\cU_p)\cdot (1+\e)\cdot t^2(1-t^2)^2\frac{(1+\e t^2)(1-\e t)}{1-t^4}\cdot L(1,{\rm As}^+(\pi))\\
=&\vol(\cU_p)t^2\cdot (1+\e)\cdot \frac{L(1,{\rm As}^+(\pi))}{\zeta_\q(1)\zeta_\q(2)}.
\end{align*}
This finishes the proof of the case (ii).
\end{proof}

\begin{prop}\label{P:splitb.loc}If $p\divides N^-$, then we have
\[      \cZ_p(\test,f^\dag)
  =   \vol(\cU_p)\cdot p^{-2}(  1+   \e_\frakp\e_{\frakp^c})\cdot\frac{L(1,{\rm As}^+(\pi))}{\zeta_\q(1)\zeta_\q(2)}. \]
\end{prop}
\begin{proof}
In this case, $D_{0,p}$ is the division quaternion algebra over $\Q_\q$ and $\cU_p=H^0_1(\Qp)\cap (R_0^\x\times R_0^\x)/\Zp^\x$ where $R_0=\stt{a\in D_0\mid {\rm n}(a)\in\Z_\q}$ is the maximal order of $D_0$. Note that $\test$ is the characteristic function on $R_0\oplus R_0$ and $H^0_1(\Q_\q)=\stt{(a,d)\in H^0(\Q_\q)\mid {\rm n}(a)={\rm n}(d)}$, it follows that for $h\in H^0_1(\Q_\q)$, $\om(h)\test=\test$. In addition, let $\uf_p\in D_0$ with ${\rm n}(\uf_\q)=\q$. Then $\stt{(1,1),(\uf_\q,\uf_\q)}$ is a complete set of representatives of the double coset space $\cU_p\bksl H^0_1(\Q_\q)/\cU_\q$. As $f^\dag=f^0$ is invariant by $\cU_\q$ and is also an eigenvector of $(\uf_\q,\uf_\q)$ with eigenvalue $\e_\frakp\e_{\frakp^c}$, we find that 
\[\cZ_p(\test,f^\dag)=\vol(\cU_\q)\vol(R_0)^2(1+\chi_1\chi_2(p))=\vol(\cU_p)p^{-2}\cdot (1+\e_\frakp\e_{\frakp^c}).\]
Now the proposition follows from the fact that
\[L(1,{\rm As}^+(\pi))=(1-\e_\frakp\e_{\frakp^c}p^{-1})^{-1}(1-\e_\frakp\e_{\frakp^c}p^{-2})^{-1}.\qedhere\]
\end{proof}

\subsection{Local integrals at finite places: the non-split case}\label{sec(In)}
Let $\q$ be a prime inert or ramified in $F$. Then $F$ is a quadratic extension over $\Qp$, $D={\rm M}_2(F)$, $H^0(\Qp)=(\GL_2(F)\times \Qp^\x)/F^\x$, and $\piHo=\pi\boxtimes\bfone$, where $\pi$ is an admissible and irreducible representation of $\PGL_2(F)$. Let $\cO=\cO_F$. For $m\in\Z$, we define 
\begin{align*}
 % u_m=&\begin{pmatrix} p^m & \\ & 1 \end{pmatrix}, \quad  
 u_m = (\pDII{p^m}{1}, p^{m}), \quad w= (\begin{pmatrix} 0 & -1 \\ 1 & 0 \end{pmatrix}, 1 )\in H^0_1(\Qp).
\end{align*}

\begin{prop}\label{P:inert.loc}
   Suppose that $p$ is inert in $F$. If $p\ndivides N^+$, then  
\begin{align*}
\cZ_p(\test,f^\dag)=&\vol(\cU_p)\cdot \frac{L(1,{\rm As}^+(\pi))}{\zeta_p(2)\zeta_p(4)},\\
 \intertext{ and if $p\divides N^+$, then }
         \cZ_p(\test,f^\dag)
             =& \vol(\cU_p)\cdot p^{-2}(1+\e_p) \cdot \frac{L(1,{\rm As}^+(\pi))}{\zeta_\q(1)\zeta_\q(2)}.\end{align*}
\end{prop}
\begin{proof}We let \[t=p^{-2}.\] 
\ul{Case (i) $p\ndivides N_F$:} Then $\pi=\pi(\mu,\mu^{-1})$ is an unramified principal series. Let $R={\rm M}_2(\cO)$. Then $\cU_p=H_1^0(\Qp)\cap (R^\x\times\Qp^\x)/F^\x$ and $f^\dag=f^0$. The set $\stt{u_m\mid m\in\Z_{\geq 0}}$ is a complete set of representatives of $\cU_p\bksl H_1^0(\Qp)/\cU_p$ by Cartan decomposition, and we have \[\cB_\om(\om(u_m)\test,\test)=t^{\abs{m}}\text{ for }m\in\Z.\] Let $\al=\mu(p)$ and $\beta=\mu(p)^{-1}$. Then $\abs{\al},\abs{\beta}\leq t^{1/5-1/2}$ by the Ramanujaun bound in \cite{lrs99}. By \eqref{E:1.loc} and Macdonald's formula, $\vol(\cU_p)^{-1}\cZ_p(\test,f^\dag)$ equals
\begin{align*}
&1+\sum_{m\geq 1}\cB_\om(\om(u_m)\test,\test)\cB_\pi(\pi(u_m)f^0,f^0)\#(\cU_p u_m \cU_p/\cU_p)\\
=&1+\sum_{m\geq 1}t^m\cdot t^{m/2}(\al^m\frac{\al -\beta t}{\al-\beta}-\beta^m \frac{\beta-\al t}{\al-\beta})\cdot t^{-m}(1+t)\\
=&1+\frac{(\al+\beta)t^\onehalf-t-t^2}{(1-\al t^\onehalf)(1-\beta t^\onehalf)}\\
=&\frac{(1-t^2)(1-t)}{(1-\al t^\onehalf)(1-\beta t^\onehalf)(1-t)}=\frac{L(1,{\rm As}^+(\pi))}{\zeta_p(2)\zeta_p(4)}.
\end{align*}

\ul{Case (ii) $p\divides \frakN^+$:} In this case, $\pi={\rm St}\ot\chi\circ\det$ is the unramified special representation of $\GL_2(F)$ attached to a quadratic character $\chi:F^\x\to\stt{\pm 1}$. Let $R$ be the Eichler order of $p$ given by 
\[R=\stt{g\in {\rm M}_2(\cO)\mid g\con \pMX{*}{*}{0}{*}\pmod{p\cO}}.\]
Then $\cU_p=H^0_1(\Qp)\cap (R^\x\times\Zp^\x)/\cO^\x$, and the set\begin{align*}
\stt{u_m,wu_m}_{m\in\Z} 
\end{align*}
is a complete set of representatives of the double coset space $\cU_p\bksl H^0_1(\Qp)/\cU_p$. Put
\[\bfS'_{a,b}=\stt{\pMX{x}{\delta y}{\delta z}{x^c}\mid x\in\cO, y\in \Zp\cap p^a\Zp,\,z\in p\Zp\cap p^b\Zp}.\]
Then we have 
\beq\label{E:5.loc}\begin{aligned}
\cB_\om(\om(u_m)\test,\test)&=\vol(\bfS'_{-m,m+1})^2=t^{\abs{m}+1},\\
\cB_\om(\om(wu_m)\test,\test)&=\vol(\bfS'_{1-m,m})^2=t^{\abs{m-1}+1}.\end{aligned}
\eeq
For $h=(g,\alpha)\in H^0(\Q_p)$, we put
\beq\label{E:6.loc}
\Phi_{f^0}(h) =\cB_\piHo(\piHo(h)f^0,f^0)\cdot \#(\cU_ph \cU_p/\cU_p).
\eeq
For $m\in \Z$, we have 
\[\Phi_{f^0}(u_m)=\chi(p)^m; \quad  \Phi_{f^0}(wu_m) = -\chi(p)^m. 
\]
The above equations along with \eqref{E:1.loc} imply that
\begin{align*}
\vol(\cU_p)^{-1} \cZ_p(\test,f^0)= &
       \sum_{m\in \Z}  
            \cB_\om( \om(u_m) \test, \test) \Phi_{f^0}(u_m)+ \cB(\om(wu_m) \test, \test ) \Phi_{f^0}(wu_m).
\end{align*}
By \eqref{E:5.loc} and \eqref{E:6.loc}, we have
\begin{align*}
 \sum_{m\in \Z}      \cB_\om( \om(u_m) \test, \test) \Phi_{f^0}(u_m)
   = & \sum_{m\geq 0}  t^{m+1} \cdot\chi(p)^m 
        +\sum_{ m \geq  1}  t^{m+1}\cdot\chi(p)^{m}    \\
    =& \frac{t+\chi(p)t^2}{1-\chi(p)t} , \\
 \sum_{m\in \Z} \cB_\om(\om(wu_m) \test, \test ) \Phi_{f^0}(wu_m)=&  -\sum_{m\geq 1}    t^{m} \cdot\chi(p)^m 
        - \sum_{ m \geq 0}  t^{m+2} \cdot\chi(p)^{m}    \\
   = & - \frac{\chi(p)t+t^2 }{1-\chi(p)t}.
\end{align*}
We thus obtain
\begin{align*}
\vol(\cU_p)^{-1} \cZ_p(\test,f^0)
=& \frac{t(1-\chi(p))  (1-t) }{1-\chi(p)t}\\
=&t(1-\chi(p))\cdot\frac{1+p^{-1}\chi(p)}{1-p^{-1}}\cdot \frac{(1-t)(1-p^{-1})}{(1-\chi(p)t)(1+p^{-1}\chi(p))}\\
=&p^{-2}(1-\chi(p))\cdot\frac{L(1,{\rm As}^+(\pi))}{\zeta_p(2)\zeta_p(1)}.
\end{align*}
This completes the proof.
\end{proof}

\begin{prop}\label{P:ramified.loc}Suppose that $p$ is ramified in $F$. Then we have 
\begin{align*}
   \cZ_p(\test,f^\dag)
 = &\vol(\cU_p)\cdot \abs{2^{-4}\Delta_F^3}_p(1+p^{-1})  \frac{L(1, {\rm As}^+(\pi) )}{\zeta_p(1)\zeta_p(2)}.
\end{align*}   
\end{prop}
\begin{proof}
In this case, $(p,N^+N^-)=1$, and hence $\piHo=\pi\boxtimes\bfone$, where $\pi=\pi(\mu,\nu)$ is a spherical representation of $\PGL_2(F)$. Let $R={\rm M}_2(\cO)$. Then $\cU_p=H^0_1(\Qp)\cap (R^\x\times \Zp^\x)/\cO^\x$ and $\stt{u_m}_{m\in\Z_{\geq 0}}$ is a complete set of representatives of $\cU_P\bksl H^0_1(\Qp)/\cU_p$.
Write $\cO=\Zp\oplus\Zp\theta$ and put $\delta=\theta-\theta^c$. Note that $\delta^{-1}\cO$ is the different of $F/\Qp$.  Let $t=p^{-1}$.  Since $\test$ is the characteristic function of $L_p\oplus L_p$, where
\[L_p=\stt{\pMX{x}{\delta y}{\delta z}{x^c}\mid x\in\cO,y,z\in 2^{-1}\Zp},\]
we see that $\cB_\om(\om(u_m)\test,\test)=t^{2m}\cdot\abs{2^{-4}(\delta\ol{\delta})^3}_p$. On the other hand, we have
\begin{align*}\Phi_{f^0}(u_m)=&\cB_{\piHo}(\piHo(u_m)f^0,f^0)\cdot \#(\cU_p u_m \cU_p/\cU_p)\\
=&\cB_\pi(\pi(\pDII{p^m}{1})f^0,f^0)\cdot t^{-2m}(1+t).
\end{align*}
Let $\uf$ be a uniformizer of $\cO$. Let $\alpha=\mu(\varpi)$ and $\beta=\nu(\varpi)$. Then $\abs{\al},\abs{\beta}\leq t^{1/5-1/2}$ by \cite{lrs99}. We have $\Phi_{f^0}(u_0)=1$, and for $m>0$, by Macdonald's formula
\[\Phi_{f^0}(u_m)=\frac{t^{-m}}{\alpha-\beta}  ( \alpha^{2m}(\alpha-\beta t) -\beta^{2m}(\beta-\alpha t)).  \]
By \eqref{E:1.loc} and the above equations, we see that
\begin{align*}
\abs{2^4\Delta_F^{-3}}_p\vol(\cU_p)^{-1}\cZ_p(\test,f^\dag)= & \sum_{m\geq 0}  
          \abs{2^4\Delta_F^{-3}}_p \cB_\om(\om(u_m)\test,\test)\Phi_{f^0}(u_m). \\
= & 1+
     \sum_{m>0}  \frac{t^{m}}{\alpha-\beta}\cdot(\alpha^{2m+1}-\al^{2m-1}t-\beta^{2m+1}+\beta^{2m-1}t) )\\
= & 1+\frac{1}{\al-\beta}\cdot(\frac{\al^3t-\al t^2}{1-\al^2t}-\frac{\beta^3t-\beta t^2}{1-\beta^2t} )\\
= &  1 + \frac{t(\al^2+1+\beta^2-2t-t^2)}{(1-\al^2 t)(1-\beta t^2)} \\
= & \frac{1 + t- t^2 - t^3 }{ (1-\alpha^2 t)(1-\beta^2 t)}=\frac{(1 - t^2)(1 + t)(1-t)  }{ (1-\alpha^2 t)(1-t)(1-\beta^2t)  }  \\
= & (1+t) \cdot\frac{L(1, {\rm As}^+(\pi)) }{ \zeta_p(1)\zeta_p(2)  }.
\end{align*}
This proves the proposition.\end{proof}

\section*{Acknowledgment}
The authors would like to express their gratitude to Kazuki Morimoto for his careful reading and helpful comments. They also would like to thank the referee's many valuable comments which improves the presentation of this paper. This work was done while the second author was a postdocotral fellow in National Center for Theoretical Sciences in Taiwan. 
He is deeply grateful for their supports and hospitalities. 
During this work, the first author was partially supported by MOST grant 103-2115-M-002-012-MY5, and the second author was supported by JSPS Grant-in-Aid for Research Activity Start-up Grant Number 15H06634
and  Grant-in-Aid for Young Scientists (B) Grant Number 17K14174.

\bibliographystyle{amsalpha}
\bibliography{yoshidabib}

\providecommand{\bysame}{\leavevmode\hbox to3em{\hrulefill}\thinspace}
\providecommand{\MR}{\relax\ifhmode\unskip\space\fi MR }
\providecommand{\MRhref}[2]{%
  \href{http://www.ams.org/mathscinet-getitem?mr=#1}{#2}
}
\providecommand{\href}[2]{#2}
\begin{thebibliography}{BDSP12}

\bibitem[AK13]{ak13}
Mahesh Agarwal and Krzysztof Klosin, \emph{Yoshida lifts and the {B}loch-{K}ato
  conjecture for the convolution {$L$}-function}, J. Number Theory \textbf{133}
  (2013), no.~8, 2496--2537.

\bibitem[BDSP12]{bdsp12}
Siegfried B{\"o}cherer, Neil Dummigan, and Rainer Schulze-Pillot, \emph{Yoshida
  lifts and {S}elmer groups}, J. Math. Soc. Japan \textbf{64} (2012), no.~4,
  1353--1405.

\bibitem[BSP97]{bsp97}
Siegfried B{\"o}cherer and Rainer Schulze-Pillot, \emph{Siegel modular forms
  and theta series attached to quaternion algebras. {II}}, Nagoya Math. J.
  \textbf{147} (1997), 71--106, With erratum to: ``Siegel modular forms and
  theta series attached to quaternion algebras'' [Nagoya Math. J. {{\bf{1}}21}
  (1991), 35--96];.

\bibitem[Bum97]{bump97Grey}
Daniel Bump, \emph{Automorphic forms and representations}, Cambridge Studies in
  Advanced Mathematics, vol.~55, Cambridge University Press, Cambridge, 1997.

\bibitem[CH16]{ch16}
Masataka Chida and Ming-Lun Hsieh, \emph{Special values of anticyclotomic
  {$L$}-functions for modular forms}, J. Reine Angew. Math. (2016),
  DOI:10.1515/crelle-2015-0072.

\bibitem[GI11]{gi11}
Wee~Teck Gan and Atsushi Ichino, \emph{On endoscopy and the refined
  {G}ross-{P}rasad conjecture for {$(\rm SO_5,SO_4)$}}, J. Inst. Math. Jussieu
  \textbf{10} (2011), no.~2, 235--324.

\bibitem[GJ78]{gj78}
Stephen Gelbart and Herv{\'e} Jacquet, \emph{A relation between automorphic
  representations of {${\rm GL}(2)$} and {${\rm GL}(3)$}}, Ann. Sci. \'Ecole
  Norm. Sup. (4) \textbf{11} (1978), no.~4, 471--542.

\bibitem[GQT14]{gqt14}
Wee~Teck Gan, Yannan Qiu, and Shuichiro Takeda, \emph{The regularized
  {S}iegel-{W}eil formula (the second term identity) and the {R}allis inner
  product formula}, Invent. Math. \textbf{198} (2014), no.~3, 739--831.

\bibitem[GS15]{shahidi15}
Neven Grbac and Freydoon Shahidi, \emph{Endoscopic transfer for unitary groups
  and holomorphy of {A}sai {$L$}-functions}, Pacific J. Math. \textbf{276}
  (2015), no.~1, 185--211.

\bibitem[HN17]{hn15}
Ming-Lun Hsieh and Kenichi Namikawa, \emph{Bessel periods and the non-vanishing
  of {Y}oshida lifts modulo a prime}, Math. Z. 
  \textbf{285}, (2017), 851--878.

\bibitem[JLR12]{lbrooks12}
Jennifer Johnson-Leung and Brooks Roberts, \emph{Siegel modular forms of degree
  two attached to {H}ilbert modular forms}, J. Number Theory \textbf{132}
  (2012), no.~4, 543--564.

\bibitem[Kli59]{kl59}
Helmut Klingen, \emph{Bemerkung \"uber {K}ongruenzuntergruppen der
  {M}odulgruppe {$n$}-ten {G}rades}, Arch. Math. \textbf{10} (1959), 113--122.

\bibitem[Kri03]{kris03}
M.~Krishnamurthy, \emph{The {A}sai transfer to {$\rm GL_4$} via the
  {L}anglands-{S}hahidi method}, Int. Math. Res. Not. (2003), no.~41,
  2221--2254.

\bibitem[Kri12]{kris12}
\bysame, \emph{Determination of cusp forms on {${\rm GL}(2)$} by coefficients
  restricted to quadratic subfields (with an appendix by {D}ipendra {P}rasad
  and {D}inakar {R}amakrishnan)}, J. Number Theory \textbf{132} (2012), no.~6,
  1359--1384.

\bibitem[LRS99]{lrs99}
Wenzhi Luo, Ze\'ev Rudnick, and Peter Sarnak, \emph{On the generalized
  {R}amanujan conjecture for {${\rm GL}(n)$}}, Automorphic forms, automorphic
  representations, and arithmetic ({F}ort {W}orth, {TX}, 1996), Proc. Sympos.
  Pure Math., vol.~66, Amer. Math. Soc., Providence, RI, 1999, pp.~301--310.
  \MR{1703764}

\bibitem[PW11]{pw11}
Robert Pollack and Tom Weston, \emph{On anticyclotomic {$\mu$}-invariants of
  modular forms}, Compos. Math. \textbf{147} (2011), no.~5, 1353--1381.

\bibitem[Rob01]{ro01}
Brooks Roberts, \emph{Global {$L$}-packets for {${\rm GSp}(2)$} and theta
  lifts}, Doc. Math. \textbf{6} (2001), 247--314 (electronic).

\bibitem[Sah15]{saha15}
Abhishek Saha, \emph{On ratios of {P}etersson norms for {Y}oshida lifts}, Forum
  Math. \textbf{27} (2015), no.~4, 2361--2412.

\bibitem[Sch02]{schmidt02}
Ralf Schmidt, \emph{Some remarks on local newforms for gl(2)}, J. Ramanujan
  Math. Soc. (2002).

\bibitem[Sie43]{si43}
Carl~Ludwig Siegel, \emph{Symplectic geometry}, Amer. J. Math. \textbf{65}
  (1943), 1--86.

\bibitem[SS13]{saha13}
Abhishek Saha and Ralf Schmidt, \emph{Yoshida lifts and simultaneous
  non-vanishing of dihedral twists of modular {$L$}-functions}, J. Lond. Math.
  Soc. (2) \textbf{88} (2013), no.~1, 251--270.

\bibitem[Tak09]{takeda09}
Shuichiro Takeda, \emph{Some local-global non-vanishing results for theta lifts
  from orthogonal groups}, Trans. Amer. Math. Soc. \textbf{361} (2009), no.~10,
  5575--5599.

\bibitem[Tak11]{takeda11}
\bysame, \emph{Some local-global non-vanishing results of theta lifts for
  symplectic-orthogonal dual pairs}, J. Reine Angew. Math. \textbf{657} (2011),
  81--111.

\bibitem[Tat79]{tate77Corvallis}
John Tate, \emph{Number theoretic background}, Automorphic forms,
  representations and {$L$}-functions ({P}roc. {S}ympos. {P}ure {M}ath.,
  {O}regon {S}tate {U}niv., {C}orvallis, {O}re., 1977), {P}art 2, Proc. Sympos.
  Pure Math., XXXIII, Amer. Math. Soc., Providence, R.I., 1979, pp.~3--26.

\bibitem[Yos80]{yo80}
Hiroyuki Yoshida, \emph{Siegel's modular forms and the arithmetic of quadratic
  forms}, Invent. Math. \textbf{60} (1980), no.~3, 193--248.

\end{thebibliography}

\end{document}